%% file: mirror-flows.tex
\documentclass[11pt,english,reqno]{article}
\usepackage[T1]{fontenc}
\usepackage{lmodern}

\input{Auxiliaries/Packages}

\input{Auxiliaries/Macros}

\author{Rapha\"el Berthier\thanks{Sorbonne Universit\'e, Inria, Centre Inria de Sorbonne Universit\'e, Paris, France} \and Loucas Pillaud-Vivien\thanks{CERMICS, CNRS, ENPC, Institut Polytechnique de Paris, Marne-la-Vallée, France}}

\begin{document}

\title{Incremental Learning in Mirror Flows}

\maketitle

\begin{abstract}
	We study mirror flows generated by a convex quadratic loss and a general convex lower semicontinuous mirror potential. We show that, when initialized near the boundary of the domain of the mirror potential, their rescaled trajectories converge to a limiting mirror flow whose potential is the indicator function of the domain. In this limit, the primal variable minimizes the loss over a time-dependent hypothesis set: the subdifferential of the support function of the domain, evaluated at the dual variable. This characterization provides a general mechanism for incremental learning in mirror flows.
\end{abstract}

\medskip


\smallskip


\setcounter{tocdepth}{2}
\tableofcontents
\newpage

\section{Introduction}

\paragraph*{Context.} 

Neural networks trained with gradient descent often learn solutions of increasing complexity: the model first captures simple structure, then progressively incorporates finer details~\cite{arpit2017closer,kalimeris2019sgd,zhang2025saddle}. This incremental learning phenomenon, often visible as plateaus in the training loss separated by rapid transitions, has attracted significant theoretical attention.

The most detailed analyses of incremental learning have been carried out for diagonal linear networks, including precise descriptions of transition times and plateau levels~\cite{berthier2023incremental,pesme2023saddle}. This level of detail is possible because the training dynamics of these networks reduce to a mirror flow~\cite{woodworth2020kernel}. Mirror flows themselves feature incremental learning when initialized near the boundary of the domain of the mirror potential. This paper gives a rigorous description of this phenomenon for a broad class of mirror flows, thereby generalizing the previous analyses of diagonal linear networks.

\paragraph*{Contributions.} 

We study mirror flows associated with a convex quadratic loss and a general convex lower semicontinuous mirror potential. We prove that, when the mirror flow is initialized near the boundary of the domain of the mirror potential, the rescaled dynamics converge to a limiting mirror flow whose potential is the indicator function of the domain. We establish existence and uniqueness for these limiting dynamics, building on the subgradient-flow theory of~\cite{brezis1973ope}. We show that, in the limit, the primal variable minimizes the loss over a time-varying hypothesis set---the subdifferential of the support function of the domain, evaluated at the dual variable---which evolves throughout training. This evolution is the mechanism behind incremental learning. The result yields several insights.

First, it provides a unified and more rigorous treatment of incremental learning in diagonal linear networks as a special case of mirror flows on the non-negative orthant with entropic potential.

Second, it enables the study of new geometries. We analyze the von Neumann entropic potential on the positive semidefinite cone. When the mirror flow is initialized near the boundary point $0$, we obtain an incremental learning procedure through successive rank increases. Notably, the incremental learning regime on the positive semidefinite cone is qualitatively different from the non-negative orthant case: the primal variable evolves continuously within each fixed-rank phase, rather than jumping between stationary points. This shows that incremental learning can arise from an alternation of slow and rapid dynamics, beyond the saddle-to-saddle dynamics that have already been widely studied~\cite{jacot2021saddle,li2020towards}.

We compare this mirror flow on the positive semidefinite cone with matrix factorization, which also features incremental learning through successive rank increases and for which a precise description is notoriously difficult~\cite{li2020towards,gunasekar2017implicit,arora2019implicit,razin2020implicit}.

Overall, diagonal linear networks are a special case of both mirror flows---for the entropic potential on the non-negative orthant---and matrix factorization---in the diagonal case. This paper suggests that the mirror-flow viewpoint is the sharper framework for analyzing their incremental learning behavior. Consequently, one should be cautious when extrapolating results from diagonal linear networks to matrix factorization or to other neural architectures.

\paragraph{Technical contribution.}This work provides an existence and uniqueness result for mirror flows associated with a convex quadratic loss and a general convex lower semicontinuous mirror potential. The novelty of this result lies in low regularity requirement for the mirror potential; this allows, for instance, to consider the indicator function of a closed convex set. This case is essential to define our limiting dynamics. The article of \cite{attouch2004singular} provides a similar result for more general loss functions, but requires the mirror potential to be strongly convex. This latter restriction does not allow considering indicator function of sets.

In \cite{attouch2004singular}, the regularity induced by the strong convexity of the mirror potential yields a Lipschitz dual dynamics, so that a Picard--Lindelöf fixed-point argument applies. In our setting, such regularity is unavailable; instead, we exploit the structure induced by the quadratic loss, which allows us to recast the dual dynamics as a subgradient flow and to rely on the convex nonsmooth theory of \cite{brezis1973ope,attouch1977convergence}.
\paragraph*{Additional related work.}
That neural networks acquire structure progressively---simple patterns first, finer details later---was already observed in early studies of plateaus in multi-layer perceptrons~\cite{fukumizu2000local}, and has since been a recurring observation across modern architectures and datasets. Deep linear networks with small initialization were the first model in which a sequence of phase transitions could be exhibited, with successive singular modes learned in turn~\cite{saxe2014exact,gidel2019implicit,gissin2020implicit}; this picture was sharpened for matrix factorization, where the trajectory passes through low-rank stationary points of increasing rank~\cite{jacot2021saddle}. A largely independent strand has studied diagonal linear networks (DLNs), whose training dynamics reduce to a mirror flow with entropic potential~\cite{woodworth2020kernel}, and has been pushed further to relate the entire incremental trajectory to the lasso regularization path~\cite{berthier2025diagonal,li2023implicit}. Beyond linear parametrizations, incremental learning has been established for shallow ReLU networks with orthogonal inputs~\cite{boursier2022gradient}, for two-layer networks learning multi-index targets through staircase, leap-complexity, and Grassmannian dynamics~\cite{abbe2023sgd,bietti2023learning}, and for transformers learning in-context $n$-grams~\cite{varre2025learning}, with recent work proposing a unified saddle-to-saddle viewpoint across MLPs, ReLU networks, and transformers~\cite{zhang2025saddle}.
\paragraph*{Outline.}

In Section~\ref{sec:notations}, we recall a few notions of convex analysis and introduce the notation used throughout the paper. In Section~\ref{sec:incremental}, we introduce mirror flows with a convex quadratic loss and a possibly non-smooth mirror potential, and show that they perform incremental learning when initialized near the boundary of the domain of the mirror potential. More precisely, we state the main result in Theorem~\ref{thm:main} in Section~\ref{sec:limiting}, which characterizes the limiting incremental learning dynamics as a non-smooth mirror flow. In Section~\ref{sec:examples}, we illustrate our results on three examples: the non-negative orthant and its connection to diagonal linear networks, the positive semidefinite cone and its connection to matrix factorization, and the probability simplex. Section~\ref{sec:proofs} recalls useful results from~\cite{brezis1973ope,attouch1977convergence} on subgradient flows and their evolutionary convergence, and uses them to prove the results stated in the previous sections. The most technical proofs and the experimental details are deferred to the appendix.

\section{Convex-analysis reminders and notation}
\label{sec:notations}

\paragraph{Convex analysis.} Let $C \subset \R^d$ be a convex set. We denote by $\aff C$ the affine hull of $C$, the smallest affine set containing $C$. We denote by $\ri C$ the relative interior of $C$, the interior of $C$ relative to $\aff C$. 

Let $f : \R^d \to \R \cup \{+\infty\}$ be a convex function. We denote by $\dom f = \{ x \in \R^d \mid f(x) < +\infty \}$ the domain of $f$. We say that $f$ is proper if $\dom f \neq \emptyset$. We say that $f$ is lower semicontinuous if, for all $x \in \R^d$ and all sequences $(x_n)_{n \in \N}$ converging to $x$, $f(x) \leq \liminf_{n\to\infty} f(x_n)$. We denote by $\Gamma_0(\R^d)$ the set of convex, proper, lower semicontinuous functions from $\R^d$ to $\R \cup \{+\infty\}$. We now fix $f \in \Gamma_0(\R^d)$. 

We denote by $\partial f(x) = \{ z \in \R^d \mid \forall y \in \R^d, f(y) \geq f(x) + \langle z , y-x \rangle \}$ the subdifferential of $f$ at $x$. We denote by $\dom \partial f = \{x \in \R^d \mid \partial f(x) \neq \emptyset\}$ the domain of $\partial f$. The subdifferential generalizes the gradient: if $f$ is differentiable at $x$, then $\partial f(x) = \{ \nabla f(x) \}$~\cite[Prop.~17.31]{bauschke2017convex}.

The Fenchel conjugate of $f$ is the convex function $f^*: \R^d \to \R \cup \{+\infty\}$ defined by $f^*(w) = \sup_{x \in \R^d} \langle w , x \rangle - f(x)$. Fenchel conjugacy preserves $\Gamma_0(\R^d)$ and is an involution on this class: if $f \in \Gamma_0(\R^d)$, then $f^* \in \Gamma_0(\R^d)$ and $(f^*)^* = f$~\cite[Coro.~12.2.1]{rockafellar1970convex}. Moreover, for all $x,w \in \R^d$, $w \in \partial f(x)$ if and only if $x \in \partial f^*(w)$~\cite[Coro.~16.30]{bauschke2017convex}.

Let $C$ be a closed nonempty convex subset of $\R^d$. We denote by $\iota_C : \R^d \to \R \cup \{+\infty\}$ the indicator function of $C$, defined as $\iota_C(x) = 0$ if $x \in C$ and $\iota_C(x) = +\infty$ otherwise. We denote by $\sigma_C : \R^d \to \R \cup \{+\infty\}$ the support function of $C$, defined as $\sigma_C(w) = \sup_{x \in C} \langle w , x \rangle$. The indicator and support functions are proper, convex, and lower semicontinuous. Moreover, they are conjugate to each other: $\iota_C^* = \sigma_C$ and $\sigma_C^* = \iota_C$~\cite[Thm.~13.2]{rockafellar1970convex}. 

\paragraph{Other notation.} We denote by $\Argmin_{x \in C} f(x)$ the set of minimizers of $f$ on $C$. If this set is a singleton, we denote by $\argmin_{x \in C} f(x)$ its unique element.

If $M$ is a positive semidefinite matrix, we denote by $M^\dagger$ its Moore-Penrose pseudoinverse, by $M^{1/2}$ its positive semidefinite square root, by $\Span M$ the span of its columns, and by $\ker M$ its kernel.

For $p \in \{1,2\}$, an interval $I \subset \R$, and a measure $\mu$ on $I$, we denote by $L^p(I,\R^d,\mu)$ the space of functions $f : I \to \R^d$ such that $\int_I \Vert f(t) \Vert^p \diff \mu(t) < +\infty$. When $\mu$ has density $g(t)$ with respect to the Lebesgue measure, we write $\mu = g(t) \diff t$.

We write $x \geq 0$ (resp.~$x > 0$) to mean that all coordinates of $x$ are nonnegative (resp.~positive). If $x$ is a real number, we denote by $x_+ = \max(x,0)$ its positive part and by $x_- = \max(-x,0)$ its negative part.
	
	When $\varphi:\R\to\R$ and $x \in \R^d$, we denote by $\varphi(x) = (\varphi(x_1), \dots, \varphi(x_d))$ the componentwise application of $\varphi$ to $x$. This convention is used, for instance, for $\varphi = \log$, $\varphi = \exp$, $\varphi = (\cdot)_+$, and $\varphi = (\cdot)_-$.

	We denote by $\mathbf{1} = (1, \dots, 1)$ the vector of ones, whose dimension is implicit from the context.

	If $x, y \in \R^d$, we denote by $x \circ y$ the componentwise product of $x$ and $y$.

\section{Incremental learning in mirror flows}
\label{sec:incremental}

Let $\ell$ denote a convex quadratic function 
\begin{equation*}
	\ell(x) = \frac{1}{2} \langle x , M x \rangle - \langle q , x \rangle + c \, , \qquad x \in \R^d \, , \qquad d \geq 1 \, ,
\end{equation*}
where $M \in \R^{d \times d}$ is a positive semidefinite matrix, $q \in \Span M$ (so that $\ell$ is bounded below), and $c \in \R$.
Consider $h \in \Gamma_0(\R^d)$. We denote $C = \dom h$. We assume that $C$ is closed. The functions $\ell$ and $h$ will be referred to as the loss and the mirror potential, respectively.

\subsection{Mirror flows}

Let $w : \R_{\geq 0} \to \R^d$. If $h$ is a smooth function, we say that $w$ is a solution of the \emph{mirror flow} with mirror potential $h$ and loss $\ell$ if there exists $x : \R_{\geq 0} \to \R^d$ such that for all $t \geq 0$, $w$ is differentiable at $t$ and
\begin{align}
	\label{eq:mirror-flow}
	&\frac{\diff w}{\diff t}(t) = - \nabla \ell(x(t)) \, , &&w(t) = \nabla h(x(t)) \, .
\end{align}
See~\cite{nemirovskij1983problem} for the historical introduction. The mirror flow can be seen as a generalization of the gradient flow constrained to $C$, through a change of geometry induced by the mirror potential $h$.

\begin{example}[logistic differential equation]
	\label{ex:logistic-1}
	As a first illustration of the results of this work, we use the logistic differential equation as a running example: 
	\begin{equation*}
		\frac{\diff x}{\diff t}(t) = x(t) (1-x(t)) \, , 
	\end{equation*}
	where $x(t) \geq 0$. 
	
		This equation corresponds to a mirror flow. For $x \in \R$, define the loss $\ell(x) = \frac{1}{2} (x-1)^2$ and the entropic mirror potential $h(x) = x \log x - x$ for $x \geq 0$ (with the convention $0 \log 0 = 0$) and $h(x) = +\infty$ otherwise. The logistic differential equation can be rewritten as 
\begin{align*}
	&\frac{\diff w}{\diff t}(t) = 1 - x(t) \, , &&w(t) = \log x(t) \, ,
\end{align*}
which corresponds to Eq.~\eqref{eq:mirror-flow}: $w(t) = \log x(t)$ is a solution of the mirror flow with mirror potential $h$ and loss $\ell$.

		The logistic differential equation can be solved in closed form: its positive solutions are $x(t) = f(t-t_0)$ for some $t_0 \in \R$, where $f(t) = 1/(1+e^{-t})$ is the logistic function. Thus little is gained by applying our results to this equation, but it provides a simple running example. More substantial examples are given in Section~\ref{sec:examples}.
\end{example}

In this paper, we need to generalize the definition of a mirror flow. We no longer assume that $h$ is smooth.

\begin{definition}[mirror flow]
	\label{def:mirror-flow}
	Let $w : \R_{\geq 0} \to \R^d$. We say that $w$ is a solution of the \emph{mirror flow} with mirror potential $h$ and loss $\ell$ if there exists $x : \R_{\geq 0} \to \R^d$ such that for almost every (a.e.) $t \geq 0$, $w$ is differentiable at $t$ and
\begin{align*}
	\frac{\diff w}{\diff t}(t) = - \nabla \ell(x(t)) \, , \qquad w(t) \in \partial h(x(t)) \, .
\end{align*}
\end{definition}
We refer to $w$ as the \emph{dual variable} and $x$ as a \emph{primal variable} of the mirror flow. In \cite{attouch2004singular}, this mirror flow was called a \emph{singular Riemannian barrier system}.

\begin{example}[logistic differential equation, continued]
	\label{ex:logistic-2}
		Consider the following mirror flow related to the logistic differential equation. For $x \in \R$, take the same loss $\ell(x) = \frac{1}{2}(x-1)^2$ as in Example~\ref{ex:logistic-1}, but replace the entropic mirror potential by $h(x) = \iota_{\R_{\geq 0}}(x)$. Then
	\begin{equation*}
		\partial h(x) = \partial \iota_{\R_{\geq 0}}(x) = \begin{cases}
			\{0\} & \text{if } x > 0 \, , \\
			\R_{\leq 0} & \text{if } x = 0 \, , \\
			\emptyset & \text{if } x < 0 \, .
		\end{cases}
	\end{equation*}
	As a consequence, $w$ is a solution of the mirror flow with mirror potential $h$, loss $\ell$, and primal variable $x$ if and only if for a.e.~$t \geq 0$, $w$ is differentiable at $t$ and
	\begin{align*}
		&\frac{\diff w}{\diff t}(t) = 1 - x(t) \, ,  &&w(t) \leq 0 \, , &&x(t) \geq 0 \, , &&w(t)x(t) = 0 \, . 
	\end{align*}
	Here, the differential inclusion condition of the mirror flow takes the form of sign constraints on $w$ and $x$ with a complementarity slackness condition. A solution of this mirror flow initialized from $w(0) = -1$ is given by $w(t) = \min(-1+t,0)$, with associated primal variable 
	\begin{equation}
		\label{eq:logistic-primal}
	x(t) = \begin{cases}
		0 & \text{if } t < 1 \, , \\
		1 & \text{if } t \geq 1 \, .
	\end{cases}
	\end{equation}
		In Theorem~\ref{thm:existence-mirror-flow} below, we show that this solution $w(t)$ is the unique solution among Lipschitz functions and that the associated primal variable $x(t)$ is essentially unique.

	In this example, the primal solution $x(t)$ of the mirror flow is the limit of the time-rescaled sigmoid function $x^\mu(t) = f(\mu (t-1))$ as $\mu \to +\infty$ (pointwise, for all $t \neq 1$). This turns out to be a general phenomenon: mirror flows associated with non-smooth mirror potentials naturally arise as limits of mirror flows with smooth mirror potentials. This statement is made concrete in Section~\ref{sec:limiting} below and motivates the generalized definition of mirror flows proposed in Definition~\ref{def:mirror-flow}.
\end{example}

As the function $\ell$ is quadratic, a mirror flow can be seen as a subgradient flow in the dual variable. 

\begin{proposition}
	\label{prop:equivalences-mf-gf}
	Let $w: \R_{\geq 0} \to \R^d$. Assume that $w(0) \in \ri \dom h^* + \Span M$. The following statements are equivalent.
	\begin{enumerate}[label=(\alph*)]
		\item\label{it:mf} $w$ is a solution of the mirror flow with mirror potential $h$ and loss $\ell$.
		\item\label{it:prec-gf} For a.e.~$t \geq 0$, $w$ is differentiable at $t$ and
		\begin{align*}
			&\frac{\diff w}{\diff t}(t) + M\partial \Phi(w(t)) \ni 0 \, , &&\text{where } \Phi(w) = h^*(w) - \langle M^\dagger q , w \rangle \, .
		\end{align*}
		\item\label{it:gf} Define $u(t) = (M^\dagger)^{1/2} (w(t) - w(0))$. For a.e.~$t \geq 0$, $u$ is differentiable at $t$ and
		\begin{align*}
			&\frac{\diff u}{\diff t}(t) + \partial \varphi(u(t)) \ni 0 \, , &&\text{where } \varphi(u ) = h^*\left(w(0)+M^{1/2}u\right) - \left\langle (M^\dagger)^{1/2} q , u \right\rangle \, .
		\end{align*}
	\end{enumerate}
\end{proposition}

Statement~\ref{it:prec-gf} shows that the dual variable $w$ follows a differential inclusion. When $M$ is positive definite, this differential inclusion is a preconditioned subgradient flow. However, the connection to subgradient flows is not straightforward when $M$ is not positive definite. As we want to encompass this latter case, we prefer statement~\ref{it:gf} for technical reasons: $w$ can be written as $w(0) + M^{1/2} u$ where $u$ follows a subgradient flow (see Def.~\ref{def:sol-subgradient-flow} for a definition of the subgradient flow).

We note that the equivalence between statements~\ref{it:mf} and~\ref{it:prec-gf} corresponds to the equivalence between the singular Riemannian barrier system (SRB) and its dual (SRBaux) in \cite{attouch2004singular}. However, in our case, the additional equivalence with statement~\ref{it:gf} is crucial to leverage the theory of subgradient flows.

This proposition is useful for building intuition about the dual dynamics; it is also key to our analyses of mirror flows and their limits, as it allows us to leverage corresponding results on subgradient flows. For instance, this reduction yields the following result.

\begin{thm}
	\label{thm:existence-mirror-flow}
	Let $w_0 \in \dom \partial h^* \cap (\ri \dom h^* + \Span M)$. There exists a unique Lipschitz solution $w : \R_{\geq 0} \to \R^d$ of the mirror flow with mirror potential $h$ and loss $\ell$ such that $w(0) = w_0$. 

Let $x$ be an associated primal variable. Then for a.e.~$t \geq 0$,
	\begin{equation*}
		x(t) \in \underset{y\in\partial h^*(w(t))}{\Argmin} \ell(y) \, .
	\end{equation*}
\end{thm}
If $h^*$ is differentiable on $\dom \partial h^*$, the minimization statement is vacuous as $\partial h^*(w(t)) = \{ \nabla h^*(w(t)) \}$ and $x(t) = \nabla h^*(w(t))$: we get no additional property from the definition of a mirror flow. However, we will study scalings of mirror flows that converge to mirror flows for which $h^*$ is not differentiable and thus this statement is not vacuous: typically, $\partial h^*(w(t))$ will not be a singleton. This result will be key in showing that mirror flows perform incremental learning in these scalings: they minimize $\ell$ over a hypothesis set $\partial h^*(w(t))$ that varies through training (typically with some sharp transitions). 

While there exists a solution such that the dual variable $w$ is Lipschitz, the associated primal variable $x$ might not even be continuous (see Example~\ref{ex:logistic-2}).

We note that the primal variable $x(t)$ is not necessarily unique. However, $\nabla \ell(x(t))$ is a.e.~unique as it is a.e.~the derivative of the unique dual variable $w(t)$. Similarly, by the minimization property above, $\ell(x(t))$ is a.e.~unique.

In the case where $M$ is invertible, the initialization requirement is $w_0 \in \dom \partial h^*$. Moreover, as $\ell$ is strictly convex in this case, the solution to $\min_{y \in \partial h^*(w(t))} \ell(y)$ is unique, so that the primal variable $x$ is uniquely defined a.e.

By~\cite[Corollary 16.18]{bauschke2017convex}, $\ri \dom h^* \subset \dom \partial h^*$; hence $w_0 \in \ri \dom h^*$ is a sufficient condition for the theorem above.

Finally, we note that this theorem compares with~\cite[Theorem 3.1]{attouch2004singular}, which provides a similar existence and uniqueness result when $\ell$ is only assumed to have a Lipschitz gradient, but $h$ is assumed to be strongly convex. In this case, there is also uniqueness of the primal trajectory $x(t)$, thus no minimization property is needed.

\begin{example}[logistic differential equation, continued]
	\label{ex:logistic-3}
	We return to the setting of Example~\ref{ex:logistic-2} where $x \in \R$, $\ell(x) = \frac{1}{2}(x - 1)^2$ and $h(x) = \iota_{\R_{\geq 0}}(x)$. The unique Lipschitz solution is $w(t) = \min(-1+t,0)$. We analyze the minimization property of Theorem~\ref{thm:existence-mirror-flow}.

	We have $h^* = \iota_{\R_{\geq 0}}^* = \sigma_{\R_{\geq 0}} = \iota_{\R_{\leq 0}}$. Thus 
\begin{equation*}
	\partial h^*(w) = \begin{cases}
		\{0\} & \text{if } w < 0 \, , \\
		\R_{\geq 0} & \text{if } w = 0 \, , \\
		\emptyset & \text{if } w > 0 \, .
	\end{cases}
\end{equation*}
By Theorem~\ref{thm:existence-mirror-flow}, we have:
\begin{itemize}
	\item for a.e. $t < 1$, $w(t) < 0$ and thus $x(t) \in \partial h^*(w(t)) = \{0\}$, i.e. $x(t) = 0$.
	\item for a.e. $t \geq 1$, $w(t) = 0$ and thus $x(t) = \argmin_{y \in \partial h^*(w(t))} \ell(y) = \argmin_{y \geq 0} \ell(y) = 1$.
\end{itemize}
These statements are consistent with the solution found in Eq.~\eqref{eq:logistic-primal}. 
\end{example}

\subsection{Limiting dynamics}
\label{sec:limiting}

We now assume that $\dom h^* = \R^d$; note that by~\cite[Coro.~16.18]{bauschke2017convex} we also have $\dom \partial h^* = \R^d$. Let $\varepsilon > 0$ be a small parameter, and let $w^\varepsilon$ be the unique solution of the mirror flow with mirror potential $h$ and loss $\ell$, initialized from $w^\varepsilon(0) = w^\varepsilon_0 \in \R^d$. Let $x^\varepsilon$ denote an associated primal variable. The reader may think of $h$ as smooth, as is the case in all our examples; smoothness is, however, not required for the result below.

Assume that $\mu^\varepsilon := \Vert w^\varepsilon_0 \Vert \xrightarrow[\varepsilon\to 0]{} + \infty$. This assumption is often associated with a primal initialization converging to the boundary of $C$.

\begin{example}[logistic differential equation, continued]
	\label{ex:logistic-4}
	We consider the logistic differential equation introduced in Example~\ref{ex:logistic-1}, which is a mirror flow with entropic mirror potential $h(x) = x \log x - x$ for $x \geq 0$ and $h(x) = +\infty$ otherwise. The domain of the mirror potential is $C = \R_{\geq 0}$. We initialize the primal variable as $x^\varepsilon_0 = \varepsilon$, which converges to the boundary of~$C$ as $\varepsilon \to 0$. This gives the dual initialization $w^\varepsilon_0 = \log \varepsilon$, so that, if $\varepsilon \leq 1$, $\mu^\varepsilon = \vert w^\varepsilon_0 \vert = \log \frac{1}{\varepsilon} \xrightarrow[\varepsilon\to 0]{} +\infty$.
\end{example}

More generally, when $h$ is essentially smooth~\cite[Sec.~26]{rockafellar1970convex}, the divergence of $\mu^\varepsilon$ is tied to the convergence of the primal initialization to the boundary of $C$. By definition, $h$ is essentially smooth if $C$ has nonempty interior, $h$ is differentiable on the interior of $C$, and for any sequence $(x_k)_{k \in \N}$ in the interior of $C$ converging to its boundary, $\lim_{k \to \infty} \Vert \nabla h(x_k) \Vert = +\infty$. Under this assumption, it is sufficient to initialize $x^\varepsilon_0$ in the interior of $C$ and let it converge to the boundary of $C$ as $\varepsilon \to 0$. A sufficient condition for essential smoothness is Legendre type. Example~\ref{ex:logistic-4} and the examples of Sections~\ref{subsec:diagonal}--\ref{subsec:matrix} are of Legendre type.

We return to our general discussion on the limiting dynamics. In the asymptotic regime $\varepsilon\to 0$, we have $\mu^\varepsilon \to + \infty$, and we rescale the primal and dual variables in time and space as follows:
\begin{align*}
	&\overline{w}^\varepsilon(s) = \frac{1}{\mu^\varepsilon} w^\varepsilon(\mu^\varepsilon s) \, , &&\overline{x}^\varepsilon(s) = x^\varepsilon(s \mu^\varepsilon) \, , &&s \geq 0 \, .
\end{align*}
This time rescaling is the one required to obtain a nontrivial limit and incremental learning. Indeed, the rescaled functions themselves follow a mirror flow: for a.e.~$s \geq 0$, $\overline{w}^\varepsilon$ is differentiable at $s$ and
\begin{align*}
	&\frac{\diff \overline{w}^\varepsilon}{\diff s}(s) = -\nabla \ell(\overline{x}^\varepsilon(s)) \, , &&\overline{w}^\varepsilon(s) \in \partial \left( \frac{1}{\mu^\varepsilon} h \right) (\overline{x}^\varepsilon(s)) \, ,
\end{align*}
see Section~\ref{sec:proof-main} for the computations. Informally, taking the limit $\varepsilon \to 0$, we have $\displaystyle \frac{1}{\mu^\varepsilon} h \to \iota_C$, and thus it is natural to expect convergence to the mirror flow with mirror potential $\iota_C$ and loss $\ell$, initialized at $\overline{w}_0 = \lim_{\varepsilon \to 0} \overline{w}^\varepsilon(0) = \lim_{\varepsilon \to 0} \frac{1}{\mu^\varepsilon} w^\varepsilon_0$: for a.e.~$s \geq 0$, $\overline{w}$ is differentiable at $s$ and
	\begin{align}
		\label{eq:limit_mirror_flow}
		&\frac{\diff\overline{w}}{\diff s}(s) = -\nabla \ell(\overline{x}(s)) \, , &&\overline{w}(s) \in \partial \iota_C(\overline{x}(s)) \, .
	\end{align}
The following theorem specifies under which assumptions and in which sense the above derivation is true.
\begin{thm}
	\label{thm:main}
	Denote $\overline{w}_0^\varepsilon = \overline{w}^\varepsilon(0) = \frac{1}{\mu^\varepsilon} w^\varepsilon_0$. We assume that:
	\begin{enumerate}[label=(\roman*)]
		\item\label{it:main-1} as $\varepsilon \to 0$, $\overline{w}^\varepsilon_0$ converges to some $\overline{w}_0 \in \dom \partial \sigma_C \cap \left(\ri \dom \sigma_C + \Span M\right)$, 
		\item\label{it:main-2} for $\varepsilon > 0$ small enough, $\overline{w}_0^\varepsilon - \overline{w}_0 \in \aff \dom \sigma_C$.
	\end{enumerate}
	Then there exists a unique Lipschitz solution $\overline{w}$ of the mirror flow with mirror potential~$\iota_C$ and loss~$\ell$, initialized at $\overline{w}(0) = \overline{w}_0$, as written in Eq.~\eqref{eq:limit_mirror_flow}. Let $\overline{x}$ be an associated primal variable. Let $S > 0$. Then:
	\begin{enumerate}[label=(\alph*)]
		\item\label{it:main-statement-1} $\overline{w}^\varepsilon \xrightarrow[\varepsilon\to 0]{} \overline{w}$ uniformly on $[0,S]$.
		\item \label{it:main-statement-2} We have the convergences
		\begin{align*}
			\nabla \ell(\overline{x}^\varepsilon(s))
			&\xrightarrow[\varepsilon\to 0]{}
			\nabla \ell(\overline{x}(s))
			&&\text{in } L^2([0,S],\R^d, s\diff s),\\
			\ell(\overline{x}^\varepsilon(s))
			&\xrightarrow[\varepsilon\to 0]{}
			\ell(\overline{x}(s))
			= \min_{y \in \partial \sigma_C(\overline{w}(s))} \ell(y)
			&&\text{in } L^1([0,S],\R, s\diff s).
		\end{align*}
		\item Let $\overline{y}:[0,S]\to\R^d$ be an accumulation point of $\overline{x}^\varepsilon$ in the sense of essential pointwise convergence: there exists a subsequence $\varepsilon_k \xrightarrow[k\to\infty]{} 0$ such that, for a.e.~$s \in [0,S]$, $\overline{x}^{\varepsilon_k}(s) \xrightarrow[k\to\infty]{} \overline{y}(s)$. Then, for a.e.~$s \in [0,S]$, $\overline{y}(s) \in \underset{y \in \partial \sigma_C(\overline{w}(s))}{\Argmin} \ell(y)$.
	\end{enumerate}
\end{thm}

As before, $\overline{w}_0 \in \ri \dom \sigma_C$ is a sufficient condition for requirement~\ref{it:main-1} of the theorem.

In the case where $M$ is invertible, requirement~\ref{it:main-1} reduces to $\overline{w}_0 \in \dom \partial \sigma_C$. Moreover, in this case, the primal variable $\overline{x}(s)$ is uniquely defined and $\overline{x}^\varepsilon$ converges to $\overline{x}$ in $L^2([0,S],\R^d, s\diff s)$ by statement~\ref{it:main-statement-2}.

Note that the $L^2$ and $L^1$ convergence statements are with respect to the measure $s \diff s$, that is, the measure with density $s$ with respect to the Lebesgue measure. This technicality is inherited from the theory of convergence of subgradient flows~\cite{attouch1977convergence}; unweighted convergence should not be expected. 

Requirement~\ref{it:main-2} is automatic whenever $\aff \dom \sigma_C = \R^d$ (in particular, it holds in all three examples of Section~\ref{sec:examples}); it is best viewed as a non-degeneracy assumption on the approach of $\overline{w}_0^\varepsilon$ to $\overline{w}_0$.

\begin{example}[logistic differential equation, continued]
	\label{ex:logistic-5}
	Theorem~\ref{thm:main} shows that the solution of the logistic differential equation, when initialized at $x^\varepsilon_0 = \varepsilon$ and rescaled in time as $\overline{x}^\varepsilon(s) = x^\varepsilon(\mu^\varepsilon s)$ with $\mu^\varepsilon = \log \frac{1}{\varepsilon}$, converges to the primal solution of the mirror flow with mirror potential $h = \iota_{\R_{\geq 0}}$ given in Eq.~\eqref{eq:logistic-primal}. In this example, the convergence can also be derived by elementary means. Recall that $x^\varepsilon(t) = f(t - t_0^\varepsilon)$ where $f$ is the logistic function and $t_0^\varepsilon$ satisfies $f(-t_0^\varepsilon) = \varepsilon$. This implies $t_0^\varepsilon \xrightarrow[\varepsilon\to 0]{} +\infty$, and thus $\varepsilon = f(-t_0^\varepsilon) = \frac{1}{1+e^{t_0^\varepsilon}} = e^{-t_0^\varepsilon}(1+o(1))$. Taking logarithms, we get $t_0^\varepsilon = \mu^\varepsilon + o(1)$. We thus obtain
	\begin{equation*}
		\overline{x}^\varepsilon(s) = x^\varepsilon(\mu^\varepsilon s) = f(\mu^\varepsilon s - t_0^\varepsilon) = f((s-1) \mu^\varepsilon + o(1)) \xrightarrow[\varepsilon\to 0]{} \begin{cases}
		0 & \text{if } s < 1 \, , \\
		1 & \text{if } s > 1 \, .
	\end{cases}
	\end{equation*} 
	In other words, the rescaled primal variable $\overline{x}^\varepsilon$ converges a.e.~pointwise to the primal solution of the limiting mirror flow, provided in Eq.~\eqref{eq:logistic-primal}. A similar computation shows the convergence of the rescaled dual variable.
\end{example}

\section{Applications and simulations}
\label{sec:examples}

\subsection{Optimization on the non-negative orthant \texorpdfstring{$\R_{\geq 0}^d$}{}}
\label{subsec:diagonal}

\subsubsection{Limiting dynamics}
 
We consider the minimization of the loss $\ell(x) = \frac{1}{2} \langle x, M x \rangle - \langle q, x \rangle + c$, with $M$ positive semidefinite and $q \in \Span M$, via the mirror flow with entropic mirror potential $h(x) = \sum_{i=1}^{d} (x_i \log x_i - x_i)$ and domain $\dom h = \R_{\geq 0}^d$. We also have $\nabla h(x) = \log x$, with inverse $\nabla h^*(w) = e^w$, and Fenchel conjugate $h^*(w) = \sum_{k = 1}^d e^{w_k}$, with $\dom h^* = \R^d$.
 
\paragraph{Support function.} We have $C = \dom h = \R_{\geq 0}^d$, $\sigma_{C} = \sigma_{\R_{\geq 0}^d} = \iota_{\R_{\leq 0}^d}$, and $\dom \sigma_C = \dom \partial \sigma_C = \R_{\leq 0}^d$. Requirement~\ref{it:main-1} of Theorem~\ref{thm:main} is satisfied if $\overline{w}_0 \in \R_{< 0}^d$. Moreover, $\aff \dom \sigma_C = \R^d$, so requirement~\ref{it:main-2} of Theorem~\ref{thm:main} is automatically satisfied.
\paragraph{The limiting dynamics.} We initialize the mirror flow at $x_0^\varepsilon = \varepsilon \mathbf{1}$, that is, $w_0^\varepsilon = (\log \varepsilon)\mathbf{1}$ where $\mathbf{1} = (1, \ldots, 1) \in \R^d$. Thus $\mu^\varepsilon = \|w_0^\varepsilon\| = \sqrt{d}\log(1/\varepsilon) \to +\infty$ and $\overline{w}_0 = \lim_{\varepsilon \to 0} w_0^\varepsilon / \mu^\varepsilon = -\mathbf{1}/\sqrt{d}$. Hence requirements~\ref{it:main-1} and~\ref{it:main-2} are satisfied, and applying Theorem~\ref{thm:main}, the limiting mirror flow reads:
\begin{align}
	\label{eq:diagonal-limit-dual}
	&\frac{\diff \overline{w}}{\diff s} (s) = - \nabla \ell(\overline{x}(s))  \, , &&\overline{x}(s) \in \partial \iota_{\R_{\leq 0}^d}(\overline{w}(s)) \, .
\end{align}
\paragraph{Primal dynamics and sparsity.} For $w \leq 0$, the set $\partial \iota_{\R_{\leq 0}^d}(w)$ is the convex set of vectors complementary to $w$ in the following sense:
\begin{equation*}
	\partial \iota_{\R_{\leq 0}^d}(w) = \left\{ x \in \R_{\geq 0}^d \,\middle|\, \langle x , w \rangle = 0 \right\} \, .
\end{equation*}
Therefore, the limiting primal trajectory satisfies $\overline{x}_i(s) = 0$ whenever $\overline{w}_i(s) < 0$, and $\overline{x}_i(s) \geq 0$ whenever $\overline{w}_i(s) = 0$. Define the \emph{active set} $I(w) = \{i \mid w_i = 0\}$; then $\overline{x}(s)$ is the minimizer of~$\ell$ restricted to coordinates in $I(\overline{w}(s))$ with the constraint $x \geq 0$:
\begin{equation*}
	\overline{x}(s) \in \Argmin \left\{  \ell(y) \,\middle|\, y \geq 0 , \, \supp(y) \subset I(\overline{w}(s)) \right\} \, .
\end{equation*}
The trajectory $\overline{w}(s)$ is piecewise linear, and the active set $I(\overline{w}(s))$ is piecewise constant, with changes at the times when some coordinate $\overline{w}_i(s)$ hits $0$. Consequently, the limiting primal trajectory $\overline{x}(s)$ can be chosen piecewise constant; in each phase it is the solution of a constrained quadratic optimization problem.

\paragraph{Implementation.} We simulate the original mirror flow for two values $\varepsilon \in \{10^{-20}, 10^{-100}\}$ and compare it to the limiting dynamics in Eq.~\eqref{eq:diagonal-limit-dual}. The result is shown in Figure~\ref{fig:diagonal_simulation}. In this simulation, $M \in \R^{d \times d}$, with $d = 8$, is generated randomly as a Wishart matrix with $n = 6$ degrees of freedom and scale matrix $\I_d$, that is $M = \frac{1}{n} \sum_{i = 1}^n a_i a_i^\top$, where the $a_i$ are i.i.d.\ vectors in $\R^d$ distributed according to $\mathcal{N}(0,\I_d)$. The vector $q \in \Span M$ is generated as $q = M z$, where $z$ is a standard Gaussian vector. The visualization of the dual trajectory of $\overline{w}$ is given in Appendix~\ref{app:diagonal_simulation}. Further implementation details are given in Appendix~\ref{app:implementation}.

\begin{figure}[ht]
	\centering
	\includegraphics[width=0.95\textwidth]{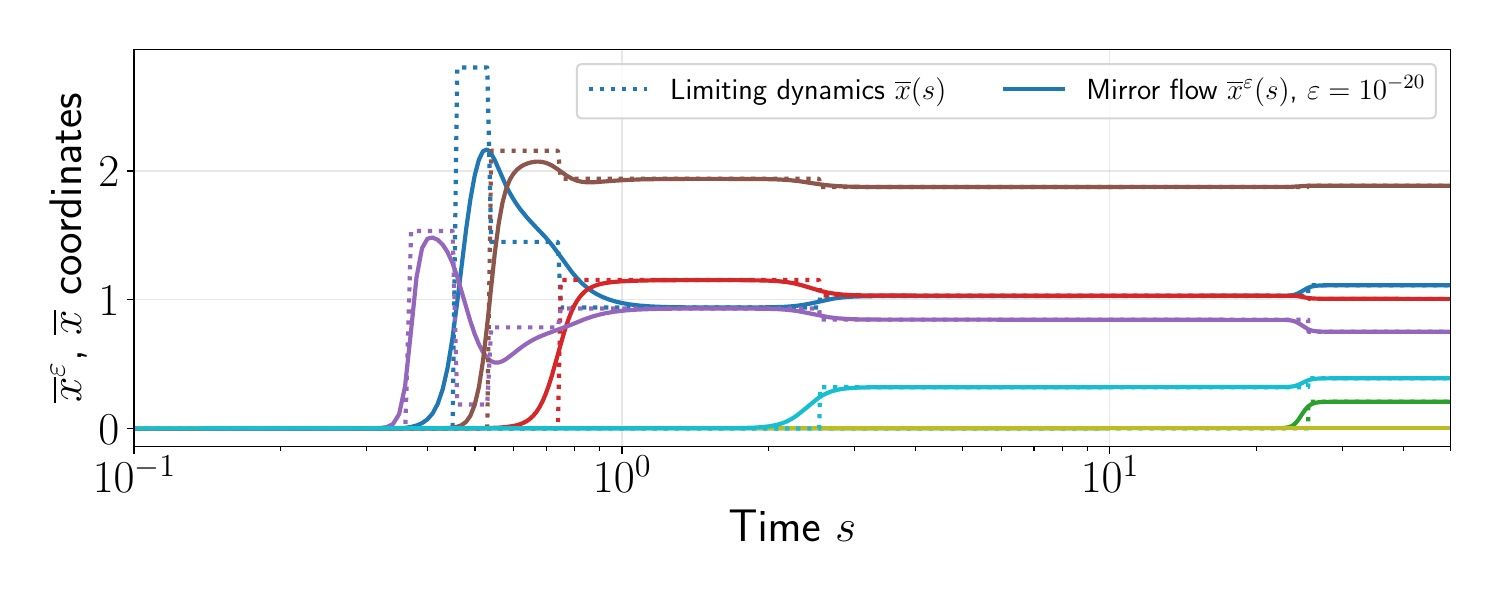}
	\includegraphics[width=0.95\textwidth]{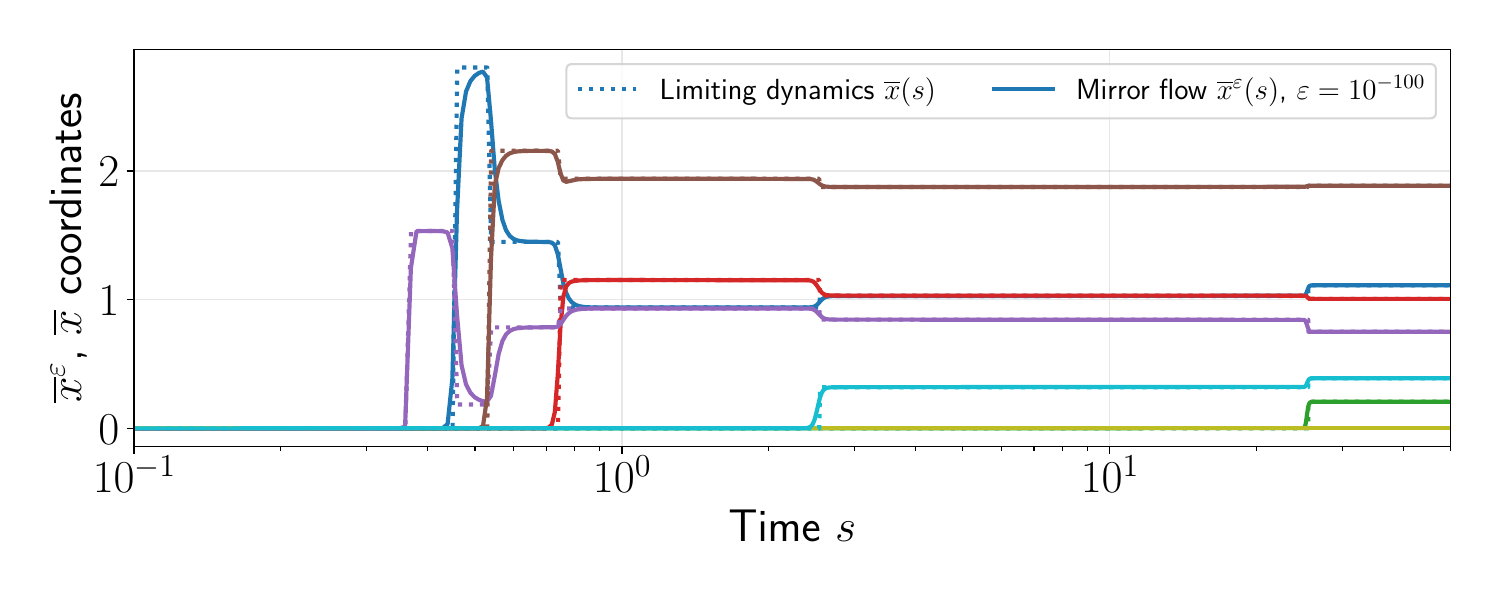}
	\includegraphics[width=0.95\textwidth]{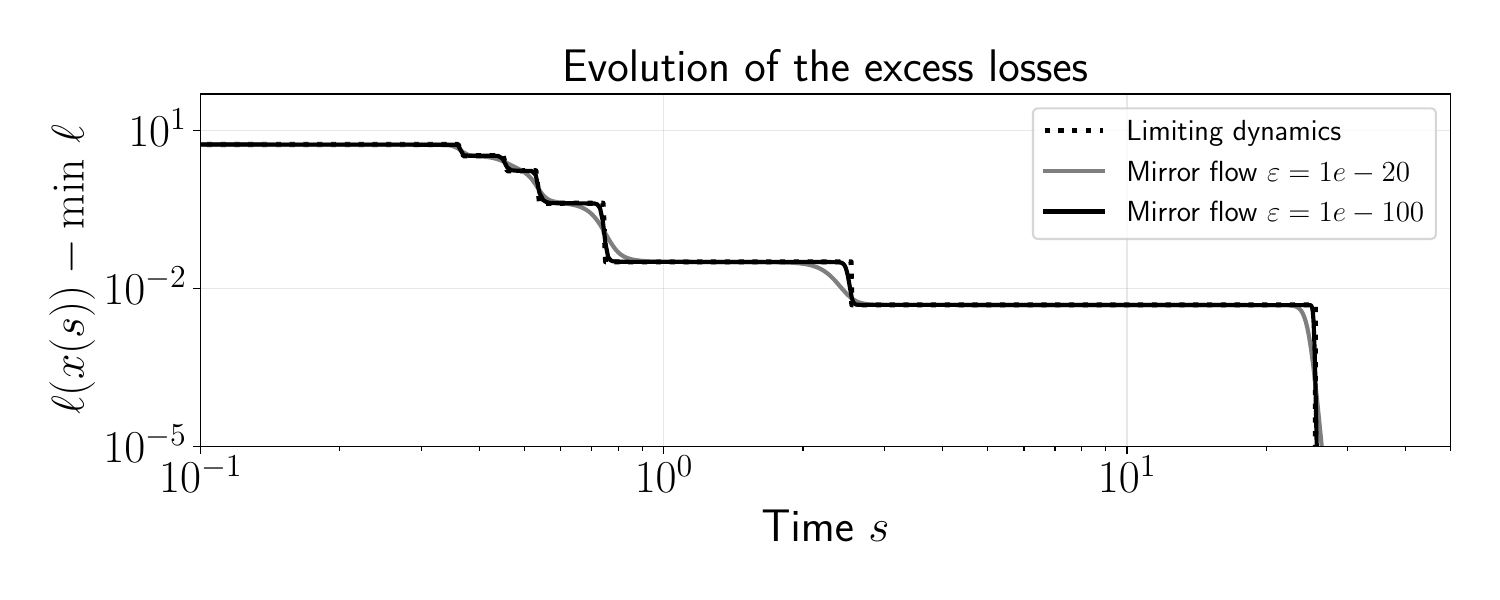}
	\caption{\textbf{Non-negative orthant.} Simulation of the mirror flow for two small values of $\varepsilon$ and comparison with the limiting dynamics. The top two plots show the coordinate trajectories of $x$ for the mirror flow (solid line) and for the limiting dynamics (dotted line). The top plot corresponds to $\varepsilon = 10^{-20}$ and the middle plot to $\varepsilon = 10^{-100}$. The bottom plot displays the excess loss, i.e. the loss minus its minimum over $\R^d_{\geq 0}$. The mirror flow and the limiting dynamics agree well. The trajectories exhibit complex sparse activation patterns, and the loss is non-increasing and piecewise constant, illustrating the \textit{incremental learning} of the iterates.}
	\label{fig:diagonal_simulation}
\end{figure}

\subsubsection{Connection to diagonal linear networks}
\label{subsec:diagonal_nn}

Mirror flows on the non-negative orthant with the entropic mirror potential arise in the study of diagonal linear networks (DLNs)~\cite{li2022implicit}. This connection has been central to analyses of incremental learning in DLNs~\cite{berthier2023incremental,pesme2023saddle}. We recall it here and interpret our results in this context.

A two-layer DLN implements a linear relation $b = \langle x , a \rangle$ between its input $a \in \R^d$ and its output $b \in \R$. However, the regression parameter $x \in \R^d$ is itself parametrized as the componentwise product $x = v \circ u$ of the outer-layer weights $v \in \R^d$ and the inner-layer diagonal weights $u \in \R^d$. When minimizing the quadratic loss $\ell(x)$, the gradient-flow dynamics are run in terms of the weights $u$ and $v$, rather than in terms of $x$ itself:
\begin{align}
	\label{eq:dln-asymmetric}
	&\frac{\diff u}{\diff t} = - \nabla_u\left(\ell(v \circ u)\right) \, , &&\frac{\diff v}{\diff t} = - \nabla_v\left(\ell(v \circ u)\right) \, .
\end{align}
The induced dynamics on $x$ can be understood through a mirror flow, and the results of this paper describe the incremental learning dynamics in the limit of small initialization. 

\paragraph{Symmetric case.} We start with the symmetric case $u = v$. This can be interpreted as a weight-tied DLN in which both layers share the same weight; more importantly for us, it is a convenient theoretical intermediate setting in which the main ideas can be derived before turning to the asymmetric case~\cite{berthier2025diagonal}.

In this case, Eqs.~\eqref{eq:dln-asymmetric} become
\begin{align}
	\label{eq:dln-symmetric}
	&\frac{\diff u}{\diff t} = - \nabla_u\left(\ell(u \circ u)\right) = -2 u \circ \nabla \ell(u \circ u) \, . 
\end{align}
The induced dynamics on $x = u \circ u$ are 
\begin{align}
	\label{eq:dln-symmetric-x}
	&\frac{\diff x}{\diff t} = 2 u \circ \frac{\diff u}{\diff t} = -4 u \circ u \circ \nabla \ell(u \circ u) = -4 x \circ \nabla \ell(x) \, ,
\end{align}
	that is,
\begin{align*}
	&\frac{\diff w}{\diff t} = - \nabla \ell(x) \, , \qquad w = \nabla h(x) \, , \\
	&\text{where } h(x) = \frac{1}{4} \sum_{i=1}^{d} (x_i \log x_i - x_i) \, , \qquad x \geq 0 \, .
\end{align*}
In words, $x$ is the primal variable of a mirror flow with the entropic mirror potential (up to a multiplicative constant).

Incremental learning occurs when the primal variable $x$ of the mirror flow is initialized close to the boundary of $C = \dom h = \R_{\geq 0}^d$. This corresponds to (at least some of) the weights~$u$ being initialized at a small value. Let $\varepsilon > 0$ be a small parameter and $u^\varepsilon$ denote the solution of the DLN dynamics \eqref{eq:dln-symmetric} initialized from $u^\varepsilon_0 \in \R^d$. Denote $x^\varepsilon = u^\varepsilon \circ u^\varepsilon$ the solution of the corresponding mirror flow \eqref{eq:dln-symmetric-x}, initialized from $x^\varepsilon_0 = u^\varepsilon_0 \circ u^\varepsilon_0$. The application of Theorem~\ref{thm:main} readily gives the following result.

\begin{thm}
	Assume that, for all $i \in \{1, \dots, d\}$, $u^\varepsilon_{0,i} \neq 0$, and 
	\begin{equation*}
		\mu^\varepsilon := \frac{1}{2} \left(\sum_{i=1}^{d} \log^2 \vert u^\varepsilon_{0,i} \vert\right)^{1/2} \xrightarrow[\varepsilon\to 0]{} +\infty \, .
	\end{equation*}
	Denote $\overline{x}^\varepsilon(s) = x^\varepsilon(\mu^\varepsilon s)$, $s \geq 0$. We assume that $\frac{1}{2\mu^\varepsilon} \log \vert u^\varepsilon_{0} \vert$ converges to some limit $\overline{w}_0 \in \R_{\leq 0}^d \cap \left(\R^d_{<0} + \Span M\right)$. Under this assumption, there exists a unique Lipschitz solution $\overline{w}$ of the mirror flow with mirror potential $\iota_{\R_{\geq 0}^d}$ and loss $\ell$, initialized at $\overline{w}_0$. Let $\overline{x}$ be an associated primal variable. Let $S > 0$. Then:
	\begin{enumerate}[label=(\alph*)]
		\item\label{it:dln-1} We have the convergences
		\begin{align*}
			\nabla \ell(\overline{x}^\varepsilon(s))
			&\xrightarrow[\varepsilon\to 0]{}
			\nabla \ell(\overline{x}(s))
			&&\text{in } L^2([0,S],\R^d, s\diff s),\\
			\ell(\overline{x}^\varepsilon(s))
			&\xrightarrow[\varepsilon\to 0]{}
			\ell(\overline{x}(s))
			= \min_{y \geq 0\,,\, \langle \overline{w}(s), y \rangle = 0} \ell(y)
			&&\text{in } L^1([0,S],\R, s\diff s).
		\end{align*}
		\item\label{it:dln-2} Let $\overline{y}:[0,S]\to\R^d$ be an accumulation point of $\overline{x}^\varepsilon$ in the sense of essential pointwise convergence. Then, for a.e.~$s \in [0,S]$, $\overline{y}(s) \in \underset{y \geq 0\,,\,  \langle \overline{w}(s), y \rangle = 0}{\Argmin} \ell(y)$.
	\end{enumerate}
\end{thm}

To illustrate the results, one can choose, for instance, $u^\varepsilon_0 = \sqrt{\varepsilon} \mathbf{1}$, in which case we obtain $\overline{w}_0 = - \frac{1}{\sqrt{d}} \mathbf{1} \in \R^d_{<0} \subset \R_{\leq 0}^d \cap \left(\R^d_{<0} + \Span M\right)$. 

A similar description was obtained in~\cite{berthier2023incremental} for a specific class of Hessians~$M$ in the so-called underparameterized setting. Our generalization also covers the overparameterized setting.

\paragraph{The asymmetric case.} We now turn to the case $u \neq v$. Let $u^\varepsilon$ and $v^\varepsilon$ be the solutions of the DLN dynamics \eqref{eq:dln-asymmetric}, initialized from $u^\varepsilon_0$ and $v^\varepsilon_0$, respectively, and let $x^\varepsilon = v^\varepsilon \circ u^\varepsilon$ be the linear parameter. The connection to mirror flows is more involved in this case; we state the result directly and postpone the derivation to the proof.

\begin{thm}
	\label{thm:dln-asym}
	Assume that, for all $i \in \{1, \dots, d\}$, $u^\varepsilon_{0,i} \neq \pm v^\varepsilon_{0,i}$, and
	\begin{equation*}
		\mu^\varepsilon := \left(\sum_{i=1}^d \left[\log^2 \left\vert \frac{u^\varepsilon_{0,i} + v^\varepsilon_{0,i}}{2}\right\vert + \log^2 \left\vert \frac{u^\varepsilon_{0,i} - v^\varepsilon_{0,i}}{2} \right\vert\right]\right)^{1/2} \xrightarrow[\varepsilon\to 0]{} +\infty \, .
	\end{equation*}
	Denote $\overline{x}^\varepsilon(s) = x^\varepsilon(\mu^\varepsilon s)$, $s \geq 0$. We assume that
	\begin{itemize}
		\item $\frac{1}{\mu^\varepsilon} \log \left\vert \frac{u^\varepsilon_{0} + v^\varepsilon_{0}}{2}\right\vert$ converges to some limit $\overline{w}_0^{\pos} \in \R^d_{\leq 0}$,
		\item $\frac{1}{\mu^\varepsilon} \log \left\vert \frac{u^\varepsilon_{0} - v^\varepsilon_{0}}{2}\right\vert$ converges to some limit $\overline{w}_0^{\neg} \in \R^d_{\leq 0}$, and
		\item there exists $a \in \Span M$ such that $\overline{w}_0^{\pos} + a < 0$ and $\overline{w}_0^{\neg} - a < 0$. 
	\end{itemize}
		We write $\overline{w}_0 = \frac{\overline{w}_0^\pos-\overline{w}_0^\neg}{2}$ and $\lambda = -\frac{\overline{w}_0^\pos+\overline{w}_0^\neg}{2}$. The third assumption above implies $\lambda > 0$. Under these assumptions, there exists a unique Lipschitz solution $\overline{w}$ of the mirror flow with mirror potential $h(x) = \sum_{i=1}^{d} \lambda_i \vert x_i \vert$ and loss $\ell$, with initialization $\overline{w}(0) = \overline{w}_0$. Let $\overline{x}$ be an associated primal variable. Let $S > 0$. Then:
	\begin{enumerate}[label=(\alph*)]
		\item We have the convergences
		\begin{align*}
			\nabla \ell(\overline{x}^\varepsilon(s))
			&\xrightarrow[\varepsilon\to 0]{}
			\nabla \ell(\overline{x}(s))
			&&\text{in } L^2([0,S],\R^d, s\diff s),\\
			\ell(\overline{x}^\varepsilon(s))
			&\xrightarrow[\varepsilon\to 0]{}
			\ell(\overline{x}(s))
			&&\text{in } L^1([0,S],\R, s\diff s),
	\end{align*}
	where 
	\begin{equation*}
		\ell(\overline{x}(s)) = \min \left\{\ell(y)\,\middle|\, y\in \R^d \text{ satisfying }\begin{cases}
		\text{if }\overline{w}_i(s) = \lambda_i, & y_i \geq 0 \, , \\
		\text{if }\overline{w}_i(s) \in (-\lambda_i, \lambda_i), & y_i = 0 \, , \\
		\text{if }\overline{w}_i(s) = -\lambda_i, & y_i \leq 0 \, .
		\end{cases}\right\} \, .
	\end{equation*}
	\item Let $\overline{y}:[0,S]\to\R^d$ be an accumulation point of $\overline{x}^\varepsilon$ in the sense of essential pointwise convergence. Then, for a.e.~$s \in [0,S]$,
	\begin{equation*}
		\overline{y}(s) \in \Argmin\left\{\ell(y)\,\middle|\, y\in \R^d \text{ satisfying }\begin{cases}
		\text{if }\overline{w}_i(s) = \lambda_i, & y_i \geq 0 \, , \\
		\text{if }\overline{w}_i(s) \in (-\lambda_i, \lambda_i), & y_i = 0 \, , \\
		\text{if }\overline{w}_i(s) = -\lambda_i, & y_i \leq 0 \, .
		\end{cases}\right\} \, .
	\end{equation*}
	\end{enumerate} 
\end{thm}
The proof of the theorem is given in Appendix~\ref{sec:proof-dln-asym}. To illustrate the theorem, take for instance $u^\varepsilon_0 = 2 \varepsilon \mathbf{1}$ and $v^\varepsilon_0 = 0$. Then $\mu^\varepsilon = \sqrt{2d}\log(1/\varepsilon)$, $\overline{w}_0^\pos = \overline{w}_0^\neg = -\frac{1}{\sqrt{2d}} \mathbf{1}$, hence $\overline{w}_0 = 0$ and $\lambda = \frac{1}{\sqrt{2d}} \mathbf{1}$. The limiting dynamics correspond to a mirror flow whose mirror potential is the $\ell_1$-norm, up to a multiplicative constant:
\begin{align}
	\label{eq:iss}
	&\frac{\diff \overline{w}}{\diff s} (s) = - \nabla \ell(\overline{x}(s)) \, , &&\overline{w}(s) \in \frac{1}{\sqrt{2d}}\partial\left(\Vert \cdot \Vert_1\right)(\overline{x}(s)) \, .
\end{align}
This corresponds to the results derived in \cite{pesme2023saddle} under a certain ``general position'' assumption that we do not need here.

The dynamics \eqref{eq:iss} correspond to the inverse scale space method~\cite{burger2005nonlinear,burger2007error,burger2013adaptive}. This flow is valued in sparse recovery as an unbiased alternative to the Lasso regularization path~\cite{osher2016sparse}. It is remarkable that diagonal linear networks implement an inverse scale space flow implicitly, through a product parametrization of the regressor and an infinitesimal initialization.

\subsection{Optimization on the positive semidefinite cone \texorpdfstring{$\mathcal{S}^d_+$}{}}
\label{subsec:matrix}

\subsubsection{Limiting dynamics}
 
We now consider the matrix analogue. Let $\mathcal{S}^{d}$ be the space of $d \times d$ real symmetric matrices endowed with the trace inner product $\langle A, B \rangle = \Tr(AB)$. We consider the loss $\ell(X) = \frac{1}{2} \langle X, \mathcal{M}(X) \rangle - \langle Q, X \rangle + c$, where $\mathcal{M} : \mathcal{S}^d \to \mathcal{S}^d$ is a self-adjoint positive semidefinite linear operator and $Q$ is in the image of $\mathcal{M}$. Lifting $\log$ and $\exp$ to spectral functions on $\mathcal{S}^d$ (acting on the eigenvalues), we define the \emph{von Neumann entropic mirror potential} by
\begin{equation*}
	h(X) =
	\begin{cases}
		\Tr(X \log X - X) & \text{if } X \in \mathcal{S}^{d}_+ \, ,\\
		+\infty & \text{otherwise.}
	\end{cases}
\end{equation*}
This mirror potential is convex~\cite[Theorem 4.3.5]{bhatia2009positive}. The domain of $h$ is the cone of positive semidefinite matrices, i.e. $\dom h = \mathcal{S}^{d}_+$. We also have $\nabla h(X) = \log X$, with inverse $\nabla h^*(W) = \exp W$, and Fenchel conjugate $h^*(W) = \Tr(\exp W)$, with $\dom h^* = \mathcal{S}^d$.
 
\paragraph{Support function.} We have $C = \dom h = \mathcal{S}^{d}_+$ and $\sigma_C = \sigma_{\mathcal{S}^{d}_+} = \iota_{\mathcal{S}^{d}_-}$, where $\mathcal{S}^{d}_- = \{V \in \mathcal{S}^d \mid V \preceq 0\}$ is the cone of negative semidefinite matrices. Thus $\dom \sigma_C = \dom \partial \sigma_C = \mathcal{S}^{d}_-$. Requirement~\ref{it:main-1} of Theorem~\ref{thm:main} is satisfied if $\overline{W}_0 \in \{V \in \mathcal{S}^d \mid V \prec 0\}$, the cone of negative definite matrices. Moreover, $\aff \dom \sigma_C = \mathcal{S}^d$, so requirement~\ref{it:main-2} of Theorem~\ref{thm:main} is automatically satisfied.
 
\paragraph{The limiting dynamics.} We initialize the primal variable of the mirror flow at $X^\varepsilon_0 = \varepsilon \I_d$, that is, $W^\varepsilon_0 = (\log \varepsilon)\I_d \in \mathcal{S}^d_-$. Then $\mu^\varepsilon = \|W^\varepsilon_0\| = \sqrt{d}\log(1/\varepsilon) \to +\infty$ and $\overline{W}_0 = \lim_{\varepsilon \to 0} W^\varepsilon_0 / \mu^\varepsilon = -\I_d/\sqrt{d}$. Hence requirements~\ref{it:main-1} and~\ref{it:main-2} are satisfied, and applying Theorem~\ref{thm:main}, the limiting dual flow reads:
\begin{align*}
	&\frac{\diff \overline{W}}{\diff s} (s) = - \nabla \ell(\overline{X}(s)) \, , &&\overline{X}(s) \in \partial \iota_{\mathcal{S}^{d}_-}(\overline{W}(s)) \, .
\end{align*}
 
\paragraph{Primal dynamics and low-rank bias.} For $W \in \mathcal{S}^d_-$, the set $\partial \iota_{\mathcal{S}^{d}_-}(W)$ is the convex set of matrices complementary to $W$ in the following sense:
\begin{equation*}
	\partial \iota_{\mathcal{S}^{d}_-}(W) = \left\{ X \in \mathcal{S}^d_+ \,\middle|\, \Span X \subset \ker W \right\} \, .
\end{equation*}
Denoting $r(s) = \dim \ker \overline{W}(s)$, the limiting primal variable $\overline{X}(s)$ is positive semidefinite with $\rank \overline{X}(s) \leq r(s)$ and minimizes $\ell$ over this convex set, that is,
\begin{equation*}
	\overline{X}(s) \in \Argmin \left\{ \ell(Y) \,\middle|\, Y \succeq 0 , \, \Span Y \subset \ker \overline{W}(s)\right\} \, .
\end{equation*}
The trajectory $\overline{W}(s)$ is no longer piecewise linear. Slow dynamics occur within each fixed-rank set of predictions, and the limiting primal trajectory $\overline{X}(s)$ is not piecewise constant, although it exhibits discontinuities when a new eigenvalue activates. This illustrates the \textit{incremental learning} of the iterates, as the rank of the predictions increases through training.

\paragraph{Implementation.} We simulate the original mirror flow for two values $\varepsilon \in \{10^{-20}, 10^{-100}\}$ and compare it to the limiting dynamics. The result is shown in Figure~\ref{fig:MM_simulation}. Here, $\mathcal{M}$ is a Wishart operator on $\mathcal{S}^d$, i.e. $\mathcal{M} = \frac{1}{n} \sum_{i = 1}^n A_i \otimes A_i$, where the $(A_i)_{i \leq n}$ are i.i.d.\ GOE matrices: symmetric matrices with independent $\mathcal{N}(0,1)$ diagonal entries and $\mathcal{N}(0,1/2)$ off-diagonal entries. The matrix $Q$ is generated as $Q = \mathcal{M}(Z)$, where $Z$ is an independent GOE matrix. We take $n = 6$ and $d = 8$. The visualization of the dual trajectory of $\overline{W}$ is given in Appendix~\ref{app:MM_simulation}. Further implementation details are given in Appendix~\ref{app:implementation}.

\begin{figure}[ht]
	\centering
	\includegraphics[width=0.94\textwidth]{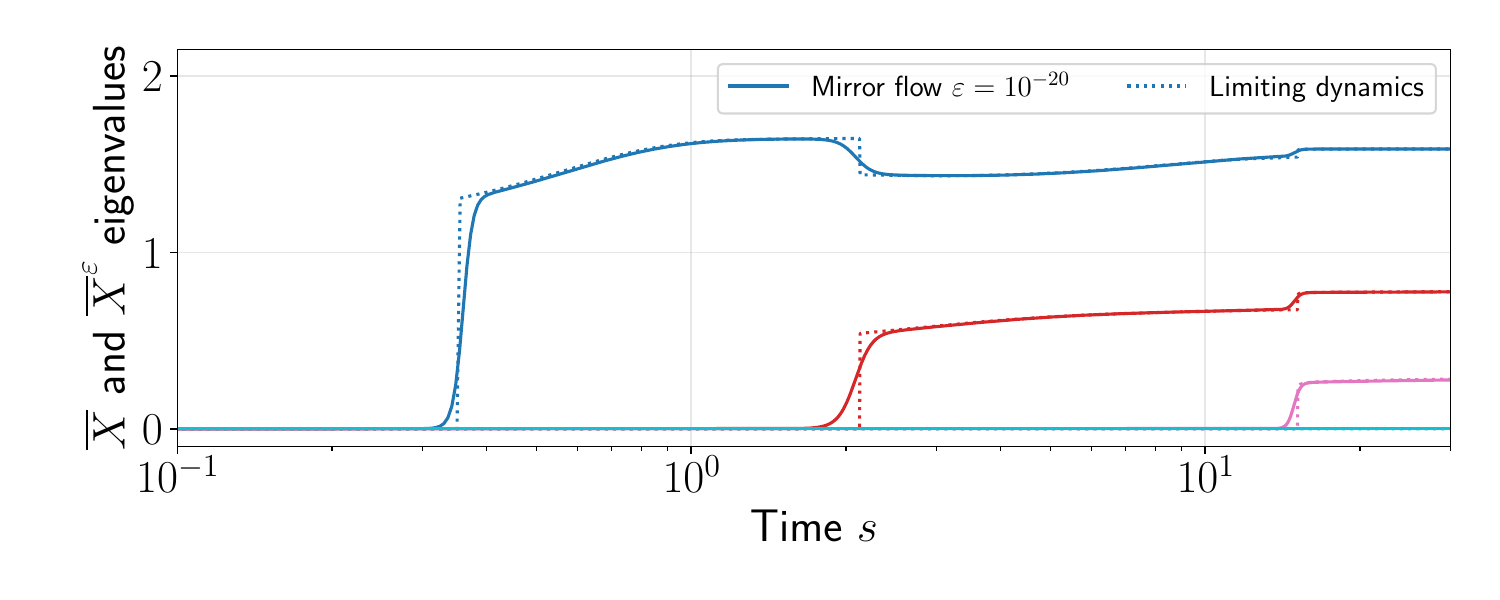}
	\includegraphics[width=0.94\textwidth]{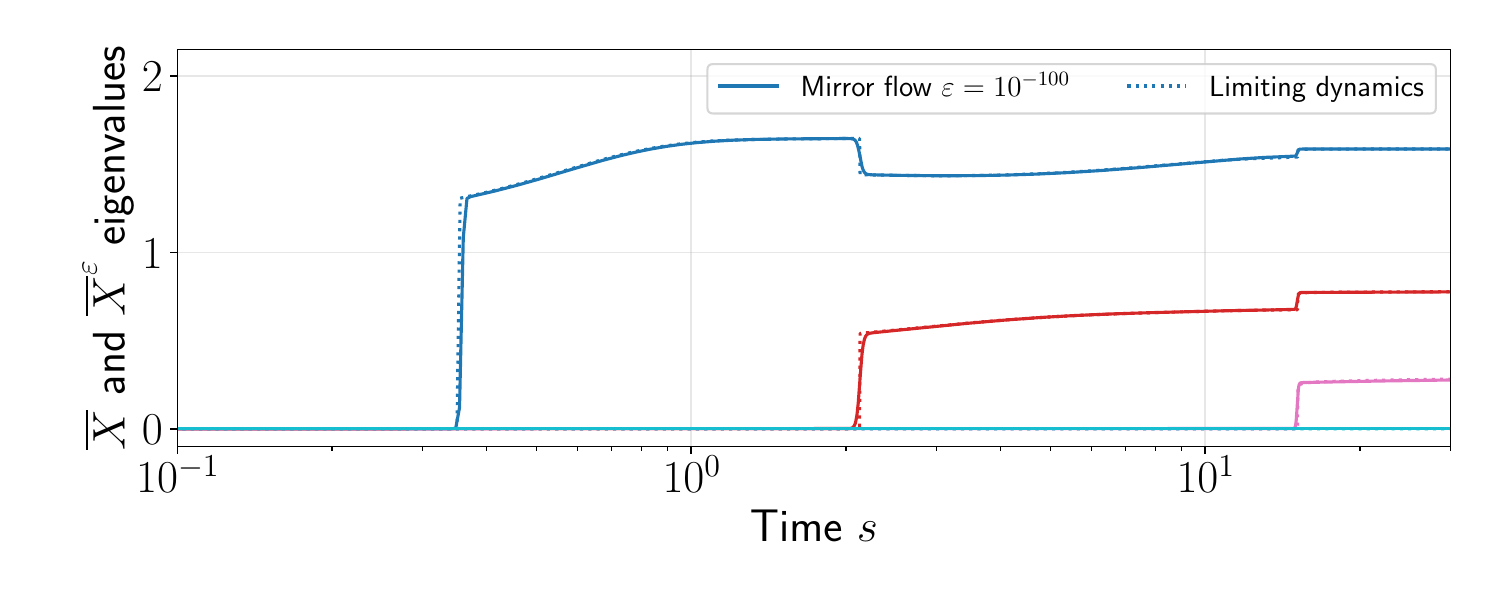}
	\includegraphics[width=0.94\textwidth]{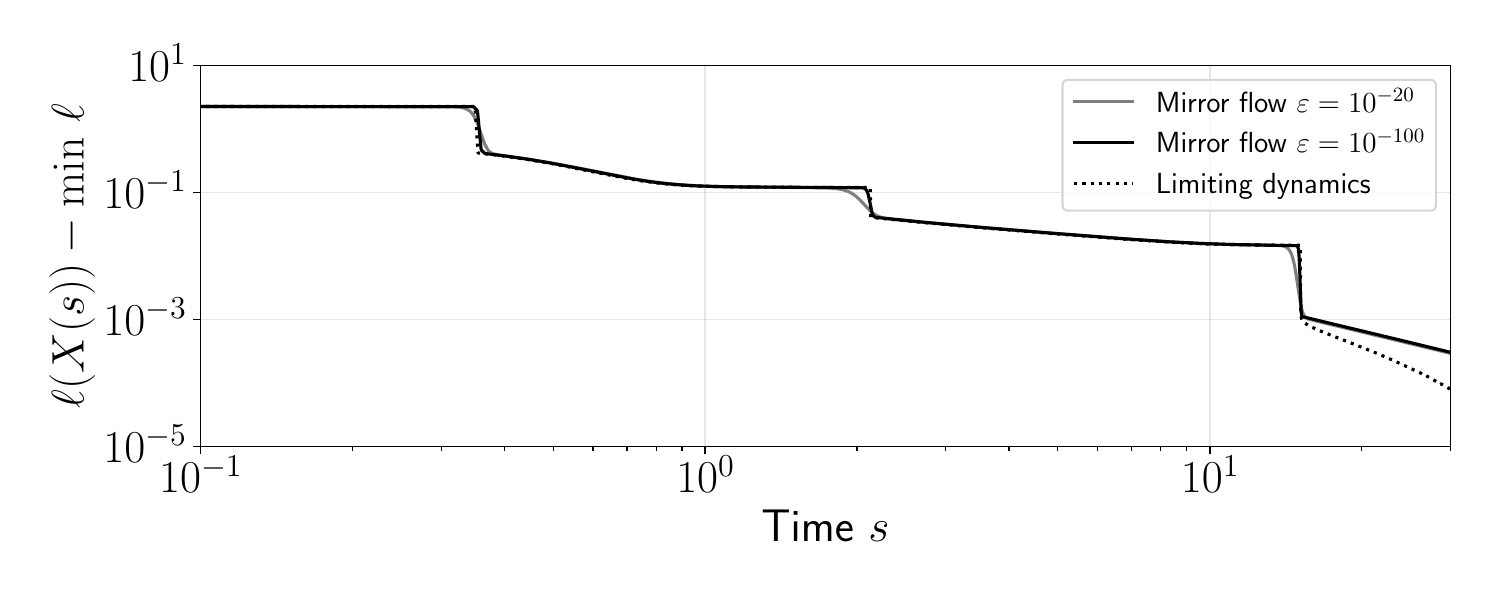}
	\caption{\textbf{Positive semidefinite cone.} Simulation of the mirror flow for two small values of $\varepsilon$ and comparison with the limiting dynamics. The top two plots show the eigenvalue trajectories of $X$ for the mirror flow (solid line) and for the limiting dynamics (dotted line). The top plot corresponds to $\varepsilon = 10^{-20}$ and the middle plot to $\varepsilon = 10^{-100}$. The bottom plot displays the excess loss. The mirror flow and the limiting dynamics agree well. Here, the eigenvalue trajectories are not piecewise constant because slow dynamics occur within each fixed-rank set of predictions. The loss is non-increasing but not piecewise constant, and it displays discontinuities when a new eigenvalue activates. This also illustrates the \textit{incremental learning} of the iterates.}
	\label{fig:MM_simulation}
\end{figure}

\subsubsection{Connection to the matrix factorization problem}

We have seen in Section~\ref{subsec:diagonal_nn} that the mirror flow on the non-negative orthant with the entropic mirror potential corresponds to the parametrization $x = u \circ u$ and a gradient flow in $u$. The matrix analogue does \emph{not} hold: the mirror flow on $\mathcal{S}^d_+$ with the von Neumann entropic mirror potential is \emph{not} equivalent to a gradient flow in $U \in \R^{d \times d}$ for the parametrization $X = U U^\top$. Both dynamics nevertheless share a striking qualitative feature: when initialized infinitesimally close to the boundary of $\mathcal{S}^d_+$, both exhibit \emph{incremental learning} through successive rank increases. We illustrate this empirically in Appendix~\ref{app:MMcomp}, where we run both flows side by side on the same instance of the matrix problem for four random seeds. While the qualitative incremental-rank picture is shared by both flows, the precise quantitative dynamics---the values of the intermediate saddles, the order in which eigenvalues activate, and the pacing of the activations---can differ between the two.

\subsection{Optimization on the probability simplex \texorpdfstring{$\Delta_d$}{}}
\label{subsec:simplex}

To illustrate the generality of our results, we also consider mirror flows on sets with affine constraints. Consider the same quadratic loss $\ell$ as in Section~\ref{subsec:diagonal}, but now on the probability simplex $\Delta_d = \R^d_{\geq 0} \cap \mathcal{S}_1$, where $\mathcal{S}_1 = \{x \in \R^d \mid \sum_{k=1}^d x_k = 1\}$. To reflect the structure of the problem, we optimize $\ell$ via the mirror flow with entropic mirror potential $h(x) = \sum_{k=1}^d (x_k \log x_k - x_k) + \iota_{\mathcal{S}_1}(x)$, with domain $\Delta_d$. For $x \in \ri \Delta_d$, we have $\partial h (x) = \log x + \R \mathbf{1}$, $\dom h^* = \R^d$, and $\nabla h^*(w) = e^w /\|e^w\|_1$; see Lemma~\ref{lem:mirror-potential-simplex} for the conjugate calculation.

\paragraph{Support function.} We have $C = \Delta_d$ and, for all $\overline{w} \in \R^d$, $\sigma_{\Delta_d}(\overline{w}) = \max_i \overline{w}_i$. Moreover, $\dom \sigma_{\Delta_d} = \dom \partial \sigma_{\Delta_d} = \R^d$. Hence requirements~\ref{it:main-1} and~\ref{it:main-2} of Theorem~\ref{thm:main} are automatically satisfied.
\paragraph{The limiting dynamics.} We initialize the mirror flow near the first canonical vector:
\[
	x_0^\varepsilon = \left(1-\varepsilon, \frac{\varepsilon}{d-1}, \ldots, \frac{\varepsilon}{d-1}\right).
\]
Equivalently,
\[
	w_0^\varepsilon =
	\left(\log(1-\varepsilon), \log\frac{\varepsilon}{d-1}, \ldots, \log\frac{\varepsilon}{d-1}\right),
\]
so $\mu^\varepsilon = \|w_0^\varepsilon\| \sim_{\varepsilon \to 0} \sqrt{d-1}\log(1/\varepsilon)$ and
\[
	\overline{w}_0
	= \lim_{\varepsilon \to 0} \frac{w_0^\varepsilon}{\mu^\varepsilon}
	= -\left(0, \frac{1}{\sqrt{d-1}}, \ldots, \frac{1}{\sqrt{d-1}}\right).
\]
Applying Theorem~\ref{thm:main}, the limiting dual flow reads:
\begin{equation*}
	\frac{\diff \overline{w}}{\diff s} (s) = - \nabla \ell(\overline{x}(s)) \, , \qquad \overline{x}(s) \mathrel{\in} \partial \sigma_C(\overline{w}(s)) \, .
\end{equation*}
\paragraph{Primal dynamics and sparsity.} For all $w \in \R^d$, we can define its active set of coordinates as $I(w) = \{k \mid w_k = \max_i w_i\}$. We have the following expression:
\begin{equation*}
	\partial \sigma_{\Delta_d}(w) = \left\{ \sum_{k \in I(w)} \alpha_k e_k \,\middle|\, \sum_{k \in I(w)} \alpha_k = 1,\, \alpha_k \geq 0 \right\} ,
\end{equation*}
which is the convex hull of $(e_k)_{k \in I(w)}$, where the $e_k$ are the canonical vectors. We prove this expression in Lemma~\ref{lem:proof-support-function-simplex}.
Therefore $\overline{x}(s)$ is the minimizer of $\ell$ restricted to coordinates in $I(\overline{w}(s))$ under the simplex constraint:
\begin{equation*}
	\overline{x}(s) \in \Argmin \left\{ \ell(y) \,\middle|\, y \in \Delta_d , \, \supp(y) \subset I(\overline{w}(s)) \right\} \, .
\end{equation*}
The trajectory $\overline{w}(s)$ is piecewise affine, and the active set $I(\overline{w}(s))$ is piecewise constant, with changes at the times when a new coordinate of $\overline{w}(s)$ reaches the maximum value of the vector. Consequently, the limiting primal trajectory $\overline{x}(s)$ can be chosen piecewise constant; in each phase it is the solution of a constrained quadratic optimization problem.

\paragraph{Implementation.} We use the same loss $\ell$ as in Section~\ref{subsec:diagonal}. The result is shown in Figure~\ref{fig:simplex_simulation}. The visualization of the dual trajectory of $\overline{w}$ is given in Appendix~\ref{app:simplex_simulation}. Further implementation details are given in Appendix~\ref{app:implementation}.

\begin{figure}[ht]
	\centering
	\includegraphics[width=0.94\textwidth]{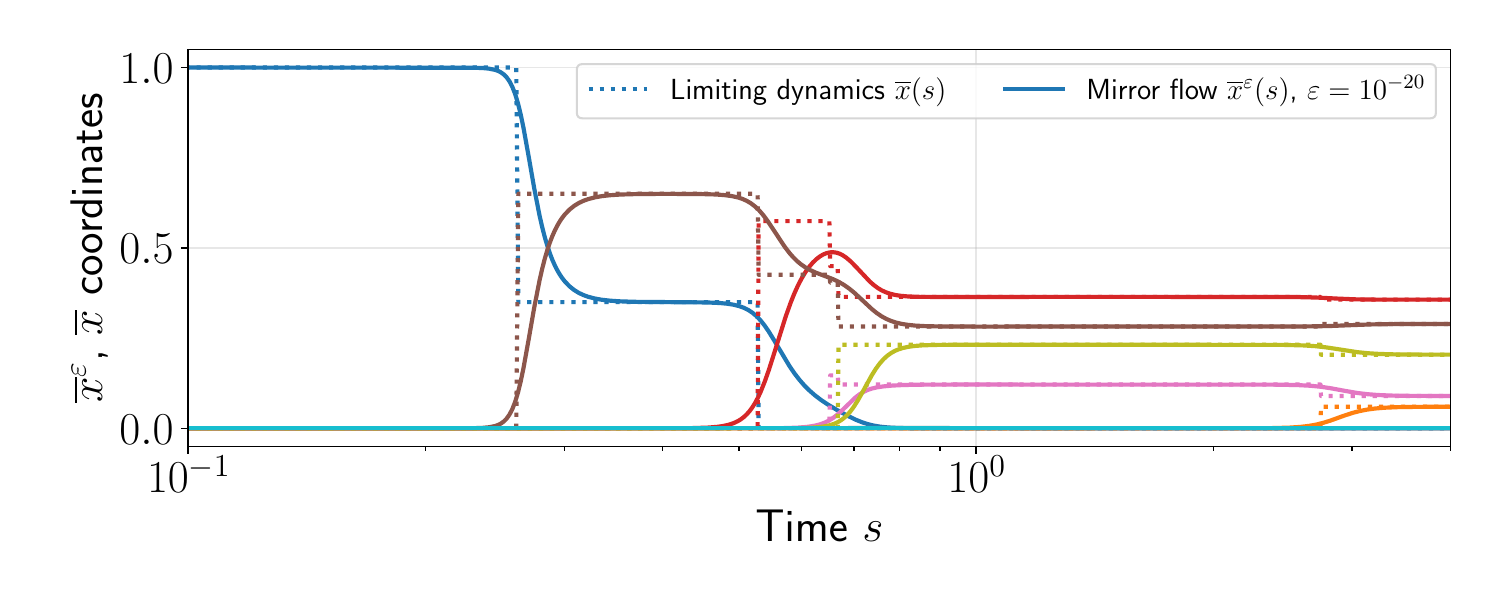}
	\includegraphics[width=0.94\textwidth]{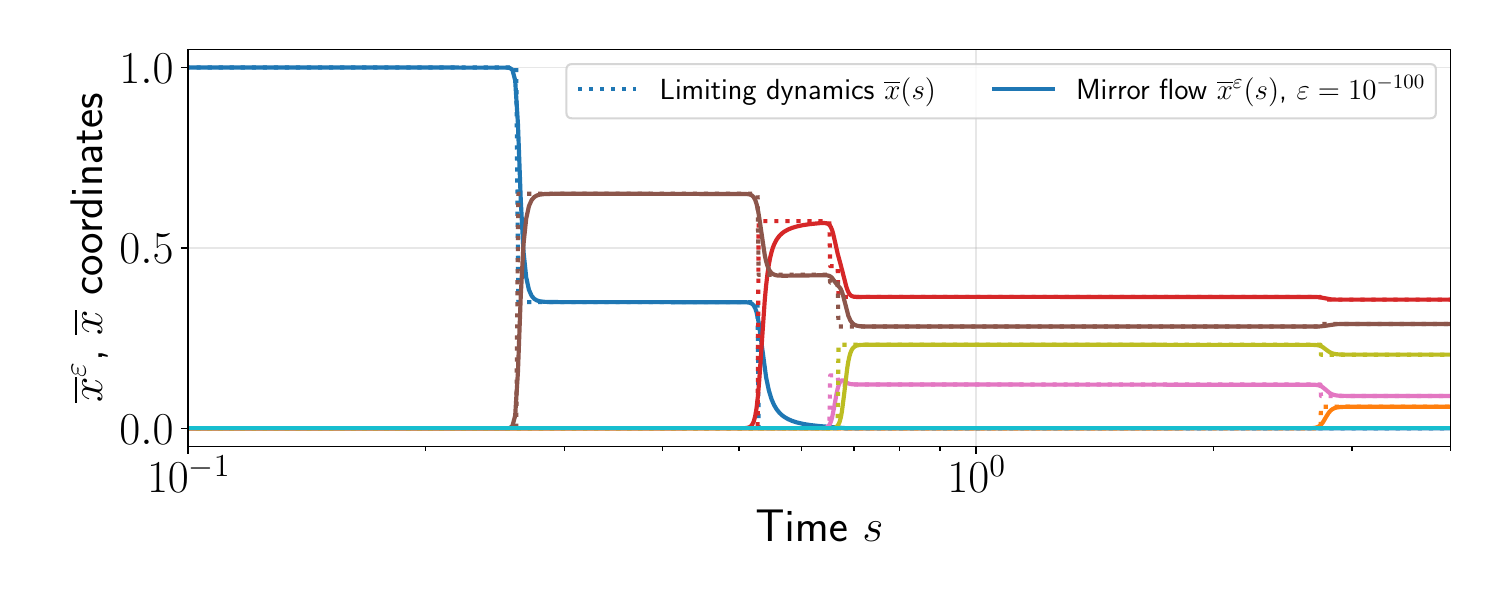}
	\includegraphics[width=0.94\textwidth]{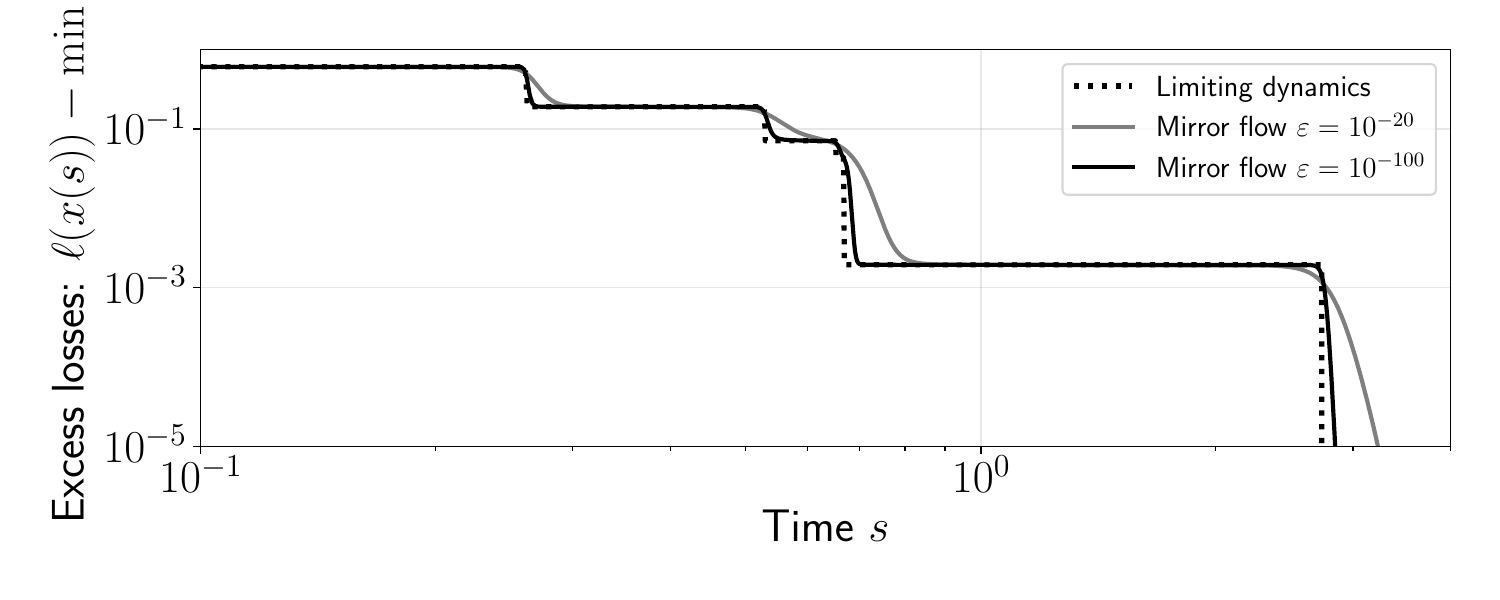}
	\caption{\textbf{Simplex.} Simulation of the mirror flow for two small values of $\varepsilon$ and comparison with the limiting dynamics. The top two plots show the coordinate trajectories of $x$ for the mirror flow (solid line) and for the limiting dynamics (dotted line). The top plot corresponds to $\varepsilon = 10^{-20}$ and the middle plot to $\varepsilon = 10^{-100}$. The bottom plot displays the excess loss, i.e. the loss minus its minimum over $\Delta_d$. The mirror flow and the limiting dynamics agree well. The trajectories exhibit complex sparse activation patterns, and the loss is non-increasing and piecewise constant, illustrating the \textit{incremental learning} of the iterates. The loss converges to the minimum of $\ell$ over the simplex.}
	\label{fig:simplex_simulation}
\end{figure}

\clearpage

\section{Proofs of the results}
\label{sec:proofs}

This section proves the results of Section~\ref{sec:incremental}. Section~\ref{sec:proof-equivalences-mf-gf} proves Proposition~\ref{prop:equivalences-mf-gf}. Section~\ref{sec:proof-existence-mirror-flow} proves Theorem~\ref{thm:existence-mirror-flow} after recalling a result of~\cite{brezis1973ope} on existence and uniqueness for subgradient flows in Section~\ref{subsec:existence}. Section~\ref{sec:proof-main} proves Theorem~\ref{thm:main} after recalling a result of~\cite{attouch1977convergence} on evolutionary convergence of subgradient flows.

\subsection{Proof of Proposition~\ref{prop:equivalences-mf-gf}}
\label{sec:proof-equivalences-mf-gf}

For the sake of future reference, we note the following lemma.

\begin{lem}
	\label{lem:subdiff}
	Let $a \in \ri \dom h^* + \Span M$. The subdifferential of 
	\begin{equation*}
		\varphi(u) = h^*(a + M^{1/2} u) - \langle (M^\dagger)^{1/2} q , u \rangle
	\end{equation*}
	is given by 
	\begin{equation*}
		\partial \varphi(u) = M^{1/2} \partial h^*(a + M^{1/2} u) - (M^\dagger)^{1/2} q \, .
	\end{equation*}
\end{lem}

\begin{proof}
	The proof follows from \cite[Thm.~16.47]{bauschke2017convex}. Note that the assumption $a \in \ri \dom h^* + \Span M$ is actually needed for these seemingly straightforward computations. 
\end{proof}

\noindent
We now prove Proposition~\ref{prop:equivalences-mf-gf}. Throughout the proof, we use two identities that follow from the assumption $q \in \Span M$ and from the spectral structure of $M$: $MM^\dagger q = q$ (since $MM^\dagger$ is the orthogonal projection onto $\Span M$) and $(M^\dagger)^{1/2}M = M^{1/2}$.

\noindent
\textbf{\ref{it:mf}$ \Rightarrow $\ref{it:prec-gf}.} From \ref{it:mf}, we have that for a.e.~$t \geq 0$, $w$ is differentiable at $t$ and
	\begin{align*}
		0 &= \frac{\diff w}{\diff t}(t) + \nabla \ell(x(t)) = \frac{\diff w}{\diff t}(t) + M x(t) - q \, .
	\end{align*}
	As $q \in \Span M$, $x(t) \in \partial h^*(w(t))$ and $\partial \Phi(w) = \partial h^*(w) - M^\dagger q$, we obtain 
	\begin{align*}
		0 &= \frac{\diff w}{\diff t}(t) + M (x(t) - M^\dagger q) \in \frac{\diff w}{\diff t}(t) + M \left(\partial h^*(w(t)) - M^\dagger q\right) = \frac{\diff w}{\diff t}(t) + M \partial \Phi(w(t)) \, .
	\end{align*}
	\textbf{\ref{it:prec-gf}$ \Rightarrow $\ref{it:gf}.}
	We use Lemma~\ref{lem:subdiff}. As $u(t) = (M^\dagger)^{1/2} (w(t) - w(0))$, for a.e.~$t \geq 0$, $u$ is differentiable at $t$ and
\begin{align*}
	\frac{\diff u}{\diff t}(t) + \partial \varphi(u(t)) &= (M^\dagger)^{1/2} \frac{\diff w}{\diff t}(t) + M^{1/2} \partial h^*(w(0) + M^{1/2} u(t)) - (M^\dagger)^{1/2} q \\
&= (M^\dagger)^{1/2} \left[\frac{\diff w}{\diff t}(t) + M \partial h^*(w(t))-q\right] = (M^\dagger)^{1/2} \left[\frac{\diff w}{\diff t}(t) + M \partial \Phi(w(t))\right] \\
&\ni (M^\dagger)^{1/2} 0 = 0 \, .
\end{align*}
\textbf{\ref{it:gf}$ \Rightarrow $\ref{it:mf}.} Again, we use Lemma~\ref{lem:subdiff}. As $w(t) = w(0) + M^{1/2} u(t)$, for a.e.~$t \geq 0$, $w$ is differentiable at $t$ and
\begin{align*}
	\frac{\diff w}{\diff t}(t) &= M^{1/2} \frac{\diff u}{\diff t}(t) \in - M^{1/2} \left[M^{1/2} \partial h^*(w(t)) - (M^\dagger)^{1/2}q\right] = - \left[M \partial h^*(w(t)) - q\right] \, . 
\end{align*}
Thus there exists $x(t) \in \partial h^*(w(t))$ such that \begin{align*}
	\frac{\diff w}{\diff t}(t) &= - \left[M x(t) - q\right] = - \nabla \ell(x(t)) \, .
\end{align*}

\subsection{\texorpdfstring{Proof of Theorem~\ref{thm:existence-mirror-flow}}{Proof of Theorem (existence of mirror flows)}} 
\label{sec:proof-existence-mirror-flow}

We first introduce the notion of a subgradient flow associated with a function in $\Gamma_0(\R^d)$ and recall an existence and uniqueness result for such flows. We then use this result to prove Theorem~\ref{thm:existence-mirror-flow}.

\addtocontents{toc}{\protect\setcounter{tocdepth}{2}}
\subsubsection{Existence and uniqueness of subgradient flows}
\label{subsec:existence}

The existence and uniqueness result for subgradient flows that we use here was proved in~\cite{brezis1973ope}, which studies more generally the evolution equations associated with maximal monotone operators. This includes our setting: subgradients of convex, proper, lower semicontinuous functions are maximal monotone operators~\cite[Ex.~2.3.4]{brezis1973ope}.

\begin{definition}
	\label{def:sol-subgradient-flow}
Let $f \in \Gamma_0(\R^d)$ and $u:\R_{\geq 0} \to \R^d$. We say that $u$ is a solution of the subgradient flow associated with $f$ if, for almost all $t > 0$, $u$ is differentiable at $t$ and 
\begin{equation*}
	\frac{\diff u}{\diff t}(t) + \partial f(u(t)) \ni 0 \, .
\end{equation*}
\end{definition}

For all $u \in \R^d$, $\partial f(u)$ is a closed convex set~\cite[p.~28]{brezis1973ope}. When it is nonempty, we denote by $(\partial f)^\circ(u) = \argmin_{z \in \partial f(u)} \Vert z \Vert^2$ its element of minimal norm.

\begin{thm}[{\cite[Thm.~3.1]{brezis1973ope}}]
	\label{thm:brezis}
		Let $f \in \Gamma_0(\R^d)$ and $u_0 \in {\dom \partial f}$. There exists a unique Lipschitz solution $u$ of the subgradient flow associated with $f$ such that $u(0) = u_0$.
	Moreover, for a.e.~$t \geq 0$, $u$ is differentiable at $t$ and
	\begin{equation}
			\label{eq:pointwise-subgradient-flow}
			\frac{\diff u}{\diff t}(t) + (\partial f)^\circ(u(t)) = 0 \, .
		\end{equation}
\end{thm}

\subsubsection{Proof of Theorem~\ref{thm:existence-mirror-flow}}
\label{subsec:proof-existence-mirror-flow}

\paragraph{Uniqueness.} Let $w_1$ and $w_2$ be two Lipschitz solutions of the mirror flow with mirror potential $h$ and loss $\ell$ such that $w_1(0) = w_2(0) = w_0$. By Proposition~\ref{prop:equivalences-mf-gf}, as $w_0 \in \ri \dom h^* + \Span M$, $u_1(t) = (M^\dagger)^{1/2} (w_1(t) - w_0)$ and $u_2(t) = (M^\dagger)^{1/2} (w_2(t) - w_0)$ are two Lipschitz solutions of the subgradient flow associated with $\varphi(u) = h^*(w_0 + M^{1/2} u) - \langle (M^\dagger)^{1/2} q , u \rangle$ such that $u_1(0) = u_2(0) = 0$. By Theorem~\ref{thm:brezis}, $u_1 = u_2$, hence $w_1 = w_2$.

\paragraph{Existence.} Define $\varphi(u) = h^*(w_0 + M^{1/2} u) - \langle (M^\dagger)^{1/2} q , u \rangle$. As $w_0 \in \ri \dom h^* + \Span M$, Lemma~\ref{lem:subdiff} gives $\partial \varphi(u) = M^{1/2} \partial h^*(w_0 + M^{1/2} u) - (M^\dagger)^{1/2} q$. Since $w_0 \in \dom \partial h^*$, $0 \in \dom \partial \varphi$. By Theorem~\ref{thm:brezis}, there exists a unique Lipschitz solution $u$ of the subgradient flow associated with $\varphi$ such that $u(0) = 0$. Define $w(t) = w_0 + M^{1/2} u(t)$. By Proposition~\ref{prop:equivalences-mf-gf}, as $w_0 \in \ri \dom h^* + \Span M$, $w$ is a Lipschitz solution of the mirror flow with mirror potential $h$ and loss $\ell$ such that $w(0) = w_0$.

\paragraph{Minimization property.} By definition of the mirror flow, for a.e.~$t \geq 0$, $x(t) \in \partial h^*(w(t))$. We therefore only need to show that, for a.e.~$t \geq 0$ and all $y \in \partial h^*(w(t))$, $\ell(y) \geq \ell(x(t))$. This minimization property follows from Theorem~\ref{thm:brezis}, after a change of variables.

By Theorem~\ref{thm:brezis}, for a.e.~$t \geq 0$, $u$ is differentiable at $t$ and 
\begin{align*}
	&\frac{\diff u}{\diff t}(t) + (\partial \varphi)^\circ(u(t)) = 0 \, , &&(\partial \varphi)^\circ(u) = \underset{z \in \partial \varphi(u)}\argmin \Vert z \Vert^2 \, .
\end{align*}
Recall the formula for $\partial \varphi$ given by Lemma~\ref{lem:subdiff}. Consider $y \in \partial h^*(w(t))$ and define $z = M^{1/2} y - (M^\dagger)^{1/2}q \in \partial \varphi(u(t))$. We have the following equalities:
\begin{equation}
	\label{eq:aux-2}
	\Vert z \Vert^2 = \langle y , My \rangle - 2 \langle y, q \rangle + \langle q, M^\dagger q \rangle = 2 \ell(y) -2c + \langle q, M^\dagger q \rangle \, .
\end{equation}
Note that in the case where $y = x(t)$ (which is in $\partial h^*(w(t))$), we have 
\begin{align}
	\label{eq:aux-3}
	\begin{split}
	z &= M^{1/2} x(t) - (M^\dagger)^{1/2}q = (M^\dagger)^{1/2} \left(Mx(t)-q\right) = (M^\dagger)^{1/2} \nabla \ell(x(t)) \\
	&= - (M^\dagger)^{1/2} \frac{\diff w}{\diff t}(t)
	= - \frac{\diff u}{\diff t}(t) = (\partial \varphi)^\circ(u(t)) \, .
	\end{split}
\end{align}
These elements allow to conclude. For a.e.~$t \geq 0$, for all $y \in \partial h^*(w(t))$, we have
\begin{align*}
	2 \ell(y) -2c + \langle q, M^\dagger q \rangle &= \Vert z \Vert^2 &&\text{by Eq.~\eqref{eq:aux-2}} \, , \\
	&\geq \Vert (\partial \varphi)^\circ(u(t)) \Vert^2 &&\text{as $z \in \partial \varphi(u(t))$} \, ,  \\
	&= 2 \ell(x(t)) -2c + \langle q, M^\dagger q \rangle &&\text{by Eqs.~\eqref{eq:aux-3}, \eqref{eq:aux-2}} \, .
\end{align*}
Thus for all $y \in \partial h^*(w(t))$, $\ell(y) \geq \ell(x(t))$.

\subsection{Proof of Theorem~\ref{thm:main}}
\label{sec:proof-main}

As mentioned in the introduction, the article \cite{attouch2004singular} provides convergence results for mirror flows. However, they require the mirror potential to be uniformly strongly convex, which does not encompass our cases of interest. Thus we follow another strategy, based on a result of \cite{attouch1977convergence} on evolutionary convergence of subgradient flows under Mosco convergence of the associated functions. We first recall this result, and then prove Theorem~\ref{thm:main}.

\subsubsection{Evolutionary convergence of subgradient flows}
\label{subsec:evolutionary}

We first introduce Mosco convergence, which is the variational convergence notion behind the evolutionary convergence of subgradient flows.

\begin{definition}
	\label{def:mosco}
		Let $f^{\varepsilon}$, $\varepsilon > 0$, and $f \in \Gamma_0(\R^d)$. We say that $f^{\varepsilon}$ \emph{Mosco converges} to $f$ as $\varepsilon \to 0$, and write $ \displaystyle f^{\varepsilon} \xrightarrow[\varepsilon \to 0]{\mathsf{M}} f $, if
	\begin{itemize}
		\item ($\liminf$ inequality) for all $u \in \R^d$ and all $(u^{\varepsilon})_{\varepsilon > 0}$ converging to $u$ in $\R^d$ as $\varepsilon \to 0$,
		\begin{equation*}
			f(u) \leq \liminf_{\varepsilon \to 0} f^{\varepsilon}(u^{\varepsilon}) \, ,
		\end{equation*}
		\item ($\limsup$ inequality) for all $u \in \R^d$, there exists $(u^{\varepsilon})_{\varepsilon > 0}$ converging to $u$ in $\R^d$ as $\varepsilon \to 0$ such that
		\begin{equation*}
			f(u) \geq \limsup_{\varepsilon \to 0} f^{\varepsilon}(u^{\varepsilon}) \, .
		\end{equation*}
	\end{itemize}
\end{definition}

\begin{example}
	\label{ex:mosco}
	Recall that $C = \dom h$ is closed. For $\varepsilon > 0$, define $f^{\varepsilon} = {\varepsilon} h$. Then $f^{\varepsilon} \xrightarrow[\varepsilon \to 0]{\mathsf{M}} \iota_C$.

	Indeed, we check both inequalities.

	\begin{itemize}
		\item ($\liminf$ inequality) Let $u \in \R^d$ and let $(u^{\varepsilon})_{\varepsilon > 0}$ converge to $u$ in $\R^d$ as $\varepsilon \to 0$. If $u \notin C$, then $\iota_C(u) = +\infty$ and, since $C$ is closed, $u^\varepsilon \notin C$ for $\varepsilon > 0$ small enough. Thus $f^{\varepsilon}(u^\varepsilon) = +\infty = \iota_C(u)$ for $\varepsilon > 0$ small enough, and the $\liminf$ inequality holds. If $u \in C$, then $\iota_C(u) = 0$ and, since the convex function $h$ is locally bounded below~\cite[Coro.~16.18]{bauschke2017convex}, $\liminf_{\varepsilon \to 0} f^{\varepsilon}(u^\varepsilon) \geq 0 = \iota_C(u)$.
		\item ($\limsup$ inequality) Let $u \in \R^d$. We choose $u^\varepsilon = u$ for all $\varepsilon > 0$. Independently of whether $u$ is in $C$ or not, we have $\lim_{\varepsilon \to 0} f^{\varepsilon}(u^\varepsilon) = \lim_{\varepsilon \to 0} \varepsilon h(u) = \iota_C(u)$ and the $\limsup$ inequality holds. 
	\end{itemize}
\end{example}

We now recall the evolutionary convergence result under Mosco convergence of the associated functions.

\begin{thm}[{\cite{attouch1977convergence}, \cite[Thm.~3.74]{attouch1984variational}}]
	\label{thm:gamma_convergence}
	Let $f^{\varepsilon}$, $\varepsilon > 0$, and $f \in \Gamma_0(\R^d)$. Let $u^{\varepsilon}_0 \in {\dom \partial f^{\varepsilon}}$, $\varepsilon > 0$, and $u_0 \in {\dom \partial f}$. For all $\varepsilon > 0$, let $u^{\varepsilon}$ denote the unique Lipschitz solution of the subgradient flow associated with $f^{\varepsilon}$ such that $u^\varepsilon(0) = u_0^{\varepsilon}$. Similarly, let $u$ denote the unique Lipschitz solution of the subgradient flow associated with $f$ such that $u(0) = u_0$.

	Assume that $\displaystyle f^{\varepsilon} \xrightarrow[\varepsilon \to 0]{\mathsf{M}} f$ and  $u^{\varepsilon}_0 \xrightarrow[\varepsilon\to 0]{} u_0$ in $\R^d$. Then for all $T \geq 0$,  
	\begin{itemize}
		\item $u^{\varepsilon}$ converges to $u$ uniformly on $[0,T]$ and
		\item $\displaystyle \frac{\diff u^{\varepsilon}}{\diff t}$ converges to $\displaystyle \frac{\diff u}{\diff t}$ in $L^2([0,T], \R^d, t \diff t)$, that is,
		\[
			\int_0^T t \left\| \frac{\diff u^{\varepsilon}}{\diff t}(t) - \frac{\diff u}{\diff t}(t) \right\|^2 \diff t \xrightarrow[\varepsilon \to 0]{} 0.
		\]
	\end{itemize} 
\end{thm}

Finally, we recall that Fenchel conjugation is bicontinuous for Mosco convergence.

\begin{proposition}[{\cite[Thm.~3.18]{attouch1984variational}}]
	\label{prop:mosco-fenchel}
	Let $f^{\varepsilon}$, $\varepsilon > 0$, and $f\in\Gamma_0(\R^d)$. Then 
	\begin{equation*}
	f^{\varepsilon} \xrightarrow[\varepsilon \to 0]{\mathsf{M}} f \qquad \Longleftrightarrow \qquad (f^{\varepsilon})^* \xrightarrow[\varepsilon \to 0]{\mathsf{M}} f^* \, .
	\end{equation*} 
\end{proposition}

\begin{example}
	\label{ex:mosco-conjugate}
	Combining Example~\ref{ex:mosco} and the above proposition, we have $(\varepsilon h)^* \xrightarrow[\varepsilon \to 0]{\mathsf{M}} \iota_C^* = \sigma_C$.
\end{example}

\subsubsection{Proof of Theorem~\ref{thm:main}}
\label{subsec:proof-main}

The proof strategy is to reduce the rescaled mirror flows to subgradient flows and then apply the evolutionary convergence result of the previous section.

Recall that $w^\varepsilon$ is a solution of the mirror flow with mirror potential $h$ and loss $\ell$, with $x^\varepsilon$ an associated primal variable. Thus, for a.e.~$t \geq 0$, $w^\varepsilon$ is differentiable at $t$ and
\begin{align*}
	&\frac{\diff w^\varepsilon}{\diff t}(t) = - \nabla \ell(x^\varepsilon(t)) \, , &&w^\varepsilon(t) \in \partial h(x^\varepsilon(t)) \, .
\end{align*}
We now compute the induced mirror flow in the rescaled variables 
\begin{align*}
	&\overline{w}^\varepsilon(s) = \frac{1}{\mu^\varepsilon} w^\varepsilon(\mu^\varepsilon s)  \, , &&\overline{x}^\varepsilon(s) = x^\varepsilon(s \mu^\varepsilon) \, .
\end{align*}
We obtain that, for a.e.~$s \geq 0$, $\overline{w}^\varepsilon$ is differentiable at $s$ and
\begin{align}
	&\frac{\diff \overline{w}^\varepsilon}{\diff s}(s) = \frac{\mu^\varepsilon}{\mu^\varepsilon} \frac{\diff w^\varepsilon}{\diff t}(\mu^\varepsilon s) = - \nabla \ell({x}^\varepsilon(\mu^\varepsilon s)) = -\nabla \ell(\overline{x}^\varepsilon(s)) \, , \nonumber \\
	&\overline{w}^\varepsilon(s) = \frac{1}{\mu^\varepsilon} w^\varepsilon(\mu^\varepsilon s) \in \frac{1}{\mu^\varepsilon} \partial h(x^\varepsilon(\mu^\varepsilon s))  = \partial \left( \frac{1}{\mu^\varepsilon} h \right) (\overline{x}^\varepsilon(s)) \, . \label{eq:aux-4}
\end{align}
Thus $\overline{w}^\varepsilon$ is a solution of the mirror flow with mirror potential $h^\varepsilon := \frac{1}{\mu^\varepsilon} h$ and loss $\ell$, with $\overline{x}^\varepsilon$ as an associated primal variable. By assumption, the dual initialization $\overline{w}^\varepsilon(0) = \overline{w}_0^\varepsilon$ converges to a finite limit $\overline{w}_0$ as $\varepsilon \to 0$. Thus rescaling removes the degeneracy of the initialization, at the price of scaling the mirror potential.

We now transform these mirror flows into subgradient flows to apply evolutionary convergence results. By Proposition~\ref{prop:equivalences-mf-gf}, $u^\varepsilon = (M^\dagger)^{1/2} (\overline{w}^\varepsilon - \overline{w}_0^\varepsilon)$ is the unique Lipschitz solution of the subgradient flow associated with
\begin{equation*}
	\varphi^\varepsilon(u) = \left(h^\varepsilon\right)^*(\overline{w}_0^\varepsilon + M^{1/2} u) - \left\langle (M^\dagger)^{1/2} q ,  u \right\rangle \, ,
\end{equation*}
and $u^\varepsilon(0) = 0$. To analyze the evolutionary convergence, we thus need to analyze the Mosco convergence of $\varphi^\varepsilon$ as $\varepsilon \to 0$. 

By Example~\ref{ex:mosco-conjugate}, $(h^\varepsilon)^* \xrightarrow[\varepsilon \to 0]{\mathsf{M}} \sigma_C$. Since simultaneously $\overline{w}_0^\varepsilon \xrightarrow[\varepsilon \to 0]{} \overline{w}_0$, this suggests that $\varphi^\varepsilon \xrightarrow[\varepsilon\to 0]{\mathsf{M}} \varphi$, where
\begin{equation*}
	\varphi(u) = \sigma_C (M^{1/2} u + \overline{w}_0) - \langle {M^{\dagger}}^{1/2} q, u \rangle \, .
\end{equation*}
This statement is not automatic and requires the assumptions of Theorem~\ref{thm:main}. We isolate it in the lemma below and postpone its proof to Section~\ref{sec:proof-mosco_recessions}.

\begin{lem}
	\label{lem:mosco_recessions}
	Under the assumptions of Theorem~\ref{thm:main}, $\varphi^\varepsilon \xrightarrow[\varepsilon \to 0]{\mathsf{M}} \varphi$.
\end{lem}

Using this lemma, we apply Theorem~\ref{thm:gamma_convergence}. Let $u = (M^\dagger)^{1/2} (\overline{w} - \overline{w}_0)$. By Proposition~\ref{prop:equivalences-mf-gf}, $u$ is the unique Lipschitz solution of the subgradient flow associated with $\varphi$ such that $u(0) = 0$. Fix $S > 0$. By Theorem~\ref{thm:gamma_convergence}, $u^\varepsilon$ converges to $u$ uniformly on $[0,S]$ and $\frac{\diff u^\varepsilon}{\diff s}$ converges to $\frac{\diff u}{\diff s}$ in $L^2([0,S], \R^d, s \diff s)$. We now use these convergences to prove the statements of the theorem.

\begin{enumerate}[label=(\alph*)]
	\item We have $\overline{w}^\varepsilon = \overline{w}_0^\varepsilon + M^{1/2} u^\varepsilon$ and $\overline{w} = \overline{w}_0 + M^{1/2} u$, where $\overline{w}_0^\varepsilon \xrightarrow[\varepsilon\to 0]{} \overline{w}_0$ and $u^\varepsilon \xrightarrow[\varepsilon\to 0]{} u$ uniformly on $[0,S]$. Thus $\overline{w}^\varepsilon$ converges to $\overline{w}$ uniformly on $[0,S]$. 
	\item 
	We compute for a.e.~$s \in [0,S]$,
	\begin{align*}
		\frac{\diff u^\varepsilon}{\diff s}(s) = (M^\dagger)^{1/2} \frac{\diff \overline{w}^\varepsilon}{\diff s}(s) = - (M^\dagger)^{1/2} \nabla \ell(\overline{x}^\varepsilon(s)) \, ,
	\end{align*}
	and similarly
	\begin{align*}
		\frac{\diff u}{\diff s}(s) = (M^\dagger)^{1/2} \frac{\diff \overline{w}}{\diff s}(s) = - (M^\dagger)^{1/2} \nabla \ell(\overline{x}(s)) \, .
	\end{align*}
	As $\displaystyle \frac{\diff u^\varepsilon}{\diff s} \xrightarrow[\varepsilon\to 0]{}\frac{\diff u}{\diff s}$ in $L^2([0,S], \R^d, s \diff s)$, and for all $x \in \R^d$, $\nabla \ell(x) \in \Span M$, we have $\displaystyle \nabla \ell(\overline{x}^\varepsilon(s)) \xrightarrow[\varepsilon\to 0]{} \nabla \ell(\overline{x}(s))$ in $L^2([0,S], \R^d, s \diff s)$.

	Moreover, 
	\begin{align*}
		\left\Vert \frac{\diff u^\varepsilon}{\diff s}(s) \right\Vert^2 &= \left\Vert (M^\dagger)^{1/2} \nabla \ell(\overline{x}^\varepsilon(s)) \right\Vert^2 = \left\Vert M^{1/2} \overline{x}^\varepsilon(s) - (M^{\dagger})^{1/2} q \right\Vert^2 \\
		&= 2 \ell(\overline{x}^\varepsilon(s)) -2c + \langle q, M^\dagger q \rangle \, .
	\end{align*}
	Similarly,
	\begin{align*}
		\left\Vert \frac{\diff u}{\diff s}(s) \right\Vert^2 &= 2 \ell(\overline{x}(s)) -2c + \langle q, M^\dagger q \rangle \, .
	\end{align*}
	We thus obtain that $\ell(\overline{x}^\varepsilon(s))$ converges to $\ell(\overline{x}(s))$ in $L^1([0,S], \R, s \diff s)$. Finally, by Theorem~\ref{thm:existence-mirror-flow}, $\ell(\overline{x}(s)) = \min_{y \in \partial \sigma_C(\overline{w}(s))} \ell(y)$ for a.e.~$s \in [0,S]$. This concludes the proof of the second statement.
	\item We now assume that there exists a subsequence $\varepsilon_k \xrightarrow[k\to\infty]{} 0$ and an accumulation point $s \in [0,S] \mapsto \overline{y}(s)$ such that for a.e.~$s \in [0,S]$, $\overline{x}^{\varepsilon_k}(s) \xrightarrow[k\to\infty]{} \overline{y}(s)$. We need to show that for a.e.~$s \in [0,S]$, $\overline{y}(s) \in \underset{y \in \partial \sigma_C(\overline{w}(s))}\Argmin \ell(y)$. 
	
	By statement~\ref{it:main-statement-2}, $\ell(\overline{x}^{\varepsilon_k}(s)) \xrightarrow[k\to\infty]{} \min_{y \in \partial \sigma_C(\overline{w}(s))} \ell(y)$ in $L^1([0,S], \R, s \diff s)$. Moreover, by continuity of $\ell$, $\ell(\overline{x}^{\varepsilon_k}(s)) \xrightarrow[k\to\infty]{} \ell(\overline{y}(s))$ a.e. Thus we must have
	\begin{equation*}
		\ell(\overline{y}(s)) = \min_{y \in \partial \sigma_C(\overline{w}(s))} \ell(y)
	\end{equation*}
	for a.e. $s \in [0,S]$. Indeed, the convergence in $L^1([0,S], \R, s \diff s)$ and a.e.~convergence both imply convergence in measure, for which the limit is essentially unique. 

	We are left with proving that $\overline{y}(s) \in \partial \sigma_C(\overline{w}(s))$ for a.e.~$s \in [0,S]$. Fix $s \in [0,S]$ such that $\overline{x}^{\varepsilon_k}(s) \xrightarrow[k\to\infty]{} \overline{y}(s)$. Recall that, by Eq.~\eqref{eq:aux-4}:
\begin{equation}
	\label{eq:aux-5}
	\overline{w}^{\varepsilon_k}(s) \in \partial \left(\frac{1}{\mu^{\varepsilon_k}} h \right) (\overline{x}^{\varepsilon_k}(s)) \, .
\end{equation}
From statement~\ref{it:main-statement-1}, $\overline{w}^{\varepsilon_k}(s) \xrightarrow[k\to\infty]{} \overline{w}(s)$. Moreover, from Example~\ref{ex:mosco}, $\frac{1}{\mu^{\varepsilon_k}} h \xrightarrow[k\to\infty]{\mathsf{M}} \iota_C$. This implies that $\partial \left(\frac{1}{\mu^{\varepsilon_k}} h \right) \xrightarrow[k\to\infty]{} \partial \iota_C$ in the resolvent sense~\cite{attouch1977convergence}, also called the graph sense~\cite[Def.~3.58]{attouch1984variational}. Then~\cite[statement (1.7)]{attouch1977convergence} (or \cite[Prop.~3.59]{attouch1984variational}) implies that we can take the limit in Eq.~\eqref{eq:aux-5}. This gives $\overline{w}(s) \in \partial \iota_C(\overline{y}(s))$, that is $\overline{y}(s) \in \partial \sigma_C(\overline{w}(s))$. This concludes the proof of the third statement.
\end{enumerate}

\section*{Acknowledgments}

The authors thank Scott Pesme for pointing out the reference \cite{attouch2004singular}.

\section*{Statements and Declarations}

\paragraph*{Funding.} RB acknowledges support from the ANR and the Ministère de l’Enseignement Supérieur et de la Recherche. This work is supported by Hi!~PARIS and the ANR/France 2030 program (ANR-23-IACL-0005).

\paragraph*{Competing interests.} The authors have no competing interests to declare that are relevant to the content of this article.

\paragraph*{Data availability.} No datasets were generated or analyzed in this work; the code producing the numerical experiments is available from the corresponding author upon request.

\addcontentsline{toc}{section}{References}
\bibliographystyle{alpha}
\bibliography{bibliography}

\addtocontents{toc}{\protect\setcounter{tocdepth}{0}}
\clearpage

\appendix

\begin{center}
	\textsc{\Huge Appendix}	
\end{center}

\vspace*{1cm}

\section{Proof of Lemma~\ref{lem:mosco_recessions}}
\label{sec:proof-mosco_recessions}

We start the proof by recalling two properties of the Fenchel conjugate that we will use.

\begin{proposition}[{\cite[Prop.~13.23(iii)]{bauschke2017convex}}]
	\label{prop:translation-conjugate}
If $f: \R^d \to \R \cup \{ +\infty \}$ and $a \in \R^d$, denote by $\tau_a f: \R^d \to \R \cup \{ +\infty \}$ the translated function defined by $(\tau_a f)(w) = f(w - a)$, $w \in \R^d$. Then, for all $x \in \R^d$, we have $(\tau_a f)^*(x) = f^*(x) + \langle x, a \rangle$.
\end{proposition}

\begin{proposition}[{\cite[Coro.~15.28]{bauschke2017convex}}]
\label{prop:conjugate_composition}
Let $f\in \Gamma_0(\R^d)$ and let $L: \R^d \to \R^d$ be a linear operator. Assume that $0 \in \ri(\dom f + \Span L)$. Then, for all $x \in \R^d$,
\begin{equation*}
	(f \circ L)^*(x) = \inf_{y \in \R^d, \,  L^\top y = x} f^*(y)   \, .
\end{equation*}
where $L^\top$ is the adjoint of $L$. Moreover, if $x \in \dom (f \circ L)^*$, then the infimum is attained.
\end{proposition}

The challenge in the proof of Lemma~\ref{lem:mosco_recessions} is to handle the joint convergence of $(h^\varepsilon)^*$ to $\sigma_C$ and of $\overline{w}_0^\varepsilon$ to $\overline{w}_0$. We first isolate the second difficulty.

\begin{lem}
	\label{lem:mosco_recessions_rec}
	Let $g \in \Gamma_0(\R^d)$ and assume that $\overline{w}_0^\varepsilon \xrightarrow[\varepsilon\to 0]{} \overline{w}_0$, that $\overline{w}_0 \in \ri \dom g + \Span M$, and that, for $\varepsilon > 0$ small enough, $\overline{w}_0^\varepsilon - \overline{w}_0$ belongs to the subspace parallel to $\aff \dom g$. Define
	\begin{align*}
		&\overline{\Psi}^\varepsilon(u) = g(M^{1/2} u + \overline{w}_0^\varepsilon) \, , && \Psi(u) = g(M^{1/2} u + \overline{w}_0) \, .
	\end{align*}
	Then $\overline{\Psi}^\varepsilon \xrightarrow[\varepsilon \to 0]{\mathsf{M}} \Psi$.
\end{lem}

\begin{proof}
	Following Definition~\ref{def:mosco}, we show the $\liminf$ and $\limsup$ inequalities.
	\begin{itemize}
		\item ($\liminf$ inequality) Consider a sequence $z_\varepsilon \xrightarrow[\varepsilon\to 0]{} z$. Then $M^{1/2} z_\varepsilon + \overline{w}_0^\varepsilon \xrightarrow[\varepsilon\to 0]{} M^{1/2} z + \overline{w}_0$. Since $g$ is lower semicontinuous, we have
		\begin{equation*}
			\liminf_{\varepsilon \to 0} \overline{\Psi}^\varepsilon(z_\varepsilon) = \liminf_{\varepsilon \to 0} g(M^{1/2} z_\varepsilon + \overline{w}_0^\varepsilon) \geq g(M^{1/2} z + \overline{w}_0) = \Psi(z) \, .
		\end{equation*}
		\item ($\limsup$ inequality)
		Let $z \in \R^d$. We seek $z_\varepsilon \xrightarrow[\varepsilon\to0]{} z$ such that
		\begin{equation*}
			\limsup_{\varepsilon \to 0} g(M^{1/2} z_\varepsilon + \overline{w}_0^\varepsilon) \leq g(M^{1/2} z + \overline{w}_0) \, .
		\end{equation*}
		Denote $u = M^{1/2} z + \overline{w}_0$. As $\overline{w}_0 \in \ri \dom g + \Span M$, there exists $\overline{z} \in \R^d$ such that $\overline{u} := M^{1/2} \overline{z} + \overline{w}_0 \in \ri \dom g$. Consider $t_\varepsilon \geq 0$ such that $t_\varepsilon \xrightarrow[\varepsilon\to 0]{} 0$ and $\overline{w}_0^\varepsilon - \overline{w}_0 = o(t_\varepsilon)$ as $\varepsilon \to 0$. For instance, one can take $t_\varepsilon = \sqrt{\|\overline{w}_0^\varepsilon - \overline{w}_0\|}$. We then consider the recovery sequence $z_\varepsilon = (1-t_\varepsilon) z + t_\varepsilon \overline{z}$ and obtain
		\begin{align*}
			M^{1/2} z_\varepsilon + \overline{w}_0^\varepsilon &= (1-t_\varepsilon) (M^{1/2} z + \overline{w}_0) + t_\varepsilon (M^{1/2} \overline{z} + \overline{w}_0) + \overline{w}_0^\varepsilon - \overline{w}_0 \\
			&= (1-t_\varepsilon) u + t_\varepsilon \widetilde{u}_\varepsilon \, ,
		\end{align*} 
		where $\widetilde{u}_\varepsilon = \overline{u} + \frac{1}{t_\varepsilon}(\overline{w}_0^\varepsilon - \overline{w}_0)$. 
		Note that $\overline{u} \in \ri \dom g$, $\frac{1}{t_\varepsilon} (\overline{w}_0^\varepsilon - \overline{w}_0) \xrightarrow[\varepsilon\to 0]{} 0$ and for $\varepsilon > 0$ small enough, $\overline{w}_0^\varepsilon - \overline{w}_0$ belongs to the subspace parallel to $\aff \dom g$, hence $\widetilde{u}_\varepsilon \in \ri \dom g$ for $\varepsilon > 0$ small enough, and $\widetilde{u}_\varepsilon \xrightarrow[\varepsilon\to 0]{} \overline{u}$. As the restriction of $g$ to $\ri \dom g$ is continuous \cite[Thm.~10.1]{rockafellar1970convex}, $g(\widetilde{u}_\varepsilon) \xrightarrow[\varepsilon\to 0]{} g(\overline{u})$. 

		Thus, for $\varepsilon > 0$ small enough,
		\begin{align*}
			g(M^{1/2} z_\varepsilon + \overline{w}_0^\varepsilon) &= g((1-t_\varepsilon) u + t_\varepsilon \widetilde{u}_\varepsilon) \\
			&\leq (1-t_\varepsilon) g(u) + t_\varepsilon g(\widetilde{u}_\varepsilon) \xrightarrow[\varepsilon\to 0]{} g(u) = g(M^{1/2} z + \overline{w}_0) \, .
		\end{align*}
		This proves the desired $\limsup$ inequality.
	\end{itemize}
\end{proof}

We now turn to the proof of Lemma~\ref{lem:mosco_recessions}. Of course, the second part $-\langle {M^\dagger}^{1/2} q,u \rangle$ of the functions $\varphi^\varepsilon$ and $\varphi$ is independent of $\varepsilon$ and thus does not impact the Mosco convergence. Hence, we only need to prove the following lemma.

\begin{lem}
	Define 
	\begin{align*}
		\Psi^\varepsilon(u) = (h^\varepsilon)^*(M^{1/2} u + \overline{w}_0^\varepsilon) \, , && \Psi(u) = \sigma_C(M^{1/2} u + \overline{w}_0) \, .
	\end{align*}
	Under the assumptions of Theorem~\ref{thm:main}, we have $\Psi^\varepsilon \xrightarrow[\varepsilon\to 0]{\mathsf{M}} \Psi$.
\end{lem}

\begin{proof}
As mentioned above, the proof strategy consists in separating the convergence of $(h^\varepsilon)^*$ to $\sigma_C$ and of $\overline{w}^\varepsilon_0$ to $\overline{w}_0$. Accordingly, define
\begin{equation*}
	\overline{\Psi}^\varepsilon(u) = \sigma_C(M^{1/2} u + \overline{w}_0^\varepsilon) \, . 
\end{equation*}
We apply Lemma~\ref{lem:mosco_recessions_rec} with $g = \sigma_C$, using assumptions~\ref{it:main-1} and~\ref{it:main-2} of Theorem~\ref{thm:main}. Note that, in this case, $0 \in \dom \sigma_C = \dom g$, thus $\aff \dom g$ is a linear subspace. As a consequence, the assumptions of the lemma are satisfied, and we obtain $\overline{\Psi}^\varepsilon \xrightarrow[\varepsilon\to 0]{\mathsf{M}} {\Psi}$.

To finish the proof of $\Psi^\varepsilon \xrightarrow[\varepsilon\to 0]{\mathsf{M}} \Psi$, we prove the equivalent statement $(\Psi^\varepsilon)^* \xrightarrow[\varepsilon\to 0]{\mathsf{M}} \Psi^*$; see Proposition~\ref{prop:mosco-fenchel}. We will use that $(\overline{\Psi}^\varepsilon)^* \xrightarrow[\varepsilon\to 0]{\mathsf{M}} \Psi^*$. We thus compute the conjugate functions $(\Psi^\varepsilon)^*$, $(\overline{\Psi}^\varepsilon)^*$, and $\Psi^*$.

We start with $(\Psi^\varepsilon)^*$. Recall the notation for the translated of a function: $(\tau_a f)(w) = f(w - a)$. Then, we can write $\Psi^\varepsilon = \left(\tau_{-\overline{w}_0^\varepsilon} \left[(h^\varepsilon)^*\right]\right) \circ M^{1/2}$. We then use Proposition~\ref{prop:conjugate_composition} to compute the conjugate of such a composition of a convex function and of a linear operator. We check the assumption. We have 
\begin{align*}
	\ri\left(\dom \tau_{-\overline{w}_0^\varepsilon}(h^\varepsilon)^*+\Span M^{1/2}\right) &= \ri\left(\dom (h^\varepsilon)^* - \overline{w}_0^\varepsilon + \Span M^{1/2}\right) \\
	&= \ri\left(\frac{1}{\mu^\varepsilon} \dom h^* - \overline{w}_0^\varepsilon + \Span M^{1/2}\right) \, .
\end{align*}
By the standing assumption in Section~\ref{sec:limiting}, $\dom h^* = \R^d$. Hence the relative interior above is all of $\R^d$ and contains $0$, so Proposition~\ref{prop:conjugate_composition} applies. We obtain
\begin{align*}
	(\Psi^\varepsilon)^*(z) = \inf_{y \in \R^d, \,  M^{1/2}y=z} \left(\tau_{-\overline{w}_0^\varepsilon} \left[(h^\varepsilon)^*\right]\right)^*(y) \, . 
\end{align*}
We then apply Proposition~\ref{prop:translation-conjugate}: 
\begin{align}
	\label{eq:conjugate-psi-eps}
	\begin{split}
	(\Psi^\varepsilon)^*(z) &= \inf_{y \in \R^d, \,  M^{1/2}y=z} \left\{h^\varepsilon(y)-\langle \overline{w}_0^\varepsilon, y \rangle\right\} = \inf_{y \in \R^d, \,  M^{1/2}y=z} \left\{\frac{1}{\mu^\varepsilon}h(y)-\langle \overline{w}_0^\varepsilon, y \rangle\right\} \, .
	\end{split}
\end{align}

We continue with the computation of the conjugate $\Psi^*$. As above, we write $\Psi = \tau_{-\overline{w}_0}\sigma_C \circ M^{1/2}$ and use Proposition~\ref{prop:conjugate_composition}. We check the assumption: by \cite[Coro.~6.6.2]{rockafellar1970convex}, 
\begin{align*}
	\ri(\dom \tau_{-\overline{w}_0} \sigma_C + \Span M^{1/2}) &= \ri(\dom \tau_{-\overline{w}_0} \sigma_C) + \ri(\Span M^{1/2}) \\
	&= \ri(\dom \sigma_C - \overline{w}_0) + \Span M^{1/2} \\
	&= \ri(\dom \sigma_C) - \overline{w}_0 + \Span M \\
	&\ni 0 \, , 
\end{align*}
by assumption~\ref{it:main-1} of Theorem~\ref{thm:main}. We can thus apply Proposition~\ref{prop:conjugate_composition}. We obtain
\begin{align*}
	\Psi^*(z) = \inf_{y \in \R^d, \,  M^{1/2}y=z} \left[\tau_{-\overline{w}_0}\sigma_C\right]^*(y) \, .
\end{align*}
We then apply Proposition~\ref{prop:translation-conjugate}:
\begin{align}
	\label{eq:conjugate-psi}
	\Psi^*(z) &= \inf_{y \in \R^d, \,  M^{1/2}y=z} \left\{\sigma_C^*(y) - \langle \overline{w}_0, y \rangle \right\}  = \inf_{y \in \R^d, \,  M^{1/2}y=z} \left\{\iota_{C}(y) - \langle \overline{w}_0, y \rangle \right\} \, . 
\end{align}
Moreover, if $z \in \dom \Psi^*$, then the infimum is attained.

Finally, we compute the conjugate $(\overline{\Psi}^\varepsilon)^*$. The computation is very similar to the one for $\Psi^*$. As above, we write $\overline{\Psi}^\varepsilon = \tau_{-\overline{w}_0^\varepsilon}\sigma_C \circ M^{1/2}$ and use Proposition~\ref{prop:conjugate_composition}. We check the assumption: by \cite[Coro.~6.6.2]{rockafellar1970convex}, 
\begin{align*}
	\ri(\dom \tau_{-\overline{w}_0^\varepsilon} \sigma_C + \Span M^{1/2}) &= \ri(\dom \tau_{-\overline{w}_0^\varepsilon} \sigma_C) + \ri(\Span M^{1/2}) \\
	&= \ri(\dom \sigma_C - \overline{w}_0^\varepsilon) + \Span M^{1/2} \\
	&= \ri(\dom \sigma_C) - \overline{w}_0^\varepsilon + \Span M \, .
\end{align*}
By assumptions~\ref{it:main-1} and~\ref{it:main-2} of Theorem~\ref{thm:main}, for $\varepsilon$ small enough, $\overline{w}_0^\varepsilon \in \ri \dom \sigma_C + \Span M$. Thus $0 \in \ri(\dom \tau_{-\overline{w}_0^\varepsilon} \sigma_C + \Span M^{1/2})$. We can apply Proposition~\ref{prop:conjugate_composition}. With the same computations as above, we obtain
\begin{align}
	\label{eq:conjugate-overline-psi-eps}
	(\overline{\Psi}^\varepsilon)^*(z) &= \inf_{y \in \R^d, \,  M^{1/2}y=z} \left\{\iota_{C}(y) - \langle \overline{w}_0^\varepsilon, y \rangle \right\} \, . 
\end{align}
Now that we have expressed the conjugates of interest, we prove that $(\Psi^\varepsilon)^* \xrightarrow[\varepsilon\to 0]{\mathsf{M}} \Psi^*$. Following Definition~\ref{def:mosco}, we show the $\liminf$ and $\limsup$ inequalities.
\begin{itemize}
	\item ($\liminf$ inequality) Let $z^\varepsilon \xrightarrow[\varepsilon\to 0]{} z$. By Eq.~\eqref{eq:conjugate-psi-eps}, 
	\begin{align}
		\label{eq:aux-1}
		\begin{split}
			(\Psi^\varepsilon)^*(z^\varepsilon) 
			&= \inf_{y \in \R^d, \,  M^{1/2}y=z^\varepsilon} \left\{\frac{1}{\mu^\varepsilon}h(y)-\langle \overline{w}_0^\varepsilon, y \rangle\right\} \\
	&= \inf_{\substack{y \in \R^d, \,  M^{1/2}y=z^\varepsilon \\ y \in C}} \left\{\frac{1}{\mu^\varepsilon}h(y)-\langle \overline{w}_0^\varepsilon, y \rangle\right\} \\
	&\geq \frac{1}{\mu^\varepsilon} \inf_{y \in \R^d} h(y) + \inf_{\substack{y \in \R^d, \,  M^{1/2}y=z^\varepsilon \\ y \in C}} \left\{-\langle \overline{w}_0^\varepsilon, y \rangle\right\} \\
	&= -\frac{1}{\mu^\varepsilon} h^*(0) + (\overline{\Psi}^\varepsilon)^*(z^\varepsilon) \, ,
		\end{split}
	\end{align}
	where in the last equality we use Eq.~\eqref{eq:conjugate-overline-psi-eps}. We have $h^*(0) < \infty$ by the standing assumption $\dom h^*=\R^d$, and by Proposition~\ref{prop:mosco-fenchel} applied to the Mosco convergence $\overline{\Psi}^\varepsilon \xrightarrow[\varepsilon\to 0]{\mathsf{M}} \Psi$,
	\begin{equation*}
		\liminf_{\varepsilon \to 0} (\overline{\Psi}^\varepsilon)^*(z^\varepsilon) \geq \Psi^*(z) \, . 
	\end{equation*}
	Thus taking the $\liminf$ in Eq.~\eqref{eq:aux-1}, we obtain 
	\begin{equation*}
		\liminf_{\varepsilon\to 0} (\Psi^\varepsilon)^*(z^\varepsilon) \geq \Psi^*(z) \, . 
	\end{equation*}
	\item ($\limsup$ inequality) Let $z \in \R^d$. If $\Psi^*(z) = +\infty$, there is nothing to prove. We thus now assume that $z \in \dom \Psi^*$. In this case, the infimum in Eq.~\eqref{eq:conjugate-psi} is attained: there exists $y_* \in \R^d$ such that $M^{1/2}y_* = z$ and $\Psi^*(z) = \iota_{C}(y_*) - \langle \overline{w}_0, y_* \rangle$. As $\Psi^*(z) < \infty$, we actually have $y_* \in C$ and $\Psi^*(z) = - \langle \overline{w}_0, y_* \rangle$. 
	
	We choose the constant recovery sequence $z^\varepsilon = z$. By Eq.~\eqref{eq:conjugate-psi-eps},
	\begin{align*}
		(\Psi^\varepsilon)^*(z^\varepsilon) &= (\Psi^\varepsilon)^*(z) = \inf_{y \in \R^d, \,  M^{1/2}y=z} \left\{\frac{1}{\mu^\varepsilon}h(y)-\langle \overline{w}_0^\varepsilon, y \rangle\right\} \leq \frac{1}{\mu^\varepsilon}h(y_*) - \langle \overline{w}_0^\varepsilon, y_* \rangle \, .
	\end{align*}
	As $y_* \in C$, $h(y_*) < \infty$, thus we obtain, taking the $\limsup$,
	\begin{equation*}
		\limsup_{\varepsilon \to 0} (\Psi^\varepsilon)^*(z^\varepsilon) \leq - \langle \overline{w}_0, y_* \rangle = \Psi^*(z) \, . 
	\end{equation*}
\end{itemize}
\end{proof} 

\section{Proof of Theorem~\ref{thm:dln-asym}}
\label{sec:proof-dln-asym}

Previous works have observed that the DLN dynamics \eqref{eq:dln-asymmetric} induce a mirror flow in the primal variable $x$~\cite{woodworth2020kernel,pesme2023saddle}. However, the mirror potential depends on $u^\varepsilon_0$ and $v^\varepsilon_0$, and thus on $\varepsilon$. Moreover, the domain of the mirror potential is the full space $\R^d$, so there is no boundary near which to initialize. Consequently, this perspective is not suitable for applying Theorem~\ref{thm:main}.

Instead, we use a reduction of the asymmetric case to the symmetric case from~\cite{berthier2025diagonal}. We recall the main lines of the derivation. The intuition is that the asymmetric dynamics induce symmetric dynamics in two variables $x^\pos$ and $x^\neg$ such that $x = x^\pos - x^\neg$. The loss of this lifted mirror flow is $\widetilde{\ell}(x^\pos, x^\neg) = \ell(x^\pos - x^\neg)$. Indeed, consider
\begin{align*}
	&x^\pos = \left(\frac{u+v}{2}\right)^2 \, , && x^\neg = \left(\frac{u-v}{2}\right)^2 \, .
\end{align*}
Then $x = u \circ v = x^\pos - x^\neg$. Moreover, 
\begin{align}
	\label{eq:dyn-xpos}
	\begin{split}
	\frac{\diff x^\pos}{\diff t} &= \left(\frac{u+v}{2}\right) \circ \left(\frac{\diff u}{\diff t} + \frac{\diff v}{\diff t}\right) \\
	&= -\left(\frac{u+v}{2}\right) \circ \left(\nabla_u\left(\ell(u\circ v)\right) + \nabla_v\left(\ell(u\circ v)\right)\right) \\
	&= -\left(\frac{u+v}{2}\right) \circ (u+v) \circ \nabla \ell(u\circ v) \\
	&= -2 x^\pos \circ \nabla \ell(x^\pos - x^\neg) \\
	&= -2 x^\pos \circ \nabla_{x^\pos} \widetilde{\ell}(x^\pos, x^\neg) \, .
	\end{split}
\end{align}
A similar computation gives 
\begin{align}
	\label{eq:dyn-xneg}
	\begin{split}
	\frac{\diff x^\neg}{\diff t} &= 2 x^\neg \circ \nabla \ell(x^\pos - x^\neg) \\
	&= -2 x^\neg \circ \nabla_{x^\neg} \widetilde{\ell}(x^\pos, x^\neg) \, .
	\end{split}
\end{align}
Compare Eqs.~\eqref{eq:dyn-xpos} and~\eqref{eq:dyn-xneg} with Eq.~\eqref{eq:dln-symmetric-x}. These equations describe the dynamics of a symmetric DLN, up to a factor $2$. Stated differently, $(x^\pos, x^\neg)$ is the primal variable of a mirror flow with loss $\widetilde{\ell}(x^\pos, x^\neg)$ and entropic mirror potential
\begin{equation*}
	\widetilde{h}(x^\pos, x^\neg) = \begin{cases}
		\frac{1}{2} \sum_{i=1}^d \left(x^\pos_i \log x^\pos_i - x^\pos_i + x^\neg_i \log x^\neg_i - x^\neg_i\right) & \text{if } x^\pos, x^\neg \geq 0 \, , \\
		+\infty & \text{otherwise} \, .
	\end{cases}
\end{equation*} 
For later use, the computation above also gives
\begin{align}
	\label{eq:cst-product}
	\begin{split}
	\frac{\diff(x^\pos\circ x^\neg)}{\diff t} &= \frac{\diff x^\pos}{\diff t} \circ x^\neg + x^\pos \circ \frac{\diff x^\neg}{\diff t} \\
	&= -2 x^\pos \circ \nabla \ell(x^\pos - x^\neg) \circ x^\neg + 2 x^\pos \circ x^\neg \circ \nabla \ell(x^\pos - x^\neg) = 0 \, .
	\end{split}
\end{align}
Thus $x^\pos \circ x^\neg$ is a constant of the dynamics.

We apply the result of Section~\ref{sec:limiting} to the mirror flow exhibited above. We now write the dependence of $u, v, x, x^\pos, x^\neg$ on $\varepsilon$. We denote by $(w^{\pos, \varepsilon}, w^{\neg, \varepsilon})$ the dual variable. We have
\begin{equation*}
	\left(w^{\pos, \varepsilon}_0, w^{\neg, \varepsilon}_0\right) = \left(\frac{1}{2} \log x_0^{\pos, \varepsilon}, \frac{1}{2} \log x_0^{\neg, \varepsilon}\right)
\end{equation*}
and by assumption
\begin{align*}
	\left\Vert \left(w^{\pos, \varepsilon}_0, w^{\neg, \varepsilon}_0\right)\right\Vert &= \left(\frac{1}{4} \sum_{i=1}^{d} \left[\log^2 x_{0,i}^{\pos, \varepsilon} + \log^2 x_{0,i}^{\neg, \varepsilon}\right]\right)^{1/2} \\
	&= \left(\sum_{i=1}^d \left[\log^2 \left\vert \frac{u^\varepsilon_{0,i} + v^\varepsilon_{0,i}}{2}\right\vert + \log^2 \left\vert \frac{u^\varepsilon_{0,i} - v^\varepsilon_{0,i}}{2} \right\vert\right]\right)^{1/2} \\
	&= \mu^\varepsilon \xrightarrow[\varepsilon \to 0]{} +\infty \, .
\end{align*}
We denote $\overline{x}^{\pos, \varepsilon}(s) = x^{\pos, \varepsilon}(s \mu^\varepsilon)$ and $\overline{x}^{\neg, \varepsilon}(s) = x^{\neg, \varepsilon}(s \mu^\varepsilon)$. By assumption, 
\begin{align*}
	\left(\overline{w}_0^{\pos, \varepsilon}, \overline{w}_0^{\neg, \varepsilon}\right) &= \frac{1}{\mu^\varepsilon}\left(w_0^{\pos,\varepsilon},w_0^{\neg,\varepsilon}\right) = \left(\frac{1}{\mu^\varepsilon} \log \left\vert \frac{u^\varepsilon_{0} + v^\varepsilon_{0}}{2}\right\vert, \frac{1}{\mu^\varepsilon} \log \left\vert \frac{u^\varepsilon_{0} - v^\varepsilon_{0}}{2} \right\vert\right) \\
	&\xrightarrow[\varepsilon \to 0]{} (\overline{w}_0^{\pos}, \overline{w}_0^{\neg}) \in \R^{2d}_{\leq 0} = \dom \partial \sigma_{\R^{2d}_{\geq 0}} \, . 
\end{align*}
Moreover, $\widetilde{\ell}$ is a quadratic function with Hessian $\widetilde{M} = \begin{pmatrix} M & -M \\ -M & M \end{pmatrix}$, thus we need to check that $\begin{pmatrix}\overline{w}_0^{\pos} \\ \overline{w}_0^{\neg}\end{pmatrix} \in \R^{2d}_{<0} + \Span \begin{pmatrix} M & -M \\ -M & M \end{pmatrix}$. By assumption, there exists $a = Mb \in \Span M$ such that $\overline{w}_0^{\pos} + a < 0$ and $\overline{w}_0^{\neg} - a < 0$. We then have
\begin{align*}
	\begin{pmatrix}\overline{w}_0^{\pos} \\ \overline{w}_0^{\neg}\end{pmatrix} &= \begin{pmatrix}\overline{w}_0^{\pos}+a \\ \overline{w}_0^{\neg}-a \end{pmatrix} + \begin{pmatrix}-a \\ a \end{pmatrix} = \begin{pmatrix}\overline{w}_0^{\pos}+a \\ \overline{w}_0^{\neg}-a\end{pmatrix} + \begin{pmatrix} M & -M \\ -M & M \end{pmatrix} \begin{pmatrix} 0 \\ b \end{pmatrix} \\
	&\in \R^{2d}_{<0} + \Span \begin{pmatrix} M & -M \\ -M & M \end{pmatrix} \, .
\end{align*}
Finally, note that $\aff \dom \sigma_{\R^{2d}_{\geq 0}} = \R^{2d}$. Thus conditions~\ref{it:main-1} and~\ref{it:main-2} are satisfied, and Theorem~\ref{thm:main} applies.

Let $(\overline{w}^\pos, \overline{w}^\neg)$ denote the unique Lipschitz solution of the mirror flow with mirror potential $\iota_{\R^{2d}_{\geq 0}}$ and loss $\widetilde{\ell}$ such that $(\overline{w}^\pos(0), \overline{w}^\neg(0)) = (\overline{w}_0^\pos, \overline{w}_0^\neg)$. Let $(\overline{x}^\pos(s), \overline{x}^\neg(s))$ be an associated primal variable. Let $S > 0$. Then:
\begin{itemize}
	\item[(b)] We have the convergences 
	\begin{align}
		\nabla \widetilde{\ell} \left(\overline{x}^{\pos,\varepsilon}(s), \overline{x}^{\neg,\varepsilon}(s)\right) &\xrightarrow[\varepsilon \to 0]{} \nabla \widetilde{\ell} \left(\overline{x}^{\pos}(s), \overline{x}^{\neg}(s)\right) \, , &&\text{in }L^2([0,S], \R^{2d}, s\diff s) \, , \label{eq:lifted-grad}\\
		\widetilde{\ell}\left(\overline{x}^{\pos,\varepsilon}(s), \overline{x}^{\neg,\varepsilon}(s)\right) &\xrightarrow[\varepsilon \to 0]{}  \widetilde{\ell} \left(\overline{x}^{\pos}(s), \overline{x}^{\neg}(s)\right) \, , &&\text{in }L^1([0,S], \R, s\diff s) \, , \label{eq:lifted-values}
	\end{align}
	where 
	\begin{equation*}
		\widetilde{\ell}(\overline{x}^\pos(s), \overline{x}^\neg(s)) = \min \left\{\widetilde{\ell}(y^\pos, y^\neg) \,\middle|\, \substack{y^\pos, y^\neg \in \R^d_{\geq 0} , \\ \langle \overline{w}^\pos(s), y^\pos \rangle = \langle \overline{w}^\neg(s), y^\neg \rangle = 0 }\right\} \, .
	\end{equation*}
	\item[(c)] Let $(\overline{y}^\pos,\overline{y}^\neg):[0,S]\to\R^{2d}$ be an accumulation point of $(\overline{x}^{\pos,\varepsilon},\overline{x}^{\neg,\varepsilon})$ in the sense of essential pointwise convergence. Then, for a.e.~$s \in [0,S]$,
	\begin{equation}
		\label{eq:lifted-acc}
		(\overline{y}^\pos(s), \overline{y}^\neg(s)) \in \Argmin \left\{\widetilde{\ell}(y^\pos, y^\neg) \,\middle|\, \substack{y^\pos, y^\neg \in \R^d_{\geq 0} , \\ \langle \overline{w}^\pos(s), y^\pos \rangle = \langle \overline{w}^\neg(s), y^\neg \rangle = 0 }\right\} \, .
	\end{equation}
\end{itemize}
We need to show that these consequences of Theorem~\ref{thm:main} imply Theorem~\ref{thm:dln-asym}. Denote $\widetilde{w} = \frac{\overline{w}^\pos - \overline{w}^\neg}{2}$ and $\widetilde{x} = \overline{x}^\pos - \overline{x}^\neg$. We show that $\widetilde{w}$ is a Lipschitz solution of the mirror flow with loss $\ell$ and mirror potential $h(x) = \sum_{i=1}^{d} \lambda_i \vert x_i \vert$, initialized at $\widetilde{w}(0) = \overline{w}_0$, with primal variable $\widetilde{x}$. Indeed, by the definition of $\overline{w}_0$, $\widetilde{w}(0) = \frac{\overline{w}_0^\pos - \overline{w}_0^\neg}{2} = \overline{w}_0$. Further,
\begin{align*}
	\frac{\diff \widetilde{w}}{\diff s}(s) &= \frac{1}{2} \left(\frac{\diff \overline{w}^\pos}{\diff s}(s) - \frac{\diff \overline{w}^\neg}{\diff s}(s)\right) \\
	&= -\frac{1}{2} \left(\nabla_{x^\pos} \widetilde{\ell}(\overline{x}^\pos(s), \overline{x}^\neg(s)) - \nabla_{x^\neg} \widetilde{\ell}(\overline{x}^\pos(s), \overline{x}^\neg(s))\right) \\
	&= -\frac{1}{2} \left(\nabla \ell(\overline{x}^\pos(s) - \overline{x}^\neg(s)) + \nabla \ell(\overline{x}^\pos(s) - \overline{x}^\neg(s))\right) \\
	&= -\nabla \ell(\widetilde{x}(s)) \, .
\end{align*}
Note that a similar computation gives $\frac{\diff}{\diff s}\left(-\frac{\overline{w}^\pos(s)+\overline{w}^\neg(s)}{2}\right) = 0$. Thus, for all $s \geq 0$, 
\begin{equation}
	\label{eq:cst-lambda}
	-\frac{\overline{w}^\pos(s)+\overline{w}^\neg(s)}{2} = -\frac{\overline{w}^\pos(0)+\overline{w}^\neg(0)}{2} = \lambda > 0 \, .
\end{equation}
We now show $\widetilde{w}(s) \in \partial h(\widetilde{x}(s))$. Recall that $(\overline{w}^\pos(s), \overline{w}^\neg(s)) \in \partial \iota_{\R^{2d}_{\geq 0}}(\overline{x}^\pos(s), \overline{x}^\neg(s))$, that is, $\overline{x}^\pos(s), \overline{x}^\neg(s) \geq 0$, $\overline{w}^\pos(s), \overline{w}^\neg(s) \leq 0$ and $\langle \overline{w}^\pos(s), \overline{x}^\pos(s) \rangle = \langle \overline{w}^\neg(s), \overline{x}^\neg(s) \rangle = 0$. Let $i \in \{1, \dots, d\}$. 
\begin{itemize}
	\item If $\overline{x}_i(s) > 0$, then we must have $\overline{x}_i^\pos(s) > 0$. Thus $\overline{w}_i^\pos(s) = 0$. Thus, by Eq.~\eqref{eq:cst-lambda}, $\widetilde{w}_i(s) = -\frac{\overline{w}_i^\neg(s)}{2} = \lambda_i$. 
	\item If $\overline{x}_i(s) < 0$, then we must have $\overline{x}_i^\neg(s) > 0$. Thus $\overline{w}_i^\neg(s) = 0$. Thus, by Eq.~\eqref{eq:cst-lambda}, $\widetilde{w}_i(s) = \frac{\overline{w}_i^\pos(s)}{2} = -\lambda_i$.
	\item In all cases,
	\begin{equation*}
		\left\vert \widetilde{w}_i(s)\right\vert = \left\vert \frac{\overline{w}_i^\pos(s)- \overline{w}_i^\neg(s)}{2}\right\vert \leq \frac{\left\vert \overline{w}_i^\pos(s)\right\vert + \left\vert \overline{w}_i^\neg(s)\right\vert}{2} = \lambda_i \, .
	\end{equation*}
\end{itemize} 
These points show that $\widetilde{w}(s) \in \partial h(\widetilde{x}(s))$, and thus complete the proof that $\widetilde{w}(s)$ is a Lipschitz solution of the mirror flow with loss $\ell$ and mirror potential $h$, such that $\widetilde{w}(0) = \overline{w}_0$. As such a solution is unique, we can conclude that 
\begin{equation*}
	\overline{w}(s) = \widetilde{w}(s) = \frac{\overline{w}^\pos(s) - \overline{w}^\neg(s)}{2} \, .
\end{equation*}
Note that, as we do not have uniqueness of the primal solution, we do not have $\overline{x}(s) = \widetilde{x}(s)$ a priori. However, as explained below Theorem~\ref{thm:existence-mirror-flow}, we must have $\nabla \ell(\overline{x}(s)) = \nabla \ell(\widetilde{x}(s))$ and $\ell(\overline{x}(s)) = \ell(\widetilde{x}(s))$ a.e. 

We are now ready to prove the statements of the theorem.
\begin{enumerate}[label=(\alph*)]
	\item By Eq.~\eqref{eq:lifted-grad}, we have the following convergence in $L^2([0,S], \R^d, s\diff s)$:
	\begin{align*}
		\nabla \ell(\overline{x}^\varepsilon(s))
		&= \nabla \ell(\overline{x}^{\pos, \varepsilon}(s) - \overline{x}^{\neg, \varepsilon}(s)) \\
		&= \nabla_{x^\pos} \widetilde{\ell}(\overline{x}^{\pos, \varepsilon}(s), \overline{x}^{\neg, \varepsilon}(s)) \\
		&\xrightarrow[\varepsilon\to 0]{} \nabla_{x^\pos} \widetilde{\ell}(\overline{x}^{\pos}(s), \overline{x}^{\neg}(s)) \\
		&= \nabla \ell(\overline{x}^\pos(s) - \overline{x}^\neg(s)) \\
		&= \nabla \ell(\widetilde{x}(s)) = \nabla \ell(\overline{x}(s)) \, .
	\end{align*}
	This proves the first statement. Further, by Theorem~\ref{thm:existence-mirror-flow},
	\begin{align*}
		\ell(\overline{x}(s)) = \min \left\{\ell(y)\,\middle|\, y\in \R^d \text{ satisfying }\begin{cases}
		\text{if }\overline{w}_i(s) = \lambda_i, & y_i \geq 0 \, , \\
		\text{if }\overline{w}_i(s) \in (-\lambda_i, \lambda_i), & y_i = 0 \, , \\
		\text{if }\overline{w}_i(s) = -\lambda_i, & y_i \leq 0 \, .
		\end{cases}\right\}  \, .  
	\end{align*}
	Finally, by Eq.~\eqref{eq:lifted-values}, we have the following convergence in $L^1([0,S], \R, s\diff s)$:
	\begin{equation*}
		\ell(\overline{x}^\varepsilon(s)) = \widetilde{\ell}(\overline{x}^{\pos, \varepsilon}(s), \overline{x}^{\neg, \varepsilon}(s)) \xrightarrow[\varepsilon\to 0]{} \widetilde{\ell}(\overline{x}^\pos(s), \overline{x}^\neg(s)) = \ell(\widetilde{x}(s)) = \ell(\overline{x}(s)) \, .
	\end{equation*}
	This finishes the proof of part~\ref{it:dln-1}.
	\item Let $\overline{y}:[0,S]\to\R^d$ be an accumulation point of $\overline{x}^\varepsilon$ in the sense of essential pointwise convergence. To use the lifted result above, we need to show that $(\overline{y}_+(s), \overline{y}_-(s))$ is an accumulation point of $(\overline{x}^{\pos, \varepsilon},\overline{x}^{\neg, \varepsilon})$ in the same sense.
	
	Recall that, from Eq.~\eqref{eq:cst-product}, $x^\pos \circ x^\neg$ is a constant of the dynamics. As a consequence,
	\begin{equation*}
		\overline{x}^{\pos, \varepsilon}(s) \circ \overline{x}^{\neg, \varepsilon}(s) = \overline{x}^{\pos, \varepsilon}_0 \circ \overline{x}^{\neg, \varepsilon}_0 = e^{2 w_0^{\pos, \varepsilon}} \circ e^{2 w_0^{\neg, \varepsilon}} = e^{2 \mu^\varepsilon (\overline{w}_0^{\pos, \varepsilon} + \overline{w}_0^{\neg, \varepsilon})} \, ,
	\end{equation*} 
	where $\mu^\varepsilon\to \infty$ and $\overline{w}_0^{\pos, \varepsilon} + \overline{w}_0^{\neg, \varepsilon} \xrightarrow[\varepsilon \to 0]{} \overline{w}_0^\pos + \overline{w}_0^\neg = -2\lambda < 0$. Thus we obtain, for all $s \geq 0$,
	\begin{equation*}
		\overline{x}^{\pos, \varepsilon}(s) \circ \overline{x}^{\neg, \varepsilon}(s) \xrightarrow[\varepsilon \to 0]{} 0 \, . 
	\end{equation*}
	Moreover, there exists $\varepsilon_k \to 0$ such that $\overline{x}^{\pos, \varepsilon_k}(s) - \overline{x}^{\neg, \varepsilon_k}(s) = \overline{x}^{\varepsilon_k}(s) \xrightarrow[k \to \infty]{} \overline{y}(s)$ for a.e.~$s \in [0,S]$. Combining these two convergences, we have that 
	\begin{equation*}
		\left(\overline{x}^{\pos, \varepsilon_k}(s), \overline{x}^{\neg, \varepsilon_k}(s)\right) \xrightarrow[k \to \infty]{} (\overline{y}_+(s), \overline{y}_-(s)) \, , 
	\end{equation*}
	for a.e.~$s \in [0,S]$. This proves that $(\overline{y}_+,\overline{y}_-)$ is an accumulation point of $(\overline{x}^{\pos, \varepsilon},\overline{x}^{\neg, \varepsilon})$ in the sense of essential pointwise convergence. By Eq.~\eqref{eq:lifted-acc}, for a.e.~$s \in [0,S]$,
	\begin{equation*}
		(\overline{y}_+(s), \overline{y}_-(s)) \in \Argmin \left\{\widetilde{\ell}(y^\pos, y^\neg) \,\middle|\, \substack{y^\pos, y^\neg \in \R^d_{\geq 0} , \\ \langle \overline{w}^\pos(s), y^\pos \rangle = \langle \overline{w}^\neg(s), y^\neg \rangle = 0 }\right\} \, .
	\end{equation*}
	Using the same reasoning as above, $\overline{y}(s) = \overline{y}_+(s) - \overline{y}_-(s)$ satisfies
	\begin{equation*}
	\begin{cases}
		\text{if }\overline{w}_i(s) = \lambda_i, & \overline{y}_i(s) \geq 0 \, , \\
		\text{if }\overline{w}_i(s) \in (-\lambda_i, \lambda_i), & \overline{y}_i(s) = 0 \, , \\
		\text{if }\overline{w}_i(s) = -\lambda_i, & \overline{y}_i(s) \leq 0 \, .
		\end{cases}
		\end{equation*}
		Moreover, if $y$ satisfies 
		\begin{equation*}
		\begin{cases}
		\text{if }\overline{w}_i(s) = \lambda_i, & y_i \geq 0 \, , \\
		\text{if }\overline{w}_i(s) \in (-\lambda_i, \lambda_i), & y_i = 0 \, , \\
		\text{if }\overline{w}_i(s) = -\lambda_i, & y_i \leq 0 \, ,
		\end{cases}
		\end{equation*}
		then $(y_+, y_-)$ satisfies
		\begin{equation*}
		\begin{cases}
			y_+ \geq 0, y_- \geq 0 \, , \\
			\langle \overline{w}^\pos(s), y_+ \rangle = \langle \overline{w}^\neg(s), y_- \rangle = 0,
		\end{cases}
		\end{equation*}
		and thus
		\begin{equation*}
			\ell(\overline{y}(s)) = \widetilde{\ell}(\overline{y}_+(s), \overline{y}_-(s)) \leq \widetilde{\ell}(y_+, y_-) = \ell(y) \, .
		\end{equation*}
		This finishes the proof of part~\ref{it:dln-2}.
\end{enumerate}

\section{Complements to Section~\ref{sec:examples}}

\subsection{Details on the implementation}
\label{app:implementation}

\paragraph*{Mirror flows.}
We first discuss simulating the mirror flows for finite $\varepsilon$. To approximate the mirror flow, we directly implement its discretization, known as the mirror descent algorithm:
\begin{align*}
	\overline{w}_{k+1} &= \overline{w}_k - \eta \nabla \ell(\overline{x}_k) \\
	\overline{x}_{k+1} &= \nabla h^*(\overline{w}_{k+1})\, .
\end{align*}
In the simulations, we use a small step size $\eta \in [10^{-4}, 10^{-3}]$. 

\paragraph*{Limiting mirror flow.} Simulating the limiting mirror flow is more involved because $h^*$, being the support function of a convex set $C$, may be non-smooth and have a restricted domain. We propose the following algorithm, that uses the minimization property of Theorem~\ref{thm:existence-mirror-flow}. This algorithm is sensible but we do not have a consistency proof. 
\begin{align}
	\overline{v}_{k+1} &= \overline{w}_k - \eta \nabla \ell(\overline{x}_k)\nonumber\\
	\label{eq:projection-limiting-mirror-flow}
	\overline{w}_{k+1} &= \Pi_{\dom \partial \sigma_C}\left[\overline{v}_{k+1}\right] \\
	\label{eq:argmin-limiting-mirror-flow}
	\overline{x}_{k+1} &\in \Argmin_{x \in \partial \sigma_C(\overline{w}_{k+1})} \ell(x)\, ,
\end{align}
Here $\Pi_{\dom \partial \sigma_C}$ denotes the Euclidean projection onto $\dom \partial \sigma_C$ in Step~\eqref{eq:projection-limiting-mirror-flow}. In all cases, Step~\eqref{eq:argmin-limiting-mirror-flow} requires a convex optimization solver to minimize $\ell$ over $\partial \sigma_C(\overline{w}_{k+1})$.

We detail the steps of the algorithm in the three examples of Section~\ref{sec:examples}. All routines we refer to are implemented in Python using the \texttt{SciPy} library:
\begin{itemize}
	\item In the case of the non-negative orthant, $\dom \partial \sigma_C = \R_{\leq 0}^d$. Hence the projection is $w \mapsto \min \{0, w\}$. As explained in the main text, the constrained minimization requires minimizing $\ell$ over vectors $x \in \R^d_{\geq 0}$ supported on the zero coordinates of $\overline{w}_{k+1}$. We implement this by calling a non-negative least-squares solver on these coordinates, namely the \texttt{scipy.optimize.nnls} routine.
	\item For optimization on positive semidefinite matrices, the projection onto $\mathcal{S}^d_- = \dom \partial \sigma_C$ sets the positive eigenvalues of $\overline{W}_{k+1}$ to zero. The set $\partial \sigma_C(\overline{W}_{k+1})$ consists of positive semidefinite matrices whose range is a subspace of the kernel of $\overline{W}_{k+1}$; equivalently, it is encoded by the linear equality $\langle -\overline{W}_{k+1}, X \rangle = 0$ on $X \succeq 0$. The minimization of $\ell$ over $\partial \sigma_C(\overline{W}_{k+1})$ is then a semidefinite program, which we model in Python via the \texttt{CVXPY} library---a domain-specific language for convex optimization that canonicalizes the problem into a standard conic form---and solve using the SCS backend (Splitting Conic Solver), a first-order ADMM-based solver for cone programs. SCS is warm-started at each iteration from the previous iterate $\overline{X}_k$ to accelerate convergence.
	\item In the case of the simplex, there is no projection as $\dom \partial \sigma_C = \R^d$. The set $\partial \sigma_C(\overline{w}_{k+1})$ is the convex hull of the vertices of the simplex corresponding to the coordinates that are maximal in $\overline{w}_{k+1}$, which we identify in practice through a small tolerance on the gap to the maximum. As $\ell$ is quadratic in this example, the minimization of $\ell$ over $\partial \sigma_C(\overline{w}_{k+1})$ reduces to a quadratic program with non-negativity and sum-to-one constraints, which we solve by calling \texttt{scipy.optimize.minimize} with the SLSQP method (Sequential Least Squares Programming).
\end{itemize}

In all figures, we identify $s = t/\mu^\varepsilon = k\eta / \mu^\varepsilon$.

\subsection{Comparison with matrix factorization}
\label{app:MMcomp}

We numerically compare the matrix mirror flow studied in Section~\ref{subsec:matrix}
with the gradient flow on the matrix factorization parametrization $X = U U^\top$,
for the same matrix loss
\[
  \ell(X) = \tfrac{1}{2}\,\langle X,\,\mathcal{M}(X) \rangle
            - \langle Q,\,X \rangle + c \, .
\]
The matrix mirror flow with the von Neumann potential satisfies
$\dot W = -\nabla \ell(X)$, $X = \exp W$, initialized at
$W^\varepsilon_0 = (\log \varepsilon)\,\I_d$.
The matrix factorization gradient flow optimizes $\ell$ directly through the
parametrization $X = U U^\top$:
\[
  \frac{\diff U}{\diff t} = - \nabla_U \ell(U U^\top)
    = -2 \big(\mathcal{M}(U U^\top) - Q\big)\, U \, ,
    \qquad U(0) = \sqrt{\varepsilon}\, \I_d \, ,
\]
so that $X(0) = \varepsilon \I_d$ in both cases.

\paragraph{Time rescaling.}
We plot trajectories against the rescaled time $s$ from
Section~\ref{subsec:matrix}, with $\mu^\varepsilon = \log(1/\varepsilon)$.
For the mirror flow we set $s = t/\mu^\varepsilon$. For the matrix
factorization flow, the chain rule on $X = U U^\top$ multiplies the rate
of every eigenvalue of $X$ by a factor $4$ relative to the mirror flow; to align the time axes we therefore plot
it against $s = 4\, t / \mu^\varepsilon$. Under this convention both flows
have activations at $\mathcal{O}(1)$ rescaled times.

\paragraph{Setup.}
We use the same measurement model as in Section~\ref{subsec:matrix}:
$\mathcal{M} = \tfrac{1}{n} \sum_{i=1}^n A_i \otimes A_i$ where the $(A_i)$
are i.i.d. GOE matrices, and $Q = \mathcal{M}(Z)$ for an independent GOE $Z$.
Here we take $d = 10$, $n = 20$, and $\varepsilon = 10^{-100}$.

\paragraph{Results.}
We run the comparison for four random seeds and report
the eigenvalue and excess-loss trajectories per seed in
Figure~\ref{fig:MMcomp_eig}. We observe that:
\begin{itemize}
  \item Both flows perform incremental learning: eigenvalues activate one at
        a time, and the dynamics pauses at intermediate low-rank saddles
        before escaping toward the next one.
  \item The two flows can visit \emph{different} intermediate saddles, both
        in their levels and in the order in which the eigenvalues activate.
        What they share is the qualitative incremental-rank structure: each
        flow rests on a sequence of low-rank plateaus before escaping
        toward higher-rank configurations, with the rank monotonically
        increasing in time.
  \item Their pacing on the rescaled time axis differs: the matrix
        factorization flow has sharp activations (it rests longer on each
        saddle, then escapes quickly), while the mirror flow has more
        smeared-out activations (longer escape transients). This confirms that the two flows
        are \emph{not} equivalent dynamics, even though they share the same
        low-rank bias.
\end{itemize}

\begin{figure}[ht]
  \centering
  \includegraphics[width=1\textwidth]{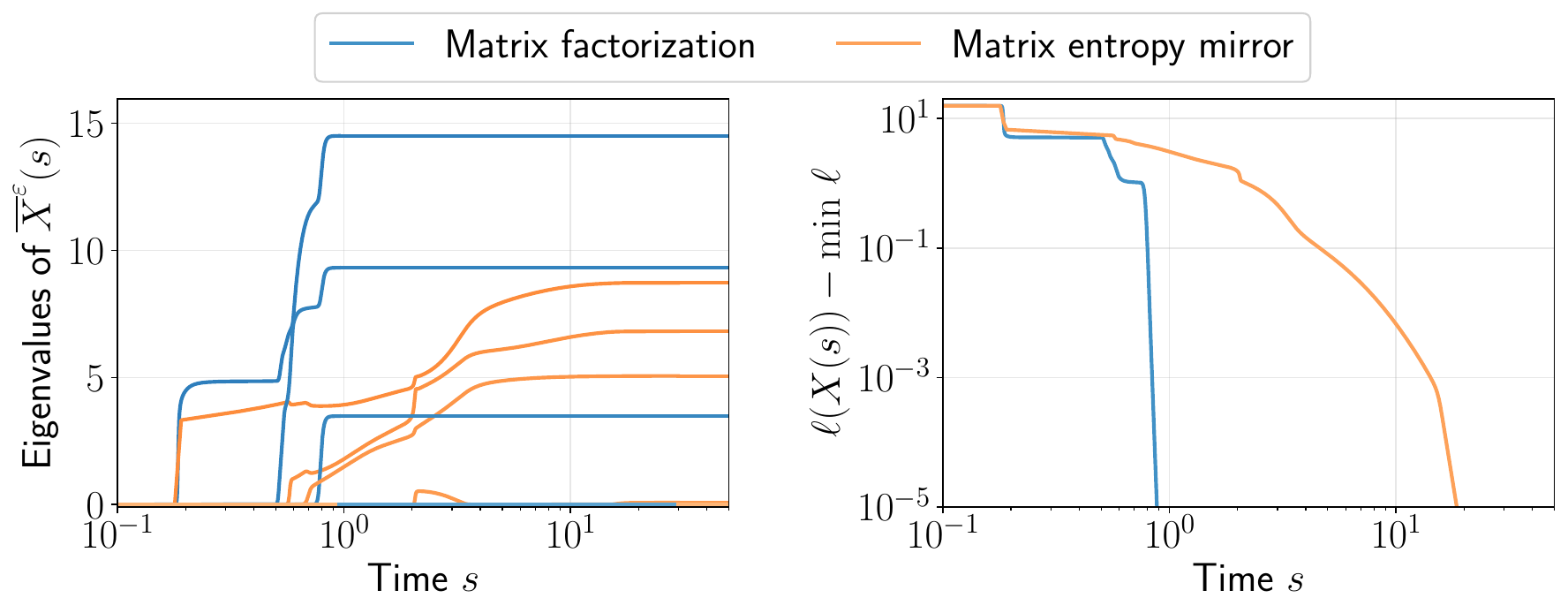}
  \vspace*{0.5cm}
  \includegraphics[width=1\textwidth]{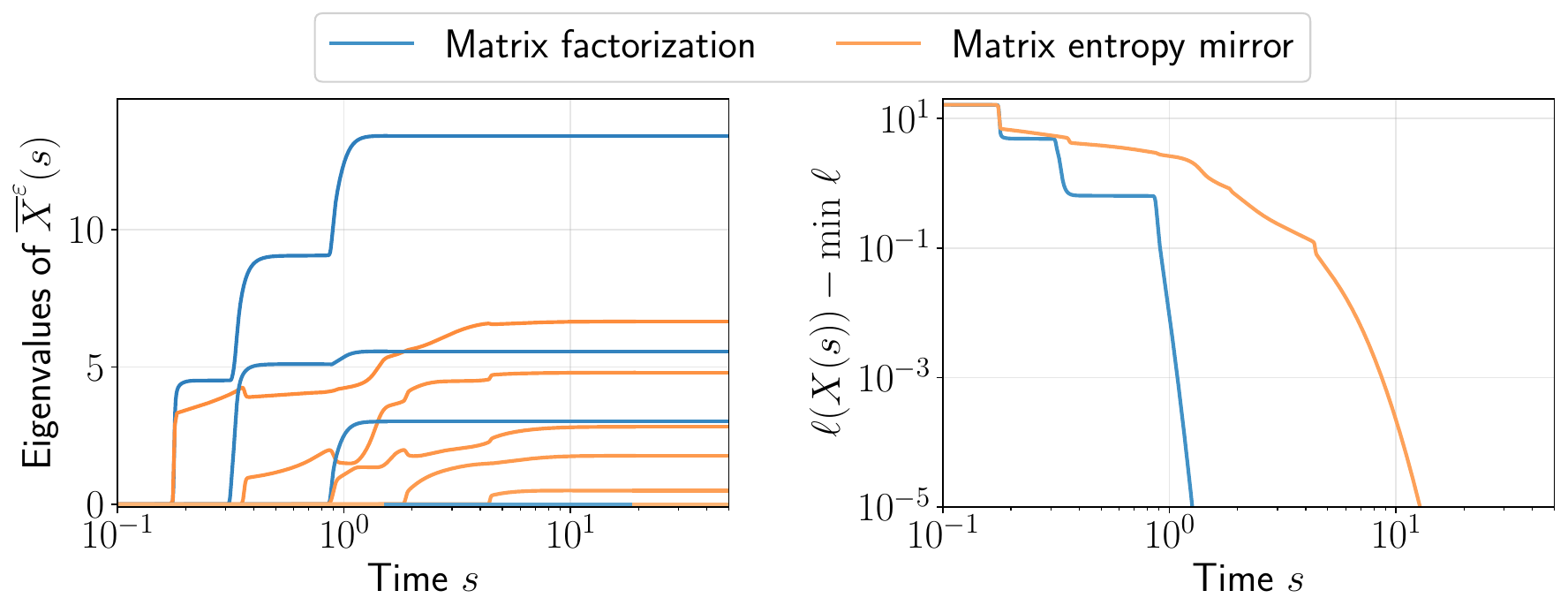}
  \includegraphics[width=1\textwidth]{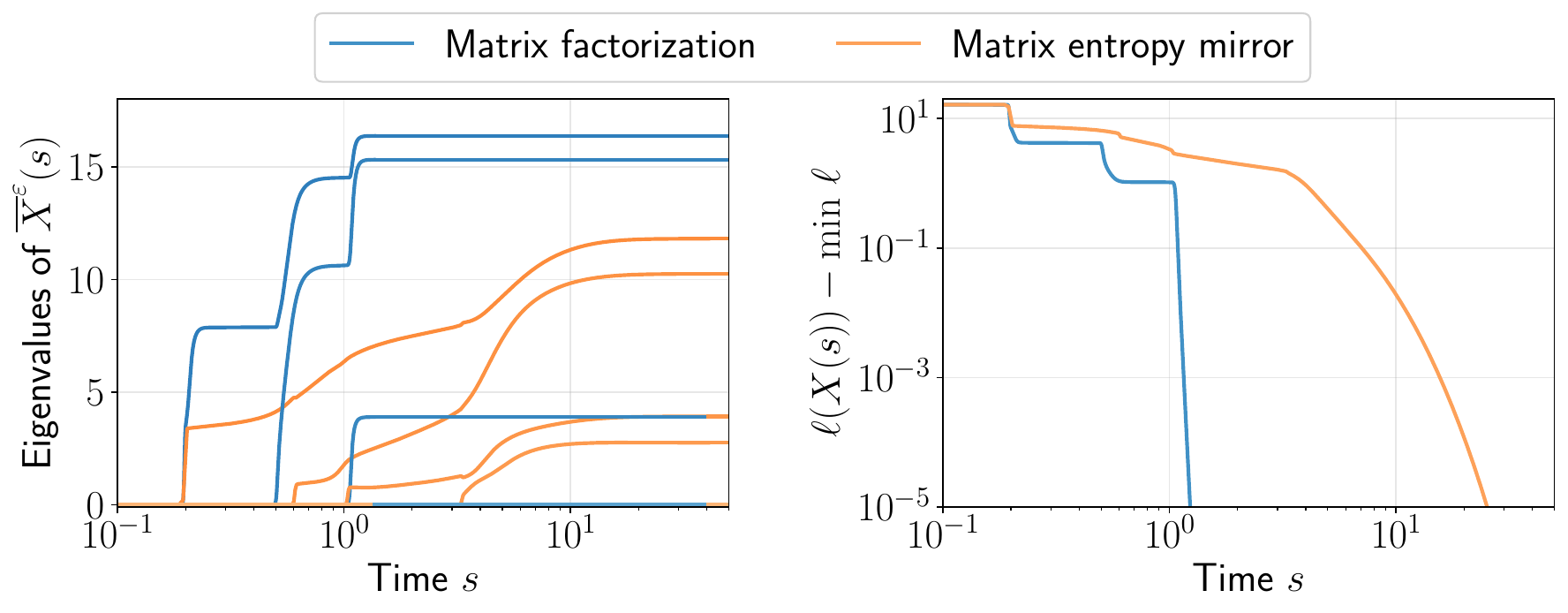}
  \includegraphics[width=1\textwidth]{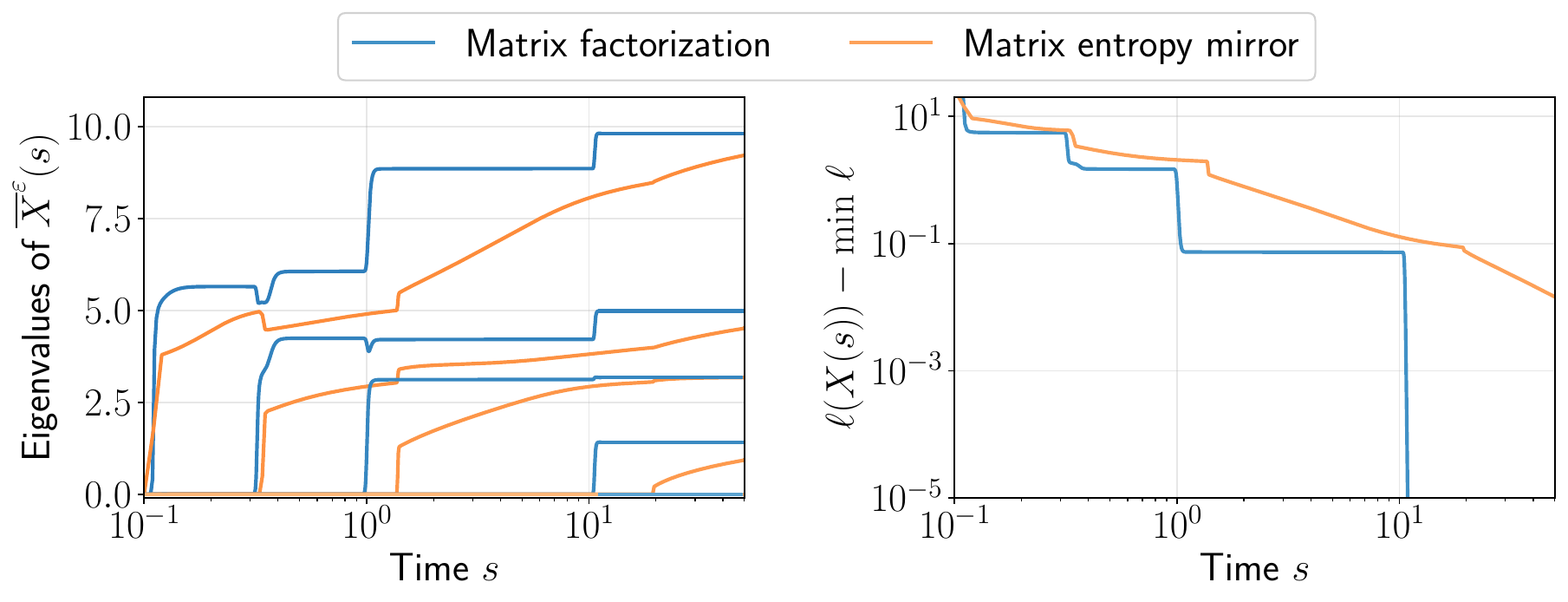}
  \caption{\textbf{(Left)} Eigenvalue trajectories of $\overline{X}^\varepsilon(s)$ for the
  matrix mirror flow (orange) and the matrix factorization gradient flow
  (blue), at $\varepsilon = 10^{-100}$, for four random seeds. \textbf{(Right)} Excess loss $\ell(\overline{X}^\varepsilon(s)) - \min \ell$ for the same flows and seeds.
  Both flows exhibit incremental
  rank-increase, but can visit different intermediate saddles in different
  orders, and pace their activations differently on the rescaled time
  axis $s$.}
  \label{fig:MMcomp_eig}
\end{figure}

\clearpage

\subsection{Additional lemmas}

\subsubsection{Conjugate of the entropic mirror potential on the simplex}

We state and prove the following lemma, which is used in the simulation of Section~\ref{subsec:simplex}. In this setting, the mirror potential is
\[
	h(x) = \sum_{i=1}^d (x_i \log x_i - x_i) + \iota_{\mathcal{S}_1}(x)
\]
if $x \in \R^d_{\geq 0}$, and $h(x) = +\infty$ otherwise, where $\mathcal{S}_1 = \{x \in \R^d \mid \sum_{i=1}^d x_i = 1\}$.
\begin{lem}
	\label{lem:mirror-potential-simplex}
	The conjugate function of $h$ is given by $\displaystyle h^*(w) = 1 + \log \left(\sum_{i=1}^d e^{w_i}\right)$, whose domain is $\R^d$. Moreover, for all $w \in \R^d$, $\displaystyle\nabla h^*(w) = e^{w}/\|e^w\|_1$, where the exponential is taken coordinatewise.
\end{lem}
\begin{proof}
	The potential $h = f + \iota_{\mathcal{S}_1}$ can be decomposed as a sum of two functions, where in this proof $f$ denotes the entropic potential with domain $\R^d_{\geq 0}$. We first note that $f^*(w) = \sum_{i=1}^d e^{w_i}$, with domain $\R^d$. Second, we compute
	\begin{align*}
	\iota_{\mathcal{S}_1}^*(w) = \sup_{x \in \mathcal{S}_1} \langle w, x\rangle = 
	\begin{cases}
		\lambda, & \text{if } w  =  \lambda \mathbf{1} \, , \\
		+\infty, & \text{otherwise} \, .
	\end{cases}
	\end{align*}
	Since $0 \in \ri(\R^d_{\geq 0} - \mathcal{S}_1)$, we can apply~\cite[Theorem 15.3]{bauschke2017convex}, so that
	\begin{align*}
	h^*(w) &= \min_{u \in \R^d} f^*(w-u) + \iota_{\mathcal{S}_1}^*(u) = \min_{\lambda \in \R} \sum_{i=1}^d e^{w_i - \lambda} + \lambda = 1 + \log \left(\sum_{i=1}^d e^{w_i}\right) \, .
	\end{align*}
	The expression of $\nabla h^*$ is then immediate.
\end{proof}

\subsubsection{Subdifferential of the support function of the simplex}

Recall that for all $w \in \R^d$, we can define its active set of coordinates as $I(w) = \{k \mid w_k = \max_i w_i\}$. We prove here the following expression, which is used in the simulation of Section~\ref{subsec:simplex}.

\begin{lem}
	\label{lem:proof-support-function-simplex}
	For all $w \in \R^d$, 
	\begin{equation*}
		\partial \sigma_{\Delta_d}(w) = \left\{ \sum_{k \in I(w)} \alpha_k e_k \,\middle|\, \sum_{k \in I(w)} \alpha_k = 1,\, \alpha_k \geq 0 \right\} ,
	\end{equation*}
	where $(e_k)_{1 \leq k \leq d}$ are the canonical vectors.
\end{lem}
\begin{proof}
	In this proof, for any vector $w \in \R^d$, we write $w_{\max} := \max_i w_i = \sigma_{\Delta_d}(w)$. Note that
	\begin{align*}
		p \in \partial \sigma_{\Delta_d}(w) &\iff \forall v \in \R^d, v_{\max} \geq w_{\max} + \langle p, v - w\rangle
	\end{align*}
	We show the lemma by double inclusion.

	First, let $p$ be of the form $\sum_{k \in I(w)} \alpha_k e_k$ where $\sum_{k \in I(w)} \alpha_k = 1$ and $\alpha_k \geq 0$. Let $v \in \R^d$. We have
	\begin{align*}
	\langle p, v - w\rangle &= \sum_{k \in I(w)} \alpha_k (v_k - w_k) = \sum_{k \in I(w)} \alpha_k v_k - w_{\max} \leq v_{\max} - w_{\max}\,,
	\end{align*}
	hence $\partial \sigma_{\Delta_d}(w) \supseteq \left\{ \sum_{k \in I(w)} \alpha_k e_k \,\middle|\, \sum_{k \in I(w)} \alpha_k = 1,\, \alpha_k \geq 0 \right\} $.

	Second, let $p \in \partial \sigma_{\Delta_d}(w)$.
	
	We first show that $p \in \Delta_d$. Let $u \in \R^d$, take $v = w + t u$ for some $t > 0$. We have that 
	\begin{align*}
	\max_{i} \{w_i + t u_i\} &\geq w_{\max} + t\langle p, u\rangle \, ,
	\end{align*}
	Dividing by $t$ and letting $t \to \infty$, as $\max_{i} \{w_i + t u_i\} / t \to \max_i u_i$, we obtain $\langle p, u\rangle \leq \max_i u_i$. Since this holds for all $u \in \R^d$, evaluating the inequality at $u = -e_k$ gives $p \in \R_{\geq 0}^d$, and evaluating it at $u = \pm \mathbf{1}$ gives $p \in \mathcal{S}_1$. Hence $p \in \Delta_d$.

	We now show that $p$ is supported on the active coordinates of $w$. Taking $v = 0$, we obtain $\langle p, w\rangle \geq w_{\max}$. Since $p \in \Delta_d$, we also have $\langle p, w\rangle \leq w_{\max}$. Hence $\langle p, w\rangle = w_{\max}$. Since $p \in \Delta_d$, this can happen only if $p$ is supported on the active coordinates of $w$. This finishes the proof of the second inclusion, and thus of the lemma.
\end{proof}

\clearpage

\subsection{Additional plots showing trajectories of the dual variables}

\subsubsection{Optimization on the non-negative orthant \texorpdfstring{$\R_{\geq 0}^d$}{}}
\label{app:diagonal_simulation}

We provide here the trajectory of the dual variable $\overline{w}$ in the simulation setting of Section~\ref{subsec:diagonal}. 

Figure~\ref{fig:diagonal_simulation_v} displays this trajectory.

\begin{figure}[h!]
	\centering
	\includegraphics[width=1\textwidth]{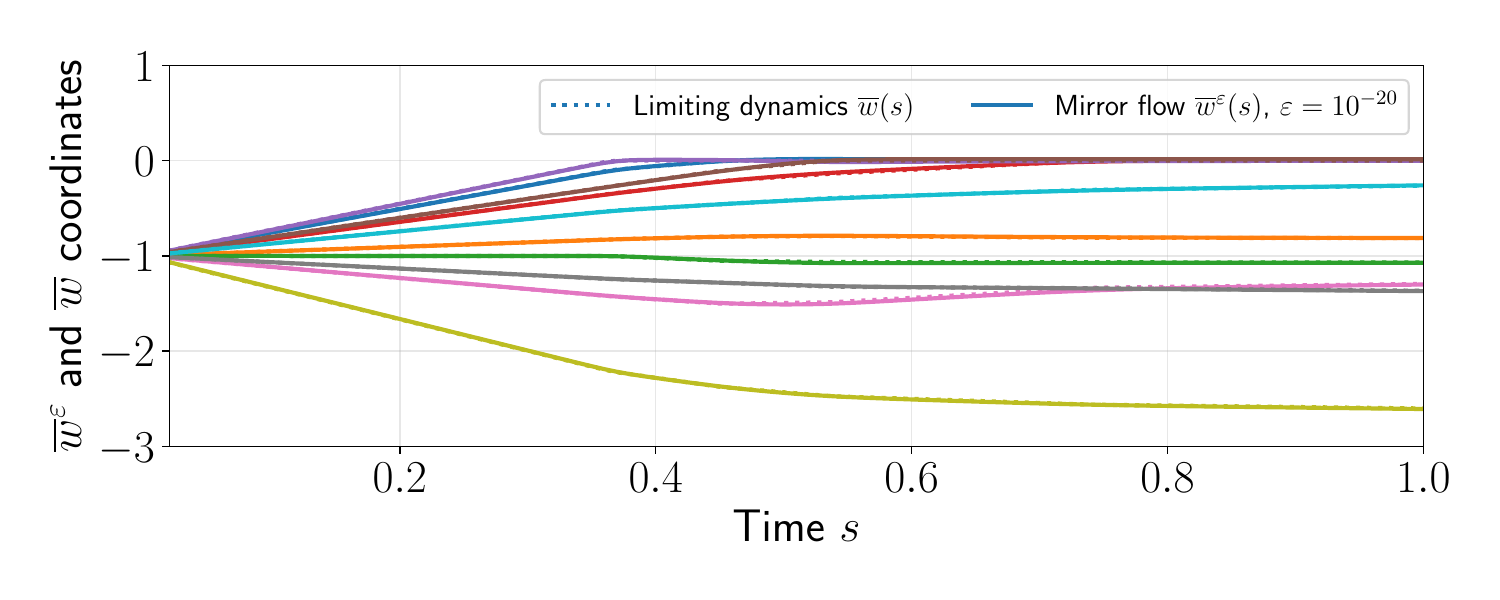}
	\includegraphics[width=1\textwidth]{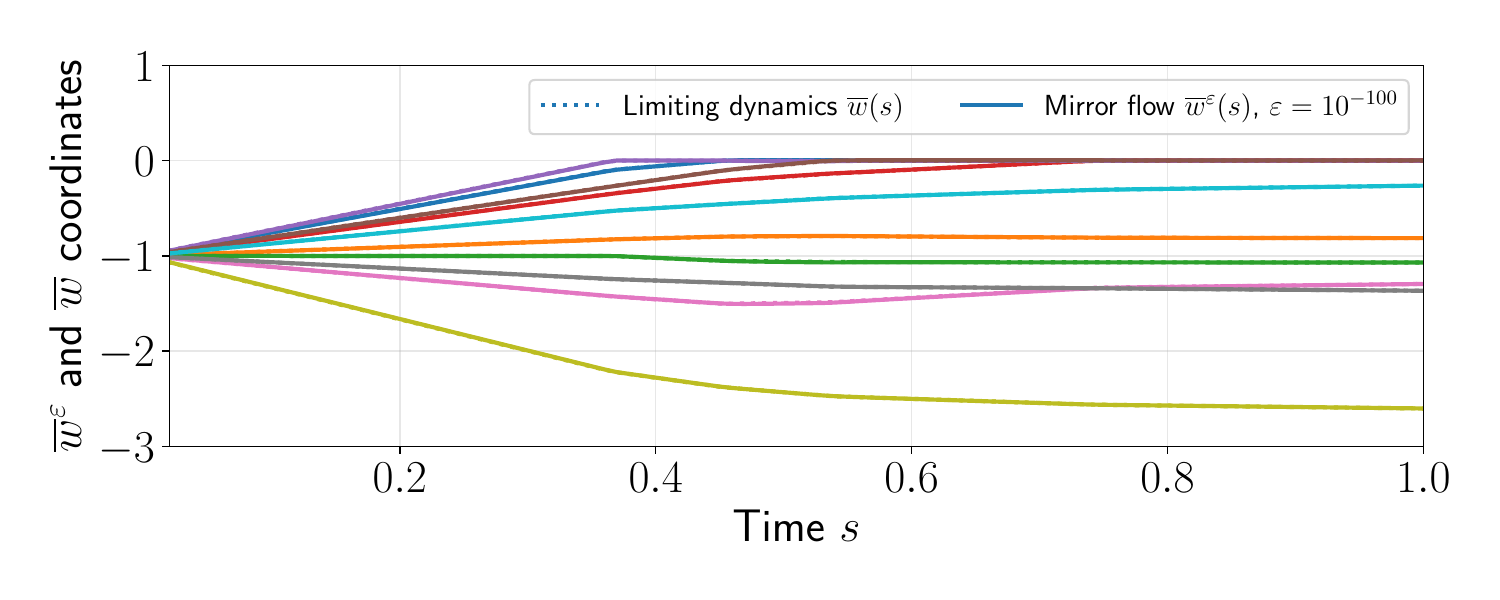}
	\caption{Trajectories of the dual variable $\overline{w}$ for the same simulation as in Figure~\ref{fig:diagonal_simulation} for $\varepsilon = 10^{-20}$ (top plot) and $\varepsilon = 10^{-100}$ (bottom plot). The trajectory of $\overline{w}$ is piecewise affine, with breakpoints corresponding to the times where a coordinate of $\overline{w}$ hits zero.}
	\label{fig:diagonal_simulation_v}
\end{figure}

\clearpage

\subsubsection{Optimization on the positive semidefinite cone \texorpdfstring{$\mathcal{S}^d_+$}{}}
\label{app:MM_simulation}

We provide here the trajectory of the dual variable $\overline{W}$ in the simulation setting of Section~\ref{subsec:matrix}. 

Figure~\ref{fig:MM_simulation_dual} displays its eigenvalues.

\begin{figure}[h!]
	\centering
	\includegraphics[width=1\textwidth]{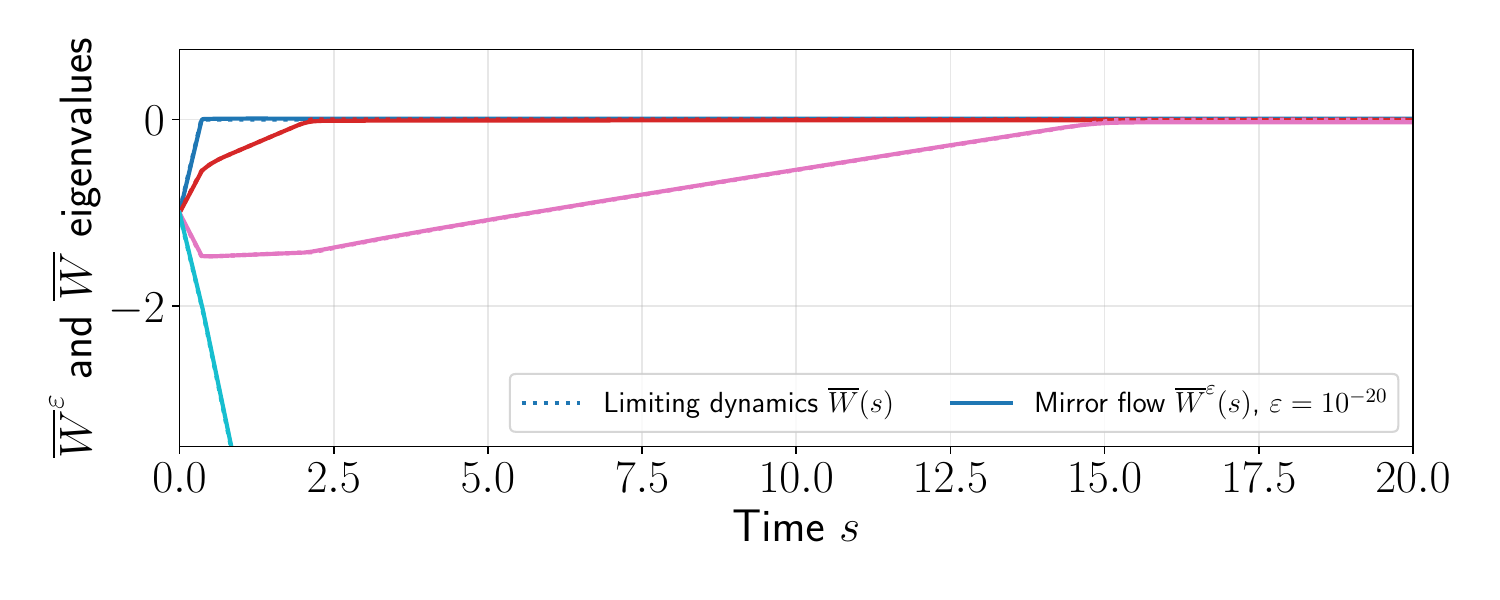}
	\includegraphics[width=1\textwidth]{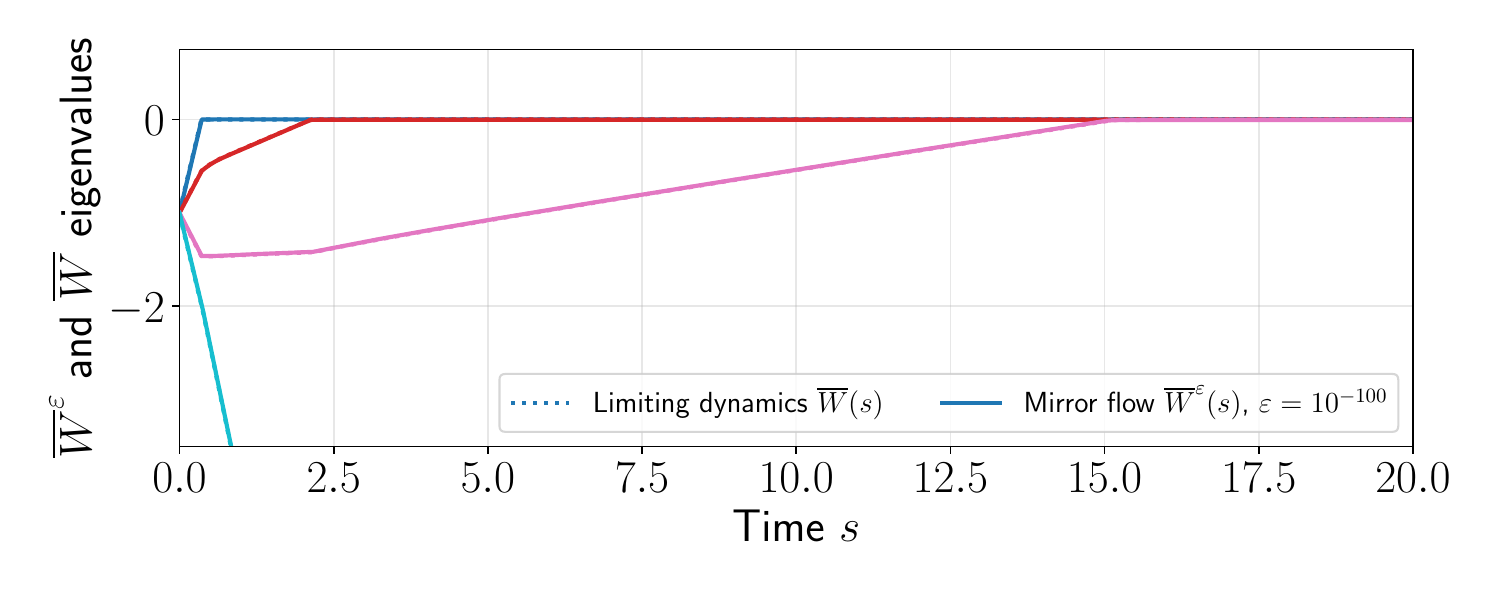}
	\caption{Eigenvalues of the dual variable $\overline{W}$ for the same simulation as in Figure~\ref{fig:MM_simulation} for $\varepsilon = 10^{-20}$ (top plot) and $\varepsilon = 10^{-100}$ (bottom plot). Eigenvalues that reach zero remain there, enlarging the active eigenspace in the primal.}
	\label{fig:MM_simulation_dual}
\end{figure}

\clearpage

\subsubsection{Optimization on the probability simplex \texorpdfstring{$\Delta_d$}{}}
\label{app:simplex_simulation}

We provide here the trajectory of the dual variable $\overline{w}$ in the simulation setting of Section~\ref{subsec:simplex}. 

Figure~\ref{fig:simplex_simulation_dual} displays its coordinates.

\begin{figure}[h!]
	\centering
	\includegraphics[width=1\textwidth]{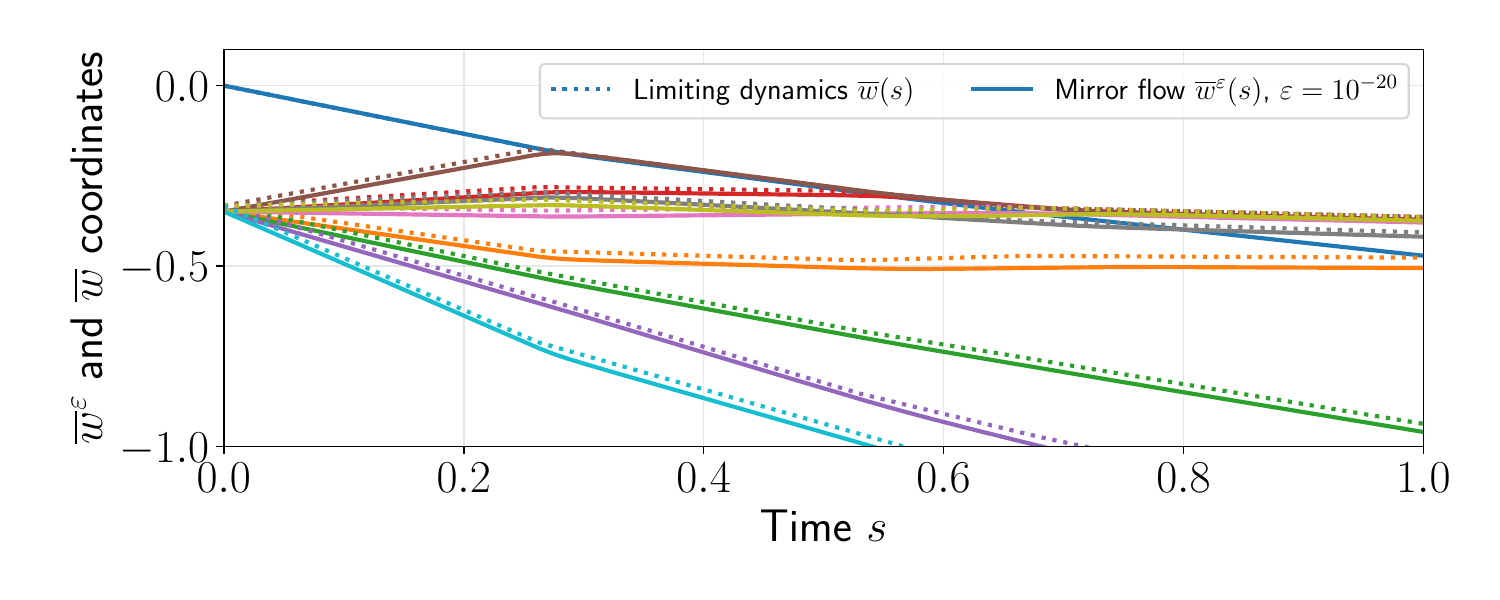}
	\includegraphics[width=1\textwidth]{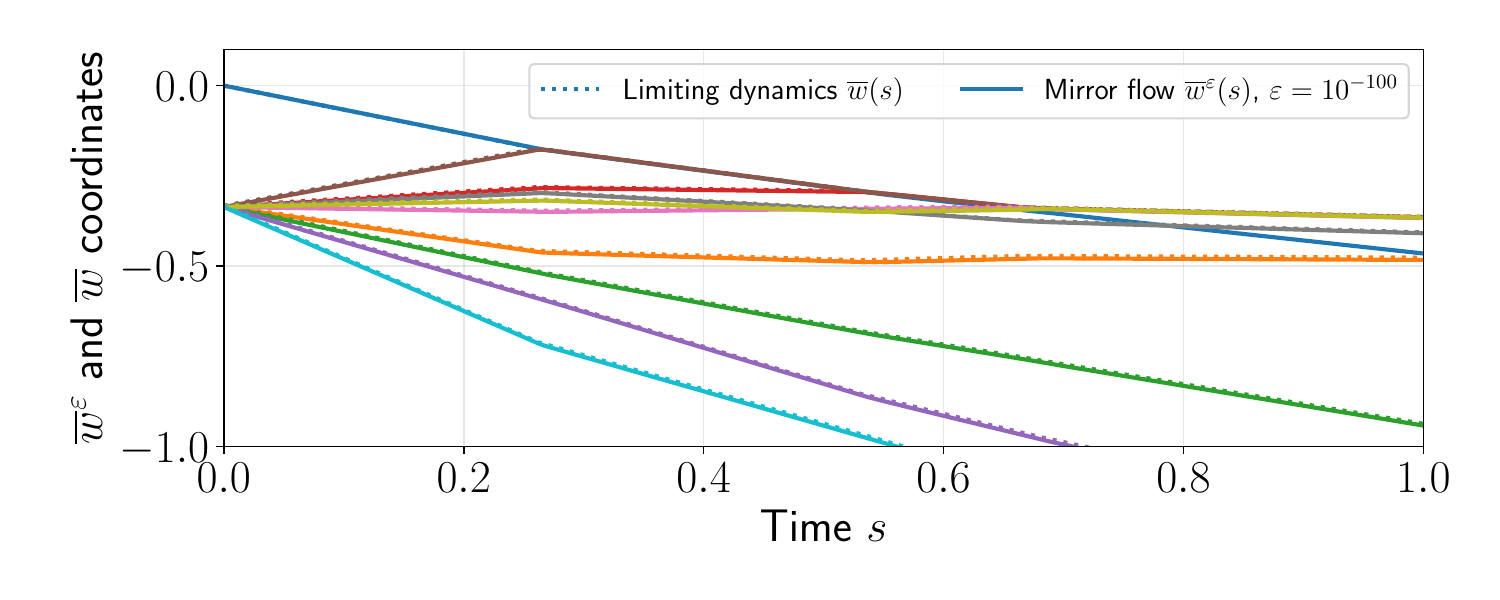}
	\caption{Trajectories of the coordinates of the dual variable $\overline{w}$ for the same simulation as in Figure~\ref{fig:simplex_simulation} for $\varepsilon = 10^{-20}$ (top plot) and $\varepsilon = 10^{-100}$ (bottom plot). The trajectory of $\overline{w}$ is piecewise affine, with breakpoints corresponding to the times a new coordinate of $\overline{w}(s)$ hits its maximum.}
	\label{fig:simplex_simulation_dual}
\end{figure}

\end{document}

%% file: Auxiliaries/Packages.tex
\usepackage{subcaption}
\usepackage{amssymb}
\usepackage{amsmath}
\usepackage{amsthm}
\usepackage{xcolor}
\usepackage[normalem]{ulem}
\usepackage{enumitem}
\usepackage[margin=3cm]{geometry}
\usepackage[hidelinks]{hyperref}
\usepackage{comment}
\usepackage{yhmath}
\usepackage{tikz}
\usetikzlibrary{arrows.meta,calc,angles,quotes}
\usepackage{algorithm}
\usepackage{algpseudocode}
\usepackage{float}
\usepackage{todonotes}
\usepackage{ulem}
\DeclareGraphicsExtensions{.pdf,.png}

%% file: Auxiliaries/Macros.tex
\DeclareMathOperator*{\argmin}{argmin}

\DeclareMathOperator{\Tr}{Tr}

\DeclareMathOperator{\Span}{Span}
\let\ker\relax
\DeclareMathOperator{\ker}{Ker}
\DeclareMathOperator{\rank}{rank}
\DeclareMathOperator{\I}{I}
\DeclareMathOperator{\dom}{dom}

\DeclareMathOperator{\Argmin}{Argmin}

\DeclareMathOperator{\supp}{supp}

\DeclareMathOperator{\ri}{ri}
\DeclareMathOperator{\aff}{aff}

\DeclareMathAlphabet{\mathbbb}{U}{bbold}{m}{n}

\newcommand{\diff}{\mathrm{d}}

\newcommand{\N}{\mathbb{N}}

\newcommand{\R}{\mathbb{R}}

\newcommand{\pos}{\text{pos}}
\renewcommand{\neg}{\text{neg}}

\renewcommand{\leq}{\leqslant}
\renewcommand{\geq}{\geqslant}

\theoremstyle{plain}
\newtheorem{proposition}{Proposition}[section]
\newtheorem{thm}[proposition]{Theorem}
\newtheorem{lem}[proposition]{Lemma}

\theoremstyle{definition}

\newtheorem{definition}[proposition]{Definition}

\theoremstyle{remark}
\newtheorem{example}[proposition]{Example}

\numberwithin{equation}{section}